\documentclass[11pt]{amsart}

\title[Pairings of combinatorial 1-cocycles]
  {
    Pairings of combinatorial 1-cocycles with loops in knot spaces
  }

\author{Butian Zhang}
\address{Institut de Mathématiques de Toulouse, Université de Toulouse, Toulouse, France}
\email{butian.zhang.math@gmail.com}

\date{21 September 2026}


\usepackage[T1]{fontenc}
\usepackage{lmodern}

\usepackage{amsmath}
\usepackage{amssymb}
\usepackage{amsthm}
\usepackage{mathtools}
\usepackage{mathrsfs}

\usepackage{graphicx}
\usepackage{subcaption}
\usepackage{xcolor}
\usepackage{enumitem}
\usepackage{tikz-cd}
\usepackage{float}
\usepackage{import}

\usepackage[margin=1.2in]{geometry}
\usepackage{hyperref}
\usepackage[nameinlink,capitalise]{cleveref}

\graphicspath{{figures/}}

\newcommand{\LongKnotSpace}{\mathcal{K}_{3,1}}

\DeclareMathOperator{\lk}{lk}
\DeclareMathOperator{\push}{push}
\DeclareMathOperator{\Rot}{Rot}

\DeclareMathOperator{\Roll}{Roll}
\DeclareMathOperator{\Repar}{Repar}

\DeclareMathOperator{\writhe}{writhe}

\DeclareMathOperator{\Diff}{Diff}
\DeclareMathOperator{\Emb}{Emb}

\newcommand{\csum}{\mathbin{\#}}

\newtheorem{theorem}{Theorem}[section]
\newtheorem{proposition}[theorem]{Proposition}
\newtheorem{lemma}[theorem]{Lemma}
\newtheorem{corollary}[theorem]{Corollary}
\newtheorem{conjecture}[theorem]{Conjecture}
\newtheorem{maintheorem}{Theorem}

\newtheorem{notation}{Notation}[section]
\newtheorem{convention}{Convention}[section]
\newtheorem{question}{Question}[section]

\newtheorem*{claim*}{Claim}

\newtheorem*{obs*}{Observation}

\theoremstyle{definition}
\newtheorem{definition}[theorem]{Definition}
\newtheorem{example}[theorem]{Example}

\theoremstyle{remark}
\newtheorem{remark}[theorem]{Remark}

\crefname{theorem}{theorem}{theorems}
\crefname{proposition}{proposition}{propositions}
\crefname{lemma}{lemma}{lemmas}
\crefname{definition}{definition}{definitions}
\crefname{conjecture}{conjecture}{conjecture}

\begin{document}

\begin{abstract}
We show that the evaluations of combinatorial 1-cocycles defined by Gauss diagrams with a triangle on the canonical loops in the space of long knots are finite type invariants. 
Using Gauss diagrams, we construct two $\mathbb{Z}$-valued combinatorial 1-cocycles $\beta_1, \beta_2$ and a $\mathbb{Z}/2\mathbb{Z}$-valued 1-cocycle $\beta_3$ on the space of long knots and prove their cocyclicity by verifying their invariance under higher Reidemeister moves coming from the codimension-two singularities of plane curves. 
We show that they represent genuinely new 1-cohomology classes and compute their pairings with the rotation, rolling, half rolling, bracket and half bracket loops. 
A key new feature is that $\beta_1$ and $\beta_2$ can pair nontrivially with bracket and half-bracket loops. 
We conjecture that $(\alpha_3^1,\beta_1,\beta_2)$ over $\mathbb{Q}$, and their mod 2 reductions together with $\beta_3$ over $\mathbb{Z}/2\mathbb{Z}$, form bases of degree-one cohomology up to order 4 in the sense of Vassiliev.
In addition, we show that the reparametrization loop is homotopic to the rolling loop concatenated with the rotation loop in $\operatorname{Emb}(S^1, S^3)$.  
Finally, we give the criteria for a 1-cohomology class in the long knot space to descend to 1-cohomology classes in $\operatorname{Emb}(S^1, S^3)$ and $\operatorname{Emb}(S^1, S^3)/\operatorname{Diff}^{+}(S^1)$. 
Using these criteria, we show that $\beta_1$ descends to a nontrivial 1-cohomology class in $\operatorname{Emb}(S^1,S^3)$, 
while $\beta_1 \bmod 2$ and $\beta_3$ descend to linearly independent nontrivial 1-cohomology classes in $\operatorname{Emb}(S^1, S^3)/\operatorname{Diff}^{+}(S^1)$. 
\end{abstract}

\maketitle

\tableofcontents

\section{Introduction}
\label{sec:Introduction}

In knot theory, one studies isotopy classes of knots, equivalently, the path components of some knot space. 
Various knot invariants, equivalently 0-cocycles in the knot space, are important tools to distinguish them. 
However, in this article, we shall care more about the paths in the knot space. 
And our tools are 1-cocycles in the knot space. 

If two knots are isotopic, there is an isotopy connecting them. 
Are any two such isotopies equivalent in some sense? 
This motivates the study of the paths in the knot space. 
Knots are commonly represented by knot diagrams. 
A sequence of Reidemeister moves represents a path in the knot space. 
It may already be difficult to determine the isotopy class represented by a single knot diagram. 
It is also hard to tell whether two paths are equivalent, i.e., path-homotopic. 
Analogously to knot invariants, 1-cocycles can play a great role in distinguishing paths. 

In this paper, we mostly work on the space of long knots $\LongKnotSpace$, which is a particularly convenient setting among the knot spaces considered here. 
The space was studied by Vassiliev via discriminant and by Hatcher and Budney from a homotopy-theoretic point of view. 

There are some canonical loops, such as the rotation loop $\Rot(\cdot)$ and the rolling loop $\Roll(\cdot)$. 
For torus knots, the rotation loop gives a generator for the fundamental group of the component. 
For hyperbolic knots, the rotation loop and the rolling loop form a basis rationally. 
For more complicated knots, there might be additional loops in the component. 
For example, given two long knots $f, g$, we consider their connected sum $f\csum g$. 
Fix a 0-framing on $g$; then shrink $f$ and push it to the right along $g$ as if $f$ were a small bead, and push this bead along the framing of $g$.
Now let $f$ stretch out to its original size. 
Next we exchange the roles of $f$ and $g$, and do the same procedure for $g$. 
This gives a loop called the bracket loop, denoted by $[f,g]_0$.
If $f= g$, we consider only the first half of the whole procedure. The resulting loop is called the half bracket loop, denoted by $\frac{1}{2}[f,f]_0$. 
This bracket loop corresponds to the Browder operation arising from the little 2-cube action on the space of framed long knots \cite{Budney07,Sakai11}. 

There are some known 1-cocycles on the space of long knots. 
The first nontrivial 1-cocycle $v_3^1$ was found by Teiblum and Turchin (see \cite{Turchin05}). 
It is the unique finite type 1-cocycle of order 3 in the sense of Vassiliev \cite{Vassiliev97}. 
Vassiliev later \cite{Vassiliev04} found an explicit description over $\mathbb{Z}/ 2\mathbb{Z}$. 
Sakai \cite{Sakai11}, using graphs and configuration space integrals, constructed an $\mathbb{R}$-valued 1-cocycle $I(\Gamma)$. 
Mortier \cite{Mortier15} found a $\mathbb{Z}$-valued combinatorial 1-cocycle $\alpha_{3}^{1}$. 
The three 1-cocycles are closely related, which can already be seen from their pairing with the loops. 
By \cite{Turchin05,Sakai11,SakiSakai23,Mortier14,Mortier15}, we have 
\[
v_3^1(\Rot(f)) = v_2(f) \mod 2, \quad v_3^1([f_1, f_2]_0) = 0 \mod 2 , v_3^1(\dfrac{1}{2}[f, f]_0) = 0 \mod 2; 
\]
\[
I(\Gamma)(\Rot(f)) = v_2(f), \quad I(\Gamma)([f_1, f_2]_0) = 0, \quad I(\Gamma)(\dfrac{1}{2}[f, f]_0) = 0, \quad I(\Gamma)(\Roll_0(f)) = 6v_3(f); 
\]
\[
\alpha_{3}^{1}(\Rot(f)) = v_2(f), \quad \alpha_{3}^{1}(\Roll_0(f)) = 6v_3(f), 
\]
and 
\[
\alpha_{3}^{1} = v_3^{1} \mod 2. 
\]
This paper focuses on the combinatorial 1-cocycles studied by Fiedler since 2003. 
Fiedler's idea is to study the generic paths in the space of long knots and to assign contributions to each Reidemeister move. 
By analysing the codimension-two singularities of plane curves, one develops higher Reidemeister moves for the invariance of generic paths under path-homotopy. 
A combinatorial 1-cocycle is constructed by assigning a weight to each Reidemeister move and summing them to obtain its value on the path. 
In the paper, we focus on the weights defined by Gauss diagram formulae with a triangle. This kind of Gauss diagram applies to R3 moves. 
For example, Mortier's 1-cocycle $\alpha_{3}^{1}$ \cite{Mortier14} is defined by a Gauss diagram formula with a triangle. See \Cref{sec:combinatorial_1cocycles,sec:Gauss_diagrams} for more details.
\begin{equation*}
\label{eq:Mortier}
  W_{\alpha_{3}^{1}} = 
    -\raisebox{-.4\height}{\includegraphics[width=3cm]{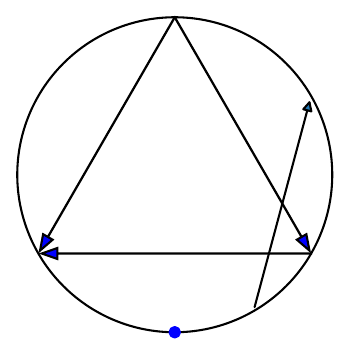}} -
    \raisebox{-.4\height}{\includegraphics[width=3cm]{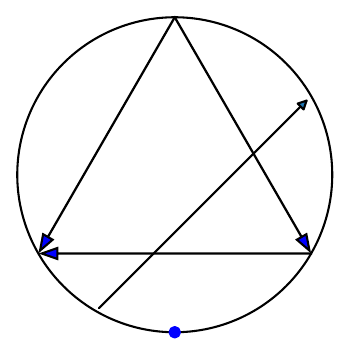}} -
    \raisebox{-.4\height}{\includegraphics[width=3cm]{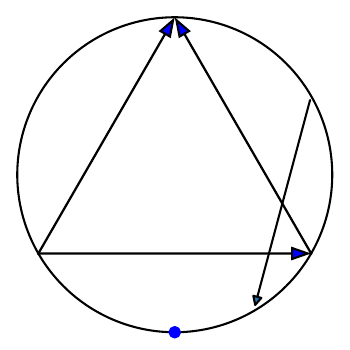}}
\end{equation*}
We use the number of arrows to define the order of a Gauss diagram. 
\begin{definition}\label{def:combinatorial_order}
When the weight function of a combinatorial 1-cocycle is given by Gauss diagrams with a triangle with at most $n$ arrows in each Gauss diagram (an R3 triangle is counted as 2 arrows), we say that the combinatorial 1-cocycle has \textbf{combinatorial order at most $n$}. 
If, moreover, the 1-cohomology class represented by this combinatorial 1-cocycle does not admit a combinatorial 1-cocycle representative having combinatorial order at most $(n-1)$, we say that the combinatorial 1-cocycle has \textbf{combinatorial order $n$}. 
\end{definition}
The combinatorial 1-cocycle $\alpha_{3}^{1}$ has combinatorial order 3. 

In this paper, we show a relation between Vassiliev invariants and the combinatorial 1-cocycles defined by Gauss diagram formulae with a triangle. 
It turns out that the pairings of these combinatorial 1-cocycles with the canonical loops yield Vassiliev invariants. 

\begin{maintheorem}\label{mainthm:pairing}
Let $\phi$ be a combinatorial $1$-cocycle defined by Gauss diagrams
with a triangle of combinatorial order at most $n$, where $n\ge 2$.
Let the coefficients be taken in a field $\mathbb K$.
Then the following hold.

\begin{enumerate}[label=(\roman*)]
    \item For any pair of long knots $(f,g)$,
    \[
        \phi([f,g]_0)
    \]
    is a finite type invariant of total Vassiliev order at most $n$.
    More precisely, if $\{v_{i,a}\}_a$ is a basis of Vassiliev
    invariants of order $i$ over $\mathbb K$, then
    \[
        \phi([f,g]_0)
        =
        \sum_{i+j\le n}\sum_{a,b}
        c_{i,a,j,b}(\phi)\,
        v_{i,a}(f)v_{j,b}(g),
    \]
    for some coefficients
    $c_{i,a,j,b}(\phi)\in\mathbb K$.

    \item For any long knot $f$,
    \[
        \phi(\Rot(f))
    \]
    is a Vassiliev invariant of order at most $(n-1)$.

    \item For any long knot $f$ and any framing number $m$,
    \[
        \phi(\Roll_m(f))
        =
        \phi(\Roll_0(f))
        -
        m\,\phi(\Rot(f)),
    \]
    where $\phi(\Roll_0(f))$ is a finite type invariant of $f$
    of order at most $n$.
\end{enumerate}
\end{maintheorem}




We give three combinatorial 1-cocycles of combinatorial order 4 representing essentially different cohomology classes and exhibiting genuinely new features. 
A key new feature is that the two integer-valued 1-cocycles $\beta_1$ and $\beta_2$ can pair nontrivially with bracket and half-bracket loops.

\begin{maintheorem}\label{mainthm:beta12}
There exist two linearly independent combinatorial $1$-cocycles
$\beta_1$ and $\beta_2$ with integer coefficients and of combinatorial
order $4$ on the space of long knots. Both represent nontrivial
$1$-cohomology classes.

For any long knots $f$ and $g$, the 1-cocycle pairings are given by



\[
\begin{array}{c|cc}
 & \beta_1 & \beta_2 \\ \hline
\Rot(f)
    & 0
    & -v_3(f)
    \\[2mm]

\Roll_0(f)
    &
    \displaystyle
    2\Bigl(
        \binom{v_2(f)}{2}-v_{4,2}(f)
    \Bigr)
    &
    \displaystyle
    2\Bigl(
        3 \cdot \binom{v_2(f)}{2}  + 2v_2(f) + 5v_3(f)-10v_{4,1}(f)-5v_{4,2}(f)
    \Bigr)
    \\[4mm]

[f,g]_0
    & -2v_2(f)v_2(g)
    & -6v_2(f)v_2(g)
    \\[2mm]

\displaystyle \frac12[f,f]_0
    & -v_2(f)^2
    & -3v_2(f)^2
\end{array}. 
\]

Here $v_2$, $v_3$, $v_{4,1}$, and $v_{4,2}$ denote primitive
Vassiliev invariants of orders $2$, $3$, $4$ and $4$, respectively (see \Cref{sec:GDF}). 
And $\binom{v_2(f)}{2} = \frac{v_2(f)(v_2(f)-1)}{2}$ is a binomial coefficient. 
\end{maintheorem}

\begin{maintheorem}\label{mainthm:beta3}
There exists a nontrivial combinatorial $1$-cocycle
$\beta_3$ of combinatorial order $4$ over $\mathbb Z/2\mathbb Z$
such that
\[
\alpha_3^1 \bmod 2,\qquad
\beta_1 \bmod 2,\qquad
\beta_2 \bmod 2,\qquad
\beta_3
\]
are linearly independent over $\mathbb Z/2\mathbb Z$.

Moreover, for any long knots $f$ and $g$,
\[
\beta_3(\Rot(f))=0,
\qquad
\beta_3(\Roll_0(f))=0,
\qquad
\beta_3([f,g]_0)=0,
\qquad
\beta_3\!\left(\frac12[f,f]_0\right)=0.
\]
\end{maintheorem}

Here “linearly independent” refers to the represented cohomology classes. 
To the best of author's knowledge, $\beta_1$ and $\beta_2$ are the first 1-cocycles that are nontrivial on the bracket loops and the half bracket loops in general. 
Although $\beta_3$ is trivial on the canonical loops, it is sometimes nontrivial on the half-rolling loops, which exist for some symmetric knots (see \Cref{fig:FH_loop} for an example). 
\begin{table}[h!]
\centering
\[
\begin{array}{c|cccc}
    & \alpha_{3}^{1} & \beta_1 & \beta_2 & \beta_3 \\ \hline
    \frac{1}{2}\Roll_0(4_1) & 0 & 1 & 1 & 0 \\ [4pt]
    \frac{1}{2}\Roll_0(6_1) & 1 & 0 & 1 & 1 
\end{array}
\]
\caption{Values of the 1-cocycles on the half rolling loops over $\mathbb{Z}/2\mathbb{Z}$. }
\end{table}

\begin{maintheorem}
    Let $f \in \Emb(S^1, S^3)$ be a closed knot. Then in $\pi_1(\Emb(S^1, S^3), f)$ we have
    \[[\Repar(f)] = [\Roll_0(f)]*[\Rot(f)] = [\Rot(f)]*[\Roll_0(f)]. \]
\end{maintheorem}

Finally, we give the criteria for the 1-cocycles on the space of long knots to descend to those on the parametrized or unparametrized spaces of oriented closed knots in $S^3$. 

\begin{maintheorem}[Descending criterion]\label{mainthm:descending}
    Let $R$ be a commutative ring with unit and $\xi \in H^{1}(\LongKnotSpace; R)$ be a 1-cohomology class. 
    Then we have 
    \begin{enumerate}
        \item $\xi$ descends to a 1-cohomology class in $H^{1}(\Emb(S^1, S^3); R)$ if and only if $2\xi([\Rot(f)]) = 0$ for all $f\in \LongKnotSpace$.
        \item $\xi$ descends to a 1-cohomology class in $H^{1}(\Emb(S^1, S^3)/\Diff^+(S^1); R)$ if and only if $2\xi([\Rot(f)]) = 0$ and $\xi([\Roll_0(f)]) + \xi([\Rot(f)]) = 0$ for all $f\in \LongKnotSpace$.
    \end{enumerate}
\end{maintheorem}
\begin{corollary}
    $\beta_1$ descends to a nontrivial 1-cohomology class in $H^{1}(\Emb(S^1, S^3); \mathbb{Z})$. \\
    $\beta_1 \bmod2$ and $\beta_3$ descend to independent nontrivial 1-cohomology classes in\\ $H^{1}(\Emb(S^1, S^3)/\Diff^+(S^1); \mathbb{Z}/2\mathbb{Z})$. 
\end{corollary}

It is interesting to compare the results with related work. 
Mortier developed a theory of combinatorial cohomology of the space of long knots \cite{Mortier14}. 
In the proof of cocyclicity of the 1-cocycles here, Gauss diagram identities that hold only for real knots rather than virtual knots are necessary. 
Mortier's theory has a virtual nature. 
However, it remains unclear whether the 1-cohomology classes represented by the 1-cocycles here can be represented within his framework. 
\begin{question}
    Can the 1-cohomology classes represented by the 1-cocycles $\beta_1, \beta_2$ and $\beta_3$ be represented by the 1-cocycles arising from Mortier's construction in \cite{Mortier14}? 
\end{question}
Mortier stated a result similar to parts (ii) and (iii) of \Cref{mainthm:pairing} in the Reidemeister-far setting and briefly indicated an approach to its proof \cite[Thm4.2]{Mortier14}. 
Our argument establishes \Cref{mainthm:pairing} without the Reidemeister-far assumption. 

Mortier also developed a Kontsevich integral for 1-cocycles \cite{Mortier22}. It would be interesting to explore the relationship between the 1-cocycles here and his Kontsevich integral for 1-cocycles. 
\begin{question}
    Do the 1-cocycles in this paper give weight systems for the Kontsevich integral in \cite{Mortier22}?
\end{question}

It would also be interesting to relate our construction to the work of Sakai \cite{Sakai11} and Kanou-Sakai \cite{SakiSakai23}. 
\begin{question}
    Can the new 1-cocycles in this paper be expressed by configuration-space integrals as the 1-cocycle $I(\Gamma)$ in \cite{Sakai11,SakiSakai23}?
\end{question}

Vassiliev defined the order of a 1-cocycle of the space of long knots in \cite{Vassiliev04}. 
Turchin showed that in the $E_1$ page of the Vassiliev spectral sequence, the order-four degree-one homology is $\mathbb{Z}^2$ over $\mathbb{Z}$ (Table 2 with $(i,j)=(4,7)$ in \cite{Turchin07}) and $(\mathbb{Z}/2\mathbb{Z})^3$ over $\mathbb{Z}/2\mathbb{Z}$ (Table 7 with $(i,j)=(4,7)$ in \cite{Turchin07}). 
Moriya proved that the Vassiliev spectral sequence collapses at the $E_1$-page over $\mathbb{Q}$ in a recent preprint \cite{Moriya25}. 
We conjecture that 
\begin{conjecture}
    The 1-cohomology classes represented by $\beta_1$ and $\beta_2$, together with $\alpha_3^1$ form a basis over $\mathbb{Q}$ of the degree-one cohomology up to order 4 in the sense of Vassiliev and 
    $\{\alpha_3^1, \beta_1, \beta_2, \beta_3\} \mod2$ form a basis over $\mathbb{Z}/2\mathbb{Z}$. 
\end{conjecture}
More generally, as in \cite[Conjecture~3.13]{Mortier14}, we have the following conjecture. 
\begin{conjecture}\label{conj:order}
  For the combinatorial 1-cocycles defined by Gauss diagrams with a triangle, 
  the combinatorial order coincides with the finite-type order in the sense of Vassiliev in \cite{Vassiliev04}. 
\end{conjecture}

\subsection{Plan of the paper}

In \Cref{sec:long_knot_space}, we introduce the space of long knots and the canonical loops. 
In \Cref{sec:combinatorial_1cocycles}, we introduce the theory of combinatorial 1-cocycles. 
In \Cref{sec:Gauss_diagrams}, we introduce Gauss diagrams and Gauss diagrams with a triangle, with which we construct combinatorial 1-cocycles. 
We give the explicit formulae for the 1-cocycles $\beta_1, \beta_2$ and $\beta_3$ in \Cref{sec:order4_1cocycles}. 
In \Cref{sec:pairing}, we first give an algorithm to realise the push arcs and the canonical loops.
Then we show that the pairings of the combinatorial 1-cocycles defined by Gauss diagram formulae with a triangle with the canonical loops give finite type invariants. 
In addition, we calculate the pairings for the 1-cocycles $\beta_1, \beta_2$ and $\beta_3$. 
In \Cref{sec:closed_knot_space}, we describe the relationship among the fundamental groups of various knot spaces. 
We prove the relation between the reparametrization loop, the rotation loop and the rolling loop. 
Finally, we give the descending criterion. 

\subsection{AI declaration}
During the preparation of this manuscript, the author used ChatGPT to revise the language and suggest improvements to the presentation. 
In response to the author's prompts, ChatGPT also drafted the proof of \Cref{prop:compactification-pi1}, gave the initial idea of the proof of \Cref{thm:repar}, proposed the formulae for $v_{4,1}$ and $v_{4,2}$ in \Cref{sec:GDF}, and drew the diagrams in that section. 
The author reviewed all AI-assisted material, and takes full responsibility for the manuscript.

\section{The space of long knots}
\label{sec:long_knot_space}

The \textbf{space of long knots} is the set 
\[
  \mathcal{K}_{3, 1} \coloneqq \{ f \in Emb(\mathbb{R}^1, \mathbb{R}^3) \Bigm| \forall |x| \geq 1, \ f(x) = (x, 0, 0) \}
\]
equipped with the Whitney $C^{\infty}$-topology (see \cite{Hirsch76FunctionSpace}).  
We use the notation $\mathcal{K}_{3,1}(f)$ to denote the component of $\mathcal{K}_{3,1}$ containing a long knot $f$.

\begin{proposition}
    Two long knots $f_0$ and $f_1$ are isotopic if and only if they are in the same path component of $\mathcal{K}_{3,1}$. 
\end{proposition}


Hatcher \cite{Hatcher83,Hatcher02} and Budney \cite{Budney05,Budney10} have studied the homotopy type of each component of the space of long knots. 
They have shown that each component is a $K(\pi, 1)$ space and  
\begin{enumerate}
  \item the component of the long unknot is contractible;
  \item the component of a nontrivial long torus knot has the homotopy type $S^{1}$,
  represented by a rotation of $2\pi$ about the axis of the long knot (the Gramain cycle, denoted by $\Rot$);
  \item the component of a long hyperbolic knot has the homotopy type $S^{1} \times S^{1}$,
  with the Gramain cycle and the 0-framed rolling loop (denoted by $\Roll_0$) spanning a cyclic covering of a torus embedded in the component of the space of long knots; 
  \item the homotopy type of the component of a long satellite knot is determined by smaller pieces using splicing. 
\end{enumerate}
\begin{remark}\label{rem:closure_operation}
    When we say a long torus knot, we mean a long knot whose closure (obtained by adding the point $\infty$) is a torus knot in $\mathbb{S}^3$. 
    The same applies to the long unknot, long hyperbolic knots, and long satellite knots. 
    This closure operation gives a one-one correspondence between the isotopy classes of long knots and those of closed knots in $\mathbb{S}^3$.
\end{remark}

\subsection{Loops in the space of long knots}
We introduce some loops in the space of long knots $\mathcal{K}_{3,1}$. 

\subsubsection{Rotation loop}\label{subsec:Gramain_cycle}
The rotation loop, which is often called the \textbf{Gramain cycle}, is given by the rotation along the axis of the long knots (i.e., the $x$-axis), defined by 
\[
\begin{aligned}
\Rot_f : \mathbb{R} \times [0,1] &\to \mathbb{R}^3, \\
(x,t) &\mapsto 
\begin{bmatrix}
  1 & 0 & 0 \\
  0 & \cos(2\pi t) & -\sin(2\pi t) \\
  0 & \sin(2\pi t) &  \cos(2\pi t)
\end{bmatrix} f(x),
\end{aligned}
\]
where $f$ is any long knot. 
Consider an isotopy of long knots $\gamma: s\mapsto f_s, \; s \in [0, 1]$. 
Then the continuous map 
\[
(t, s) \mapsto \Rot_{f_s}(\cdot, t)
\]
induces a path homotopy between the path $\Rot_{f_0}$ and the path $\gamma \ast \Rot_{f_1} \ast \gamma^{-1}$. 
\begin{remark}
    This implies that the value of a 1-cocycle in $\mathcal{K}_{3,1}$ on the Gramain loop is a knot invariant. 
\end{remark}

A diagrammatic interpretation of the Gramain cycle is given in \Cref{fig:rot}.
\begin{remark}
    Gramain \cite{Gramain77} pointed out that 
    $\pi_1(\mathcal{K}_{3,1}(f))$ contains a subgroup isomorphic to $\mathbb{Z}$ generated by the rotation loop if $f$ is nontrivial. 
\end{remark}

\subsubsection{Rolling loop}\label{subsec:Fox-Hatcher_loop} 

Given a framed long knot $(f, N)$, let $K$ be its compactified image in $S^3 = \mathbb{R}^3 \cup \{\infty\}$. 
Choose a small ball centered at $\infty \in S^3$, in which $K$ is just a straight diameter $D$ and $N$ points in a fixed direction $v_{\infty}$. 
We slide $K$ along $N$ through the small ball centered at $\infty$ one round, which, relative to $K$, looks like sliding the small ball along $K$ once. 
During the whole process, after a slight perturbation, we require that in the small ball the knot always coincide with the diameter $D$ and the framing always points in the direction of $v_{\infty}$. 
At each moment $t$, we obtain a long knot $\Roll_{(f, N)}(t) \in \LongKnotSpace$ by restricting to $\mathbb{R}^3 \subset S^3$, which depends on $(f, N)$ and $t$ continuously. 
The loop 
\[
    t \mapsto \Roll_{(f, N)}(t)
\]
is called the \textbf{rolling loop} (or \textbf{Fox-Hatcher loop} \footnote{
    The rolling loop was first introduced by Fox \cite{Fox66} to study knot spinning. 
    It was later formulated and studied by Litherland \cite{Litherland79} for knot spinning. 
    Gramain \cite{Gramain77} and Hatcher \cite{Hatcher02} studied this loop in the context of the space of knots. 
}) of the long knot $f$ with respect to the framing $N$. 
Its homotopy class depends only on the framing number of $(f, N)$. 

\begin{proposition}
    Let $f$ be a long knot with two framings $N$ and $N'$ having the same framing number. 
    Then $\Roll_{(f, N)}$ and $\Roll_{(f, N')}$ are path-homotopic. 
    We use the notation $\Roll_m(f)$ to denote the rolling loop (up to homotopy) of the $m$-framed long knot $f$. 
\end{proposition}

\begin{proposition}\label{prop:roll_framing}
    Let $f$ be a long knot. 
    Then $\Roll_{m+1}(f)$ is path-homotopic to $\Roll_{m}(f)*\Rot(f)^{-1}.$
\end{proposition}

As with the rotation loop, we have 
\begin{proposition}
    The pairing of a 1-cocycle on $\mathcal{K}_{3,1}$ with the 0-framed rolling loop $\Roll_0$ is a knot invariant. 
\end{proposition}

We have a diagrammatic description of the rolling loop \cite{Fox66, Hatcher02}. 
\begin{convention}
    Given a long knot diagram $K$, we use the notation $\Roll(K)$ to denote 
    the rolling loop with respect to the blackboard framing of $K$. 
\end{convention}

\begin{definition}\label{def:rolling move}
    Given a long knot diagram, the \textbf{rolling move} is 
    to move the last crossing of the long knot diagram to the first crossing by sweeping over or under the whole knot, or the converse. 
\end{definition} 

For a long knot diagram $K$ with $n$ crossings, the rolling loop with the blackboard framing $\Roll(K)$ can be represented by performing $2n$ rolling moves.

\begin{figure}[H] 
    \centering
    \includegraphics[width=0.7\textwidth]{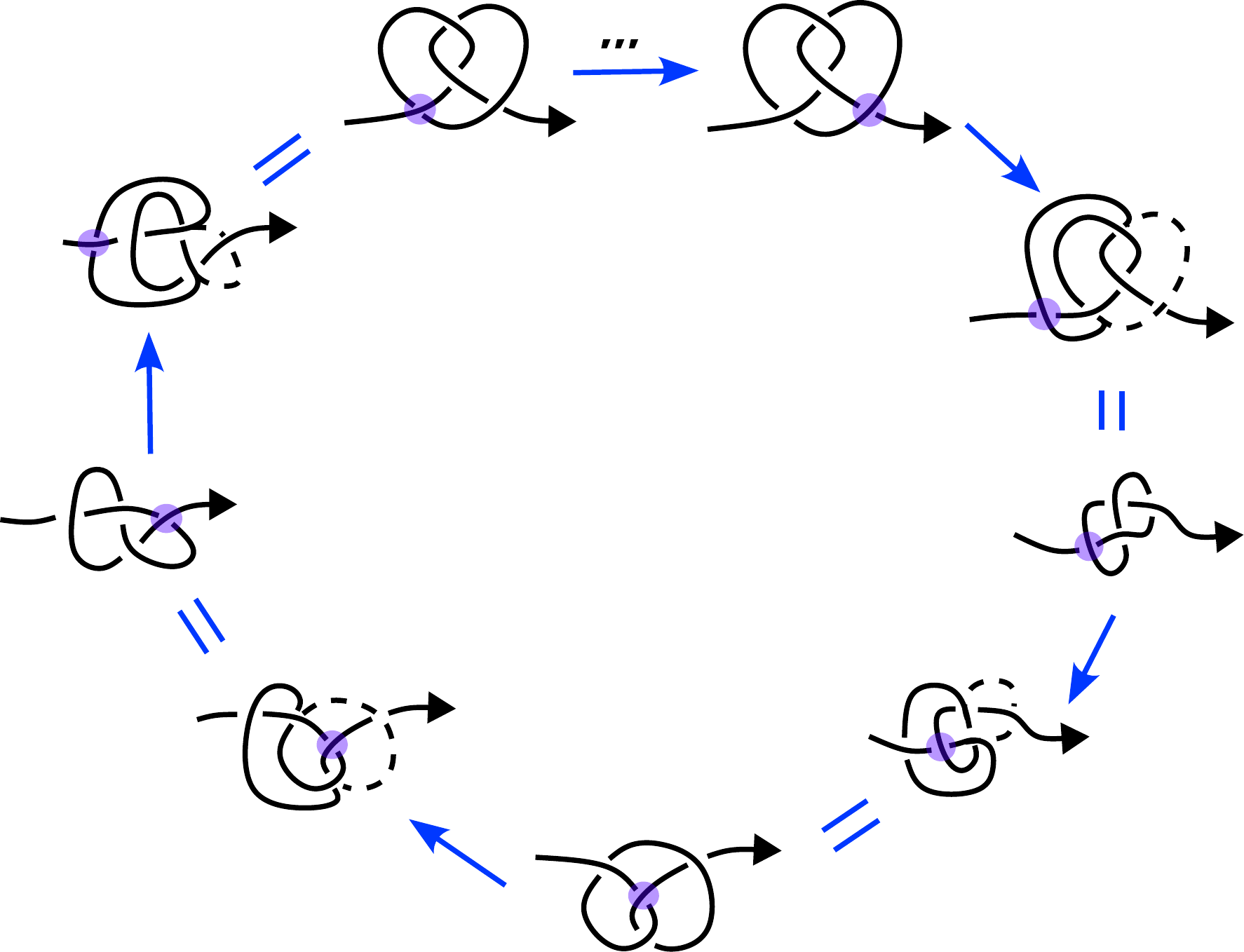}  
    \caption{The reverse of the rolling loop of $4_1$ with the blackboard framing}
    \label{fig:FH_loop}
\end{figure}

\begin{remark}
    In \Cref{fig:FH_loop}, we depict only half of the rolling loop for the knot $4_1$ (four rolling moves).
    At this stage, the traced crossing (the purple one) has not yet returned to its original position. 
    Another four rolling moves are required to complete the rolling loop. 
    However, 
    due to the symmetry of $4_1$, 
    the diagram (forgetting the traced crossing) has already returned to the initial diagram 
    after traversing only half of the whole rolling loop. 
    This shows that the rolling loop of $4_1$ is the square of a shorter loop. 
\end{remark}



\subsection{Push arc, bracket loop and half bracket loop}\label{subsec:push_arc_bracket}
Given a long knot $f$ and a framed long knot $(g, N_g)$, we consider their connected sum $f \csum g$. 
We shrink the part of $f$ in a small ball and push it along the framing $N_g$ of $g$.
We obtain a path of long knots from $f \csum g$ to $g \csum f$, which is called the \textbf{push arc} of $f$ along $(g, N_g)$, denoted by $\push(f, (g, N_g))$ (see \Cref{fig:push_arc}). 
As with the rolling loop, the path-homotopy class of the push arc depends only on the framing number of $g$. 

Suppose that $(f, N_f), (g, N_g)$ are two framed long knots. Then we have the \textbf{bracket loop} $[(f, N_f), (g, N_g)] \coloneqq \push(f, (g, N_g)) * \push(g, (f, N_f))$ (see \Cref{fig:bracket_loop}) and the \textbf{half bracket loop} $\frac{1}{2}[(f, N_f), (f, N_f)] \coloneqq \push(f, (f, N_f))$. 

\begin{notation}
Given two long knots $f, g$, we use the notation $\push_m(f, g)$ to denote 
the push arc with respect to an $m$-framing of $g$, the notation $[f, g]_0$ 
to denote the bracket loop with respect to zero framings of $f$ and $g$, 
and the notation $\frac{1}{2}[f, f]_0$ to denote the half bracket loop 
with respect to a zero-framing of $f$.
\end{notation}

\begin{convention}
Given two long knot diagrams $K, G$, we make the convention that the framings used in $\push(K, G)$, $[K, G]$ and $\frac{1}{2}[K, G]$ are all blackboard framings.
\end{convention}



\begin{figure}[H] 
    \centering
    \includegraphics[width=1.0\textwidth]{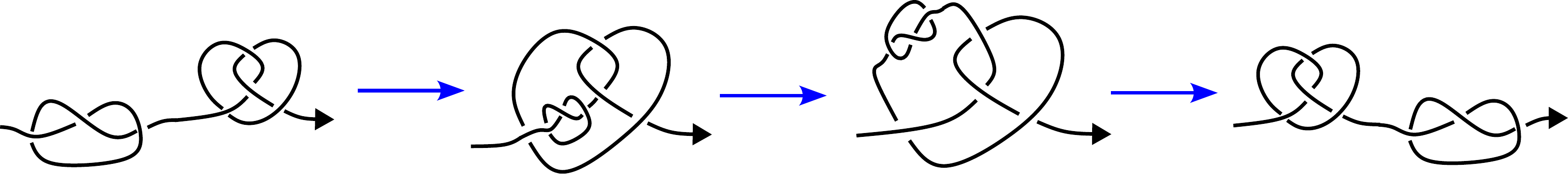}  
    \caption{The push arc $\push(3_1, 4_1)$ obtained by pushing $3_1$ through $4_1$ with the blackboard framing}
    \label{fig:push_arc}
\end{figure}

\begin{figure}[H] 
    \centering
    \includegraphics[width=0.8\textwidth]{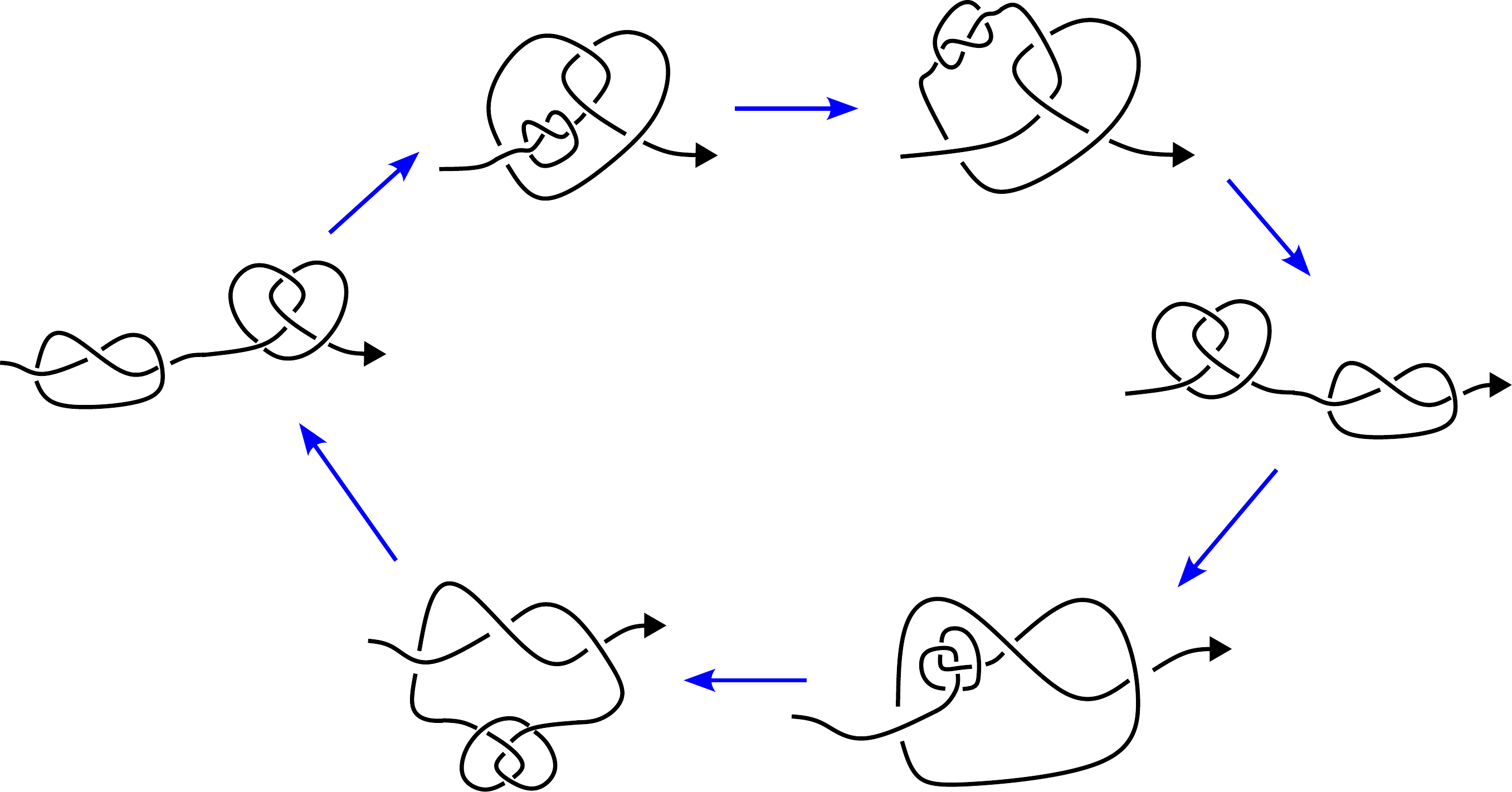}  
    \caption{The bracket loop $[3_1, 4_1]$ with the blackboard framing}
    \label{fig:bracket_loop}
\end{figure}

\begin{example}\label{eg:rot_via_push}
    Let $u$ be the trivial long knot and $f$ be any long knot.
    Then $\push_{1}(f, u) * \push_{0}(u, f)$ is path-homotopic to the rotation loop $\Rot(f)$. 
    As a corollary, we have $\push_{m}(f, g) = \push_{0}(f, g) * \Rot(f)^{m}$. 


    This also yields a diagrammatic interpretation of the rotation loop via pushing through a curl.\footnote{The author learnt this realisation from Fiedler.}   
\end{example}


\begin{figure}[H] 
    \centering
    \def\svgwidth{0.6\textwidth}
\begingroup%
  \makeatletter%
  \providecommand\color[2][]{%
    \errmessage{(Inkscape) Color is used for the text in Inkscape, but the package 'color.sty' is not loaded}%
    \renewcommand\color[2][]{}%
  }%
  \providecommand\transparent[1]{%
    \errmessage{(Inkscape) Transparency is used (non-zero) for the text in Inkscape, but the package 'transparent.sty' is not loaded}%
    \renewcommand\transparent[1]{}%
  }%
  \providecommand\rotatebox[2]{#2}%
  \newcommand*\fsize{\dimexpr\f@size pt\relax}%
  \newcommand*\lineheight[1]{\fontsize{\fsize}{#1\fsize}\selectfont}%
  \ifx\svgwidth\undefined%
    \setlength{\unitlength}{601.08790861bp}%
    \ifx\svgscale\undefined%
      \relax%
    \else%
      \setlength{\unitlength}{\unitlength * \real{\svgscale}}%
    \fi%
  \else%
    \setlength{\unitlength}{\svgwidth}%
  \fi%
  \global\let\svgwidth\undefined%
  \global\let\svgscale\undefined%
  \makeatother%
  \begin{picture}(1,0.70787313)%
    \lineheight{1}%
    \setlength\tabcolsep{0pt}%
    \put(0,0){\includegraphics[width=\unitlength,page=1]{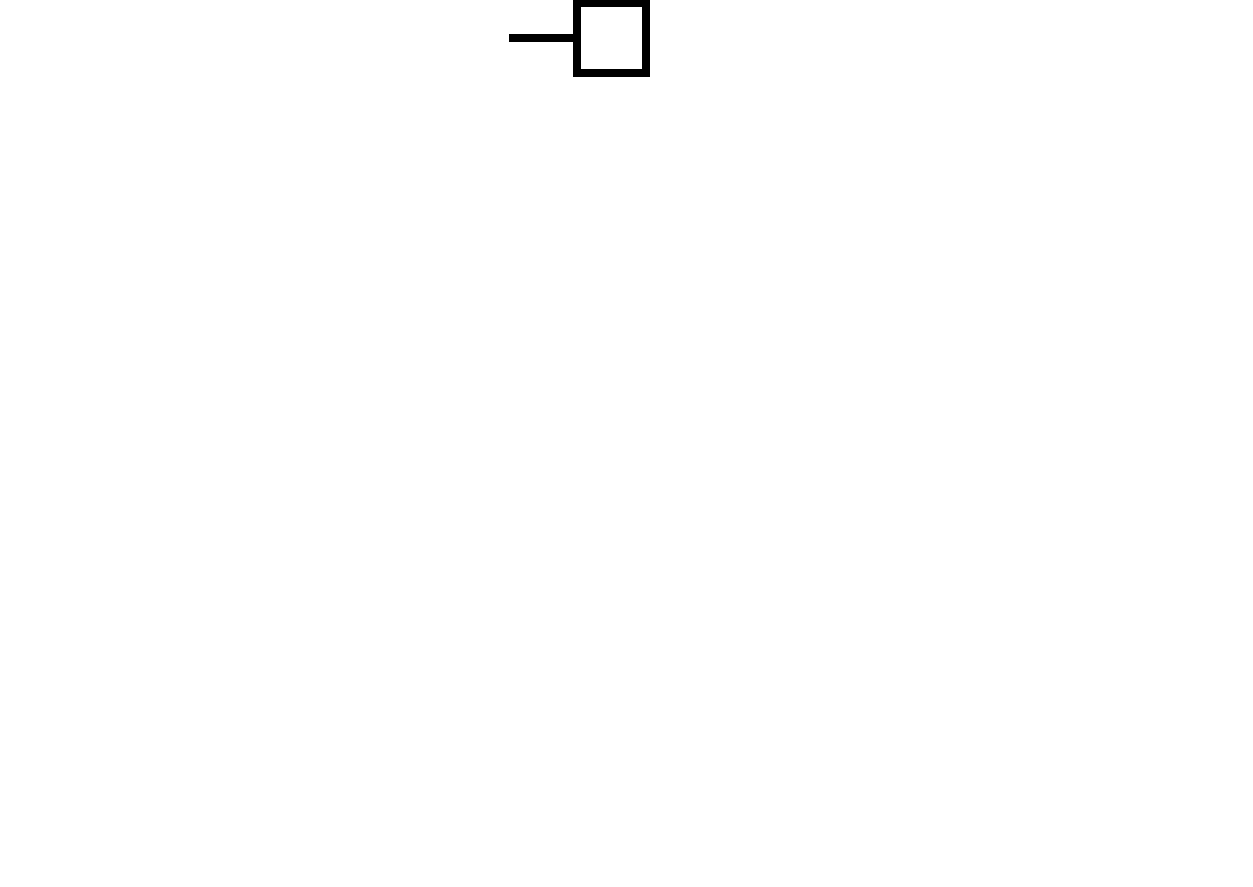}}%
    \put(0.47636724,0.66598681){\color[rgb]{0,0,0}\makebox(0,0)[lt]{\lineheight{0}\smash{\begin{tabular}[t]{l}K\end{tabular}}}}%
    \put(0,0){\includegraphics[width=\unitlength,page=2]{rot.pdf}}%
    \put(0.78321249,0.46304013){\color[rgb]{0,0,0}\makebox(0,0)[lt]{\lineheight{0}\smash{\begin{tabular}[t]{l}K\end{tabular}}}}%
    \put(0,0){\includegraphics[width=\unitlength,page=3]{rot.pdf}}%
    \put(0.69328072,0.09175454){\color[rgb]{0,0,0}\makebox(0,0)[lt]{\lineheight{0}\smash{\begin{tabular}[t]{l}K\end{tabular}}}}%
    \put(0,0){\includegraphics[width=\unitlength,page=4]{rot.pdf}}%
    \put(0.27225658,0.0588001){\color[rgb]{0,0,0}\rotatebox{89.714899}{\makebox(0,0)[lt]{\lineheight{0}\smash{\begin{tabular}[t]{l}K\end{tabular}}}}}%
    \put(0,0){\includegraphics[width=\unitlength,page=5]{rot.pdf}}%
    \put(0.17673484,0.51456109){\color[rgb]{0,0,0}\makebox(0,0)[lt]{\lineheight{0}\smash{\begin{tabular}[t]{l}K\end{tabular}}}}%
    \put(0,0){\includegraphics[width=\unitlength,page=6]{rot.pdf}}%
  \end{picture}%
\endgroup%

    \caption{Rotation loop $\Rot(K)$ via pushing through a curl}
    \label{fig:rot}
\end{figure}


Given a long knot isotopy $\gamma(t) = f_t, \; t\in[0, 1]$ and a long knot $g$, 
the commutative diagram given by
\[
\begin{tikzcd}[column sep=2.6cm, row sep=1.6cm]\label{fig:push_commutative}
f_0 \csum g
  \arrow[r, "{\push_0(f_0,g)}"]
  \arrow[d, "\gamma \# id"']
& g \csum f_0
  \arrow[r, "{\push_0(g,f_0)}"]
  \arrow[d, "id \# \gamma"]
& f_0 \csum g
  \arrow[d, "\gamma \# id"]
\\
f_1 \# g
  \arrow[r, "{\push_0(f_1,g)}"']
& g \csum f_1
  \arrow[r, "{\push_0(g, f_1)}"']
& f_0 \csum g
\end{tikzcd}
\]
indicates that the value of a 1-cocycle on the bracket loop $[f, g]_0$ is invariant of both $f$ and $g$ and that the value on $\dfrac{1}{2}[f, f]_0$ is also an invariant of $f$. 

\section{Combinatorial 1-cocycles}
\label{sec:combinatorial_1cocycles}

Combinatorial 1-cocycles were first introduced by Fiedler (see for example \cite{Fiedler07}). 
They are like discrete integrals along the generic paths. 
\begin{definition}\label{def:generic_path}
  A \textbf{generic path} in the space of long knots is 
  a path that can be described by a finite sequence of Reidemeister moves 
  starting and ending at knot diagrams with respect to the fixed projection $pr: \mathbb{R}^3 = \mathbb{R}^2 \times \mathbb{R} \rightarrow \mathbb{R}^2$. 
  We denote the set of generic paths in the space of long knots by $C_{g}(I, \mathcal{K}_{3,1})$. 
\end{definition}
\begin{definition}\label{def:combinatorial_1cocycle}
  A \textbf{combinatorial 1-cocycle} is a map 
  \[
  \phi: C_{g}(I, \mathcal{K}_{3,1}) \rightarrow G,
  \]
  where $G$ is an abelian group, such that 
  \begin{enumerate}[label=\textbf{Condition~\arabic*}, ref= Condition~\arabic*.]
      \item\label{it:1cocycle_condition1} $\phi(\gamma_0 * \gamma_1) = \phi(\gamma_0) + \phi(\gamma_1)$ for any two paths $\gamma_0, \gamma_1 \in C_{g}(I, \mathcal{K}_{3,1})$ which can be concatenated in order;
      \item\label{it:1cocycle_condition2} If $\gamma_0, \gamma_1 \in C_{g}(I, \mathcal{K}_{3,1})$ are two path-homotopic paths, then $\phi(\gamma_0) = \phi(\gamma_1)$.
  \end{enumerate}
\end{definition}

There is a general form for combinatorial 1-cocycles. 

\begin{equation}\label{eq:general_1cocycle}
\phi(\gamma) = 
\sum_{\text{Reidemeister moves in } \gamma}
\operatorname{sign}(\text{move}) \cdot \operatorname{weight}(\text{move})
\end{equation}
The sign of Reidemeister moves is to be $1$ or $-1$ 
and the weight of a Reidemeister move is determined by the diagram at the instant of the move. 
In this form, \ref{it:1cocycle_condition1} is satisfied automatically. 
\begin{remark}
    In the paper, we only use R3 moves in the formula. 
\end{remark}


A combinatorial 1-cocycle $\phi$ represents a 1-cohomology class in $H^1(\mathcal{K}_{3,1}; G)$. 
Indeed, for each path component $\mathcal{K}_{3,1}(f)$ of $\mathcal{K}_{3,1}$ we choose a base point $f$. For convenience, we require that $f$ is in general position with respect to the projection map. 
We have 
\[ 
H^{1}(\mathcal{K}_{3,1};G) \cong \prod_{\text{knot isotopy types}} H^{1}(\mathcal{K}_{3,1}(f); G)
\]
By \ref{it:1cocycle_condition1} and \ref{it:1cocycle_condition2}, we can see that $\phi$ induces a group homomorphism
\[
\Phi_{f}: \pi_1(\mathcal{K}_{3,1}(f), f) \rightarrow G
\]
for each path component $\mathcal{K}_{3,1}(f)$. 
Recall that $\operatorname{Hom}(\pi_1(\mathcal{K}_{3,1}(f), f), G) \cong H^{1}(\mathcal{K}_{3,1}(f);G)$. 
Therefore, the combinatorial 1-cocycle represents a 1-cohomology class. 

Here we have used the following theorem. 
\begin{theorem}[Propositions 4.3 and 4.4 in \cite{Roseman04}]
    \begin{itemize}
        \item For every $f \in \mathcal{K}_{3,1}$, there exists $\widetilde{f} \in \mathcal{K}_{3,1}(f)$ 
            such that $\widetilde{f}$ is in general position with respect to the projection 
            $\mathbb{R}^{3} = \mathbb{R}^{2} \times \mathbb{R} \to \mathbb{R}^{2}$.
        \item For any path $\gamma$ in the space $\mathcal{K}_{3,1}$ starting at $f_0$ and ending at $f_1$, where $f_0$ and $f_1$ are in general position,
            there exists a generic path $\widetilde{\gamma}$ from $f_0$ to $f_1$ that is path-homotopic to $\gamma$ in $\mathcal{K}_{3,1}$. 
    \end{itemize}
\end{theorem}

\subsection{Higher Reidemeister moves}\label{sec:higher_Reidemeister}
Classical Reidemeister moves correspond to the three types of \textbf{codimension 1 singularities of plane curves}: 
\begin{itemize}
    \item Ordinary triple point: the transversal intersection of 3 arcs;
    \item Simple tangency: the intersection locally given by $y = \pm x^2$ at $(0, 0)$;
    \item Ordinary cusp: the singular point of an arc locally given by $y^2 = x^3$ at $(0, 0)$. 
\end{itemize}

To study the behaviour of generic paths under path-homotopies, we need to consider the codimension 2 singularities. 
They are classified by David (Theorem 5.1 in \cite{David83}). 
Kurlin and Fiedler described in detail the bifurcation diagrams of these singularities, 
i.e., the meridians around the singularities, in \cite{FiedlerKurlin10}. 

There are 6 types of \textbf{singularities of codimension 2}:  
\begin{enumerate}
    \item Two singularities of codimension 1. 
    \item Ordinary quadruple point: the transversal intersection of 4 arcs. 
    \item Tangent triple point: the intersection of 3 arcs such that the first two arcs have a simple tangency and that the direction of the tangency is not parallel to the third arc. 
    \item Intersected cusp: the intersection of 2 arcs, where the first arc has an ordinary cusp whose derivative of order 2 does not touch the second arc. 
    \item Cubic tangency: the intersection of 2 arcs locally given by $x = 0$ and $y = x^3$. 
    \item Ramphoidal cusp: the singular point of an arc given locally by $x^2 = y^5$.
\end{enumerate}

\begin{figure}[H]
    \centering
    \begin{subfigure}[t]{0.25\textwidth}
      \centering
      \includegraphics[height=3cm]{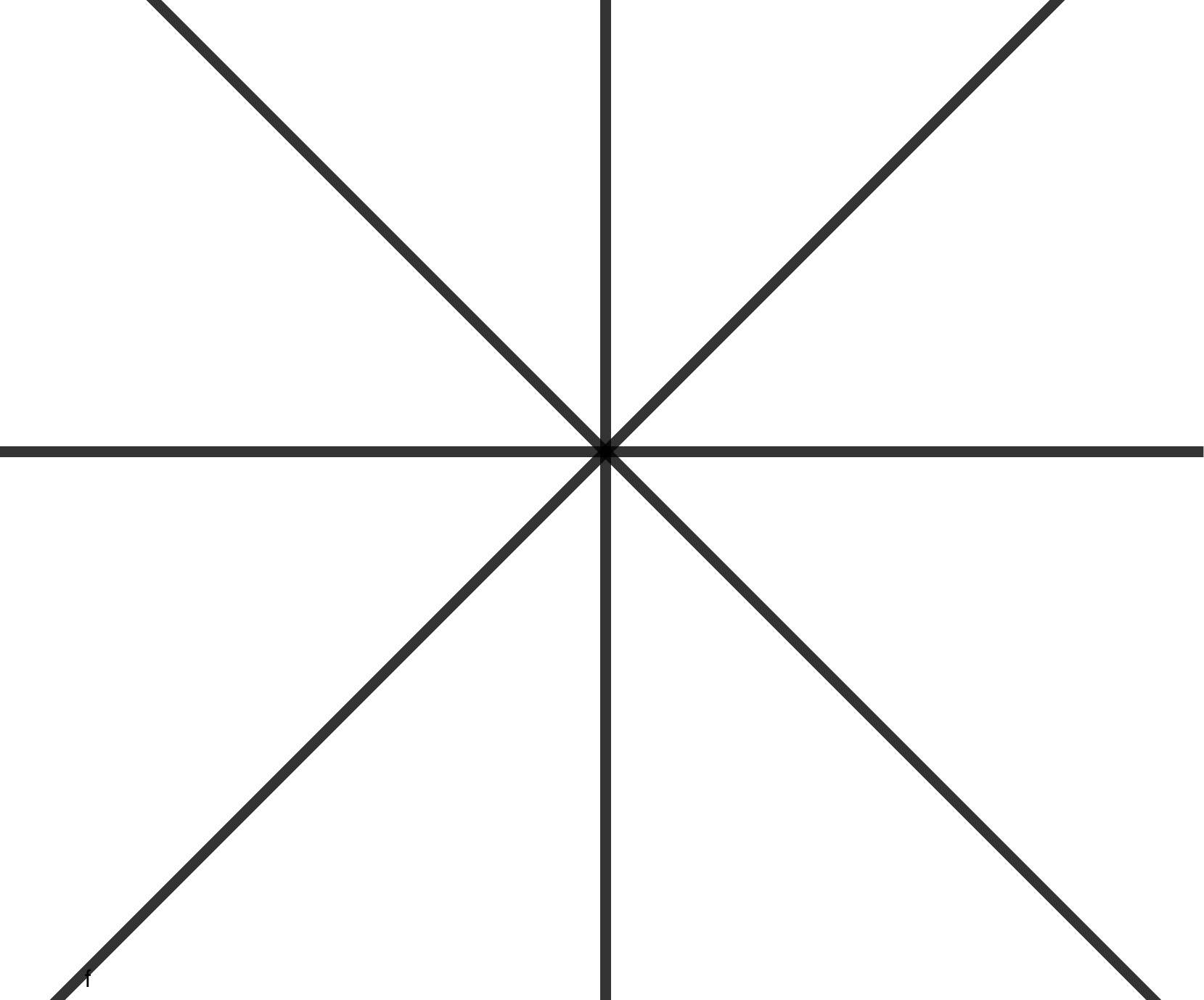}
      \caption{Ordinary quadruple point}
    \end{subfigure}
    \hspace{0.1\textwidth}
    \begin{subfigure}[t]{0.25\textwidth}
      \centering
      \includegraphics[height=3cm]{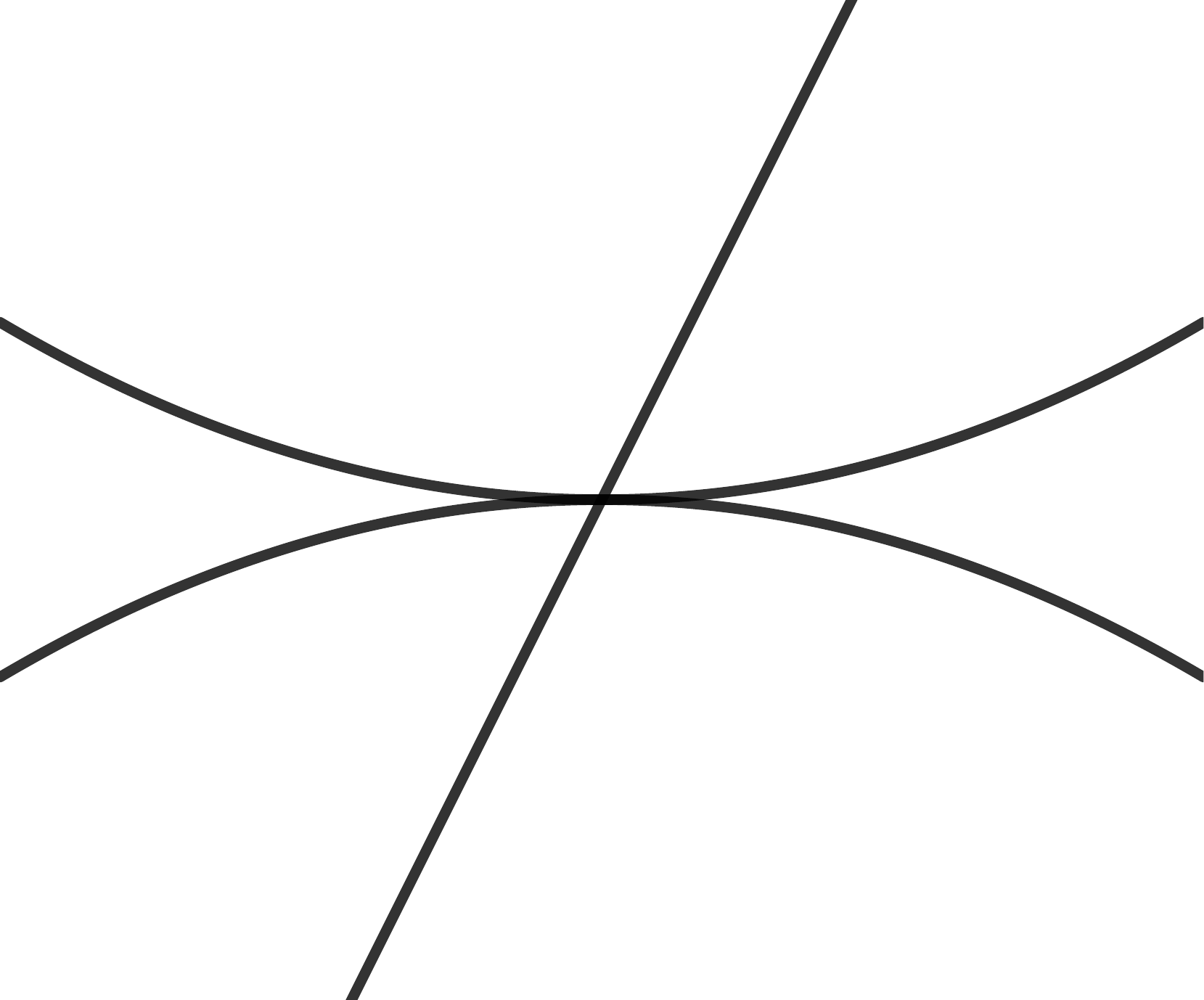}
      \caption{Tangent triple point}
    \end{subfigure}
    \hspace{0.1\textwidth}
    \begin{subfigure}[t]{0.25\textwidth}
      \centering
      \includegraphics[height=3cm]{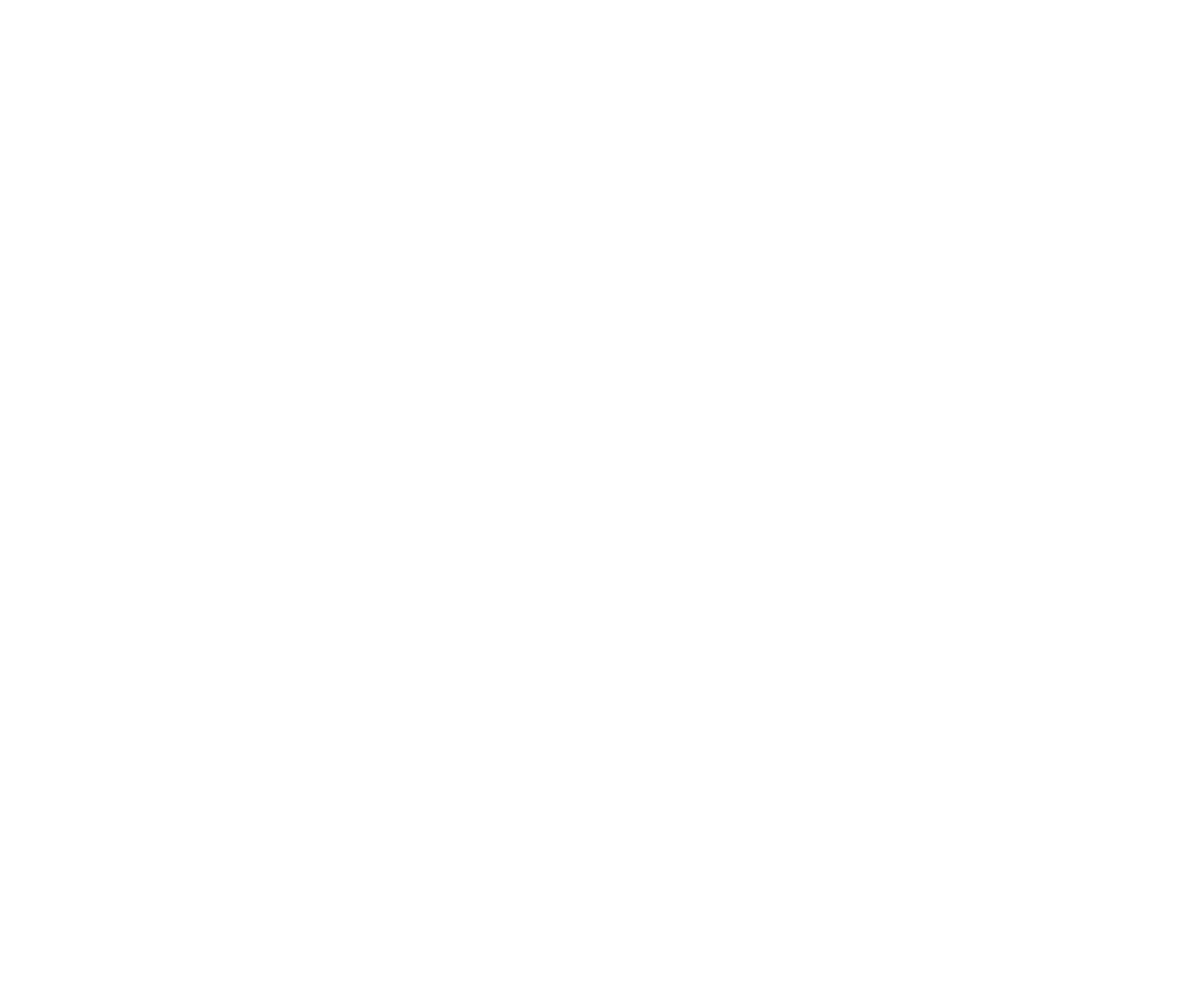}
      \caption{Intersected cusp}
    \end{subfigure}
    \hspace{0.1\textwidth}
    \begin{subfigure}[t]{0.25\textwidth}
      \centering
      \includegraphics[height=3cm]{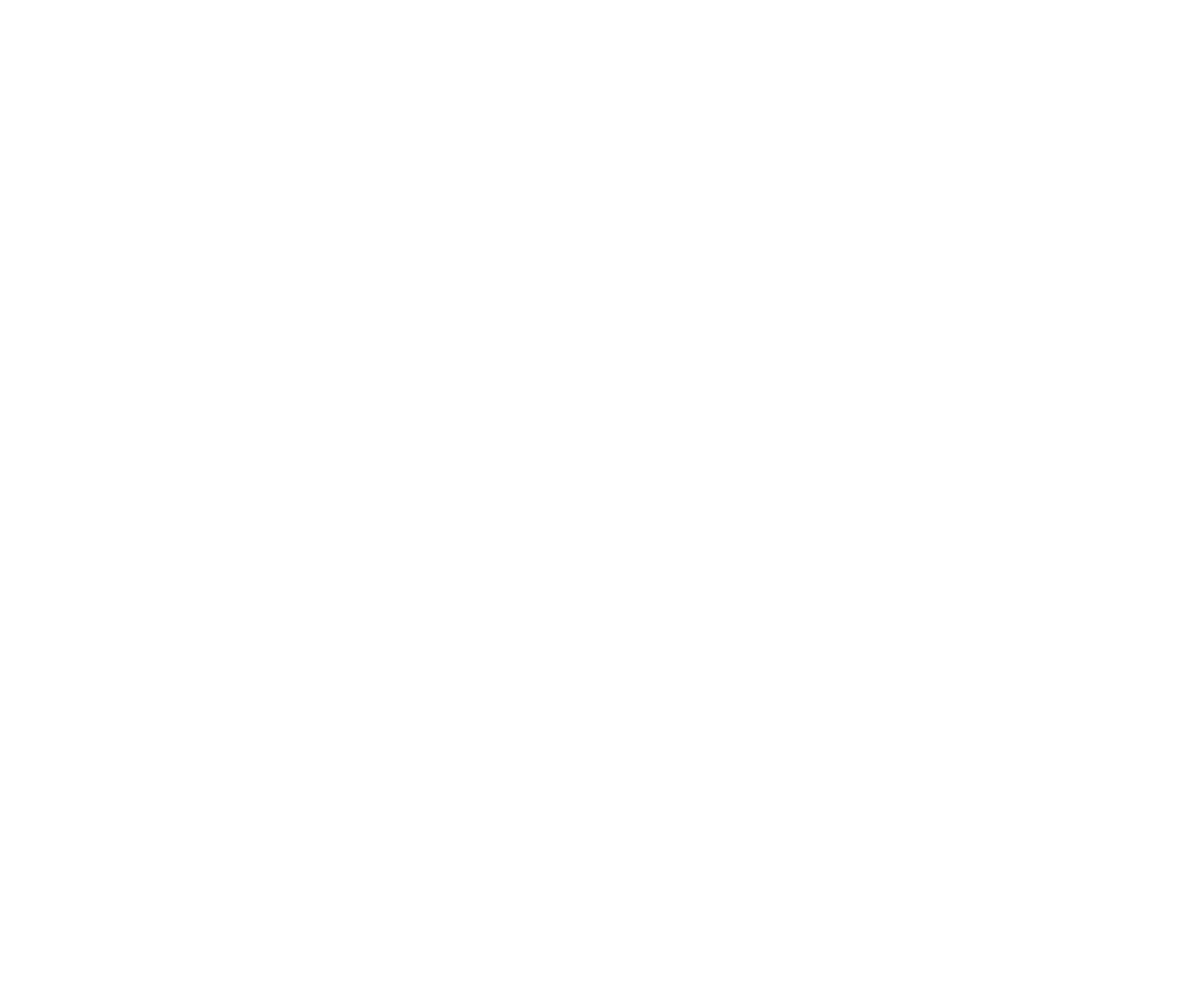}
      \caption{Cubic tangency}
    \end{subfigure}
    \hspace{0.1\textwidth}
    \begin{subfigure}[t]{0.25\textwidth}
      \centering
      \includegraphics[height=3cm]{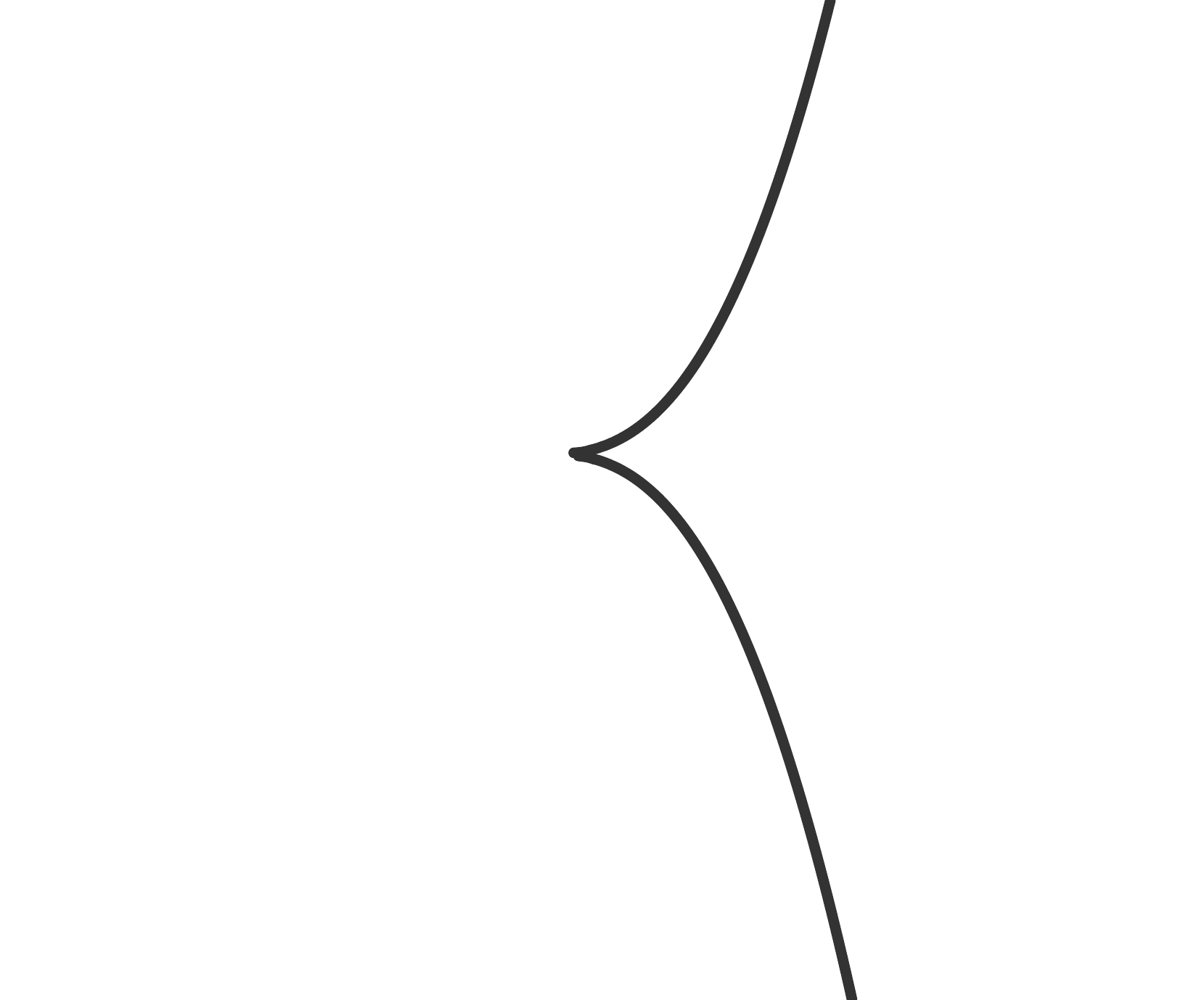}
      \caption{ramphoidal cusp}
    \end{subfigure}
    \hspace{0.1\textwidth}
\end{figure}

There are meridians (a generic loop) around these singularities. 
See \Cref{sec:Cocyclicity} for a list of the meridians of the six types of singularities. 
Given a generic path, the \textbf{higher Reidemeister moves} consist of
replacing a segment of the path which lies on the meridian of the six types of singularities 
by the remaining part of the meridian. (See \Cref{fig:higher_Reidemeister_move}). 
In addition, we identify a constant path with a path segment followed immediately by its reverse. 







\begin{figure}[H]
  \centering
  \begin{tikzcd}[column sep=3.5cm] 
    \includegraphics[height=7cm]{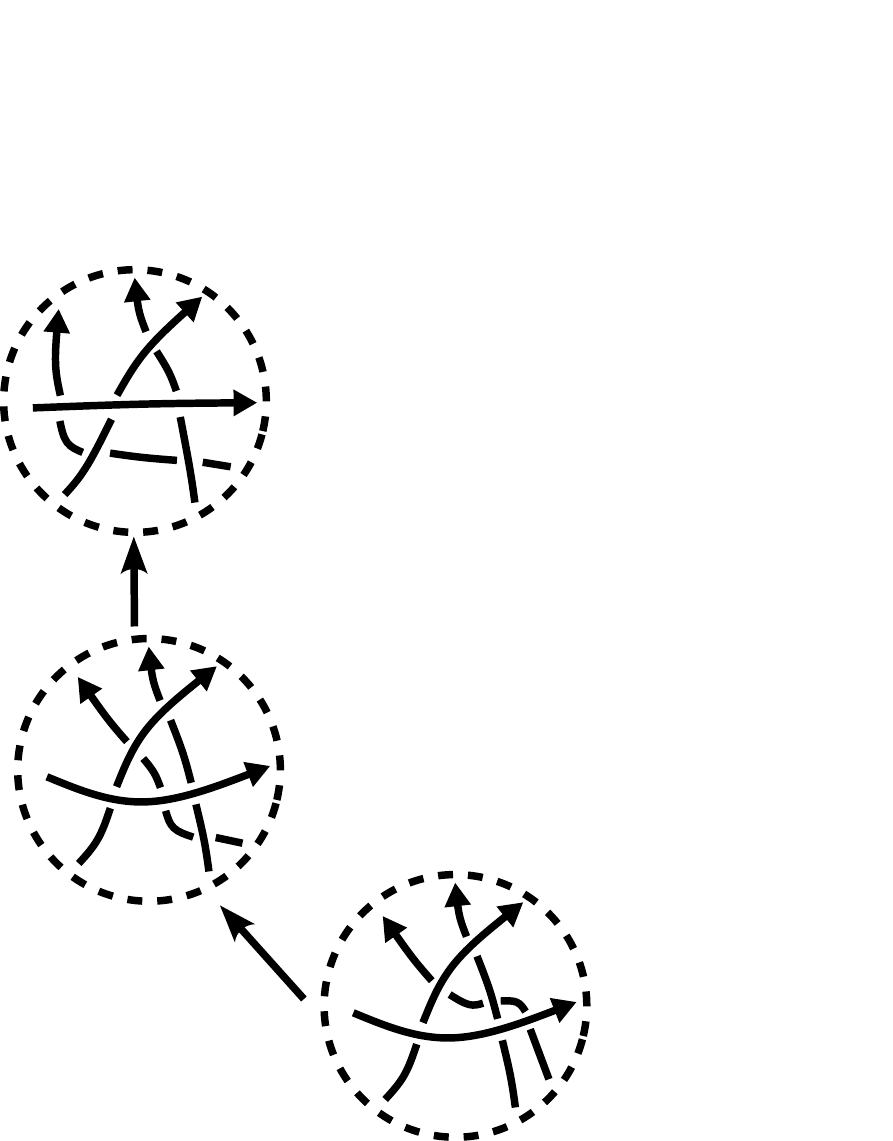}
      \arrow[r, <->, shift left=0.5ex, yshift=12ex, 
      "\text{higher Reidemeister move}" above] 
      & 
    \includegraphics[height=7cm]{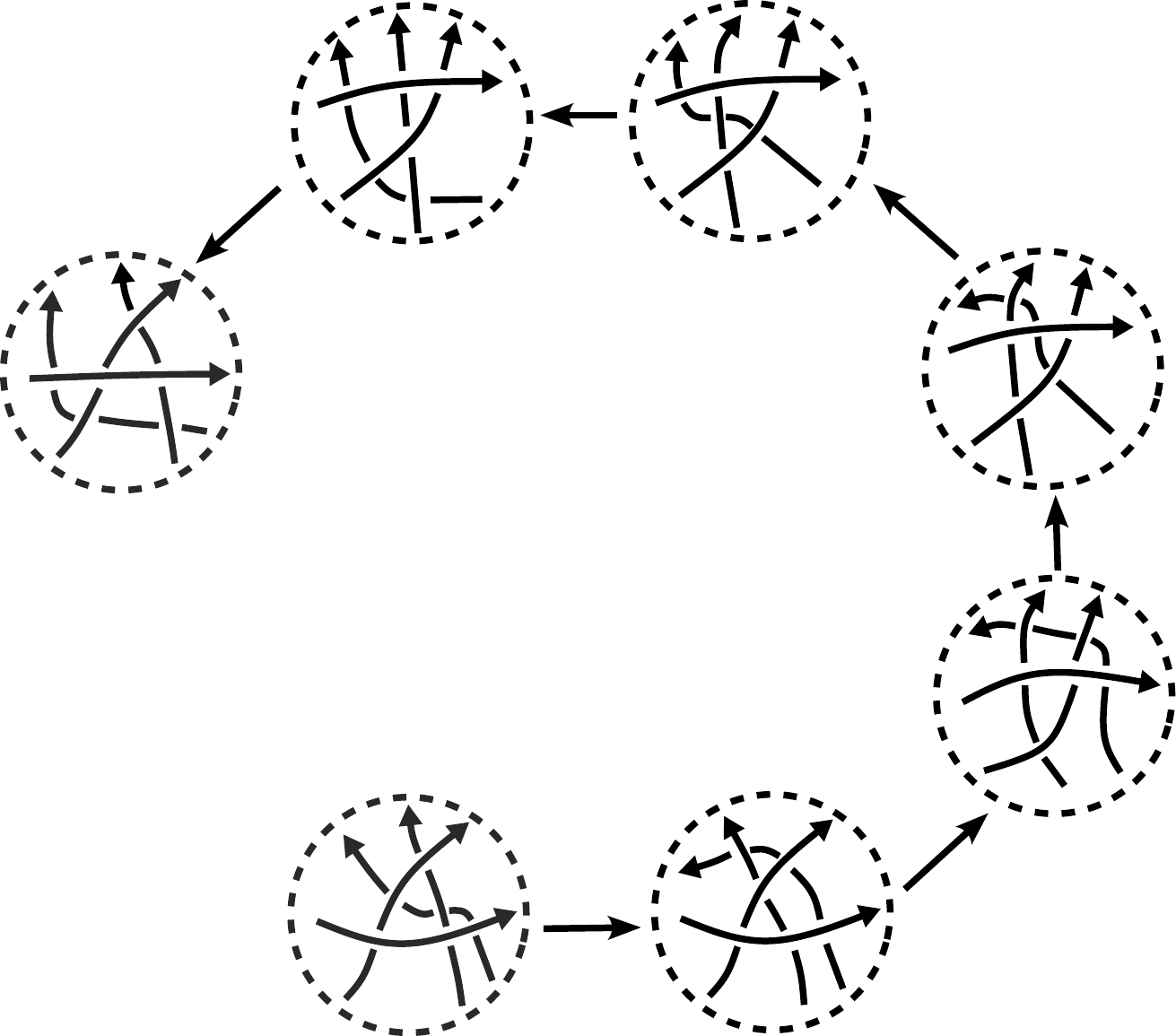} \\
  \end{tikzcd}
  \caption{A higher Reidemeister move for a path around an ordinary quadruple point}
  \label{fig:higher_Reidemeister_move}
\end{figure}

\begin{theorem}[\cite{David83, FiedlerKurlin10}]\label{thm:generic_path_homotopy}
    Two generic paths $\gamma_0$ and $\gamma_1$ are path-homotopic in $\mathcal{K}_{3,1}$ 
    if and only if there is a finite sequence of higher Reidemeister moves changing $\gamma_0$ to $\gamma_1$. 
\end{theorem}

By \Cref{thm:generic_path_homotopy}, 
the map $\phi$ defined by \Cref{eq:general_1cocycle} satisfies \ref{it:1cocycle_condition2}
if and only if $\phi(\text{meridian}) = 0$ for all the meridians around the 6 types of singularities. 

\section{Gauss diagrams}
\label{sec:Gauss_diagrams}

\begin{definition}
  A \textbf{Gauss diagram} consists of a counterclockwise oriented circle with a finite number of arrows connecting pairs of distinct points on the circle. 
  Each arrow may carry a sign ($+$ or $-$). 
  And there may also be a base point on the circle. In this paper, all Gauss diagrams are based. 
\end{definition}
\begin{definition}
  A \textbf{subdiagram} $S$ of a Gauss diagram $G$ is a Gauss diagram whose arrows form a subset of the arrows of $G$. 
\end{definition}

Gauss diagrams can be used to represent long knot diagrams. 
The base point on the Gauss diagram corresponds to the infinity point $\infty$ on the long knot. 
The circle represents the knot itself as $S^{1}$ with counterclockwise direction corresponding to the orientation of the knot diagram. 
The arrows represent the crossings of the knot diagram, with the foot at the undercrossing and the head at the overcrossing. 
The label of the arrow is the sign of the crossing. 
\begin{remark}
  \textbf{Warning}: The convention concerning the direction of the arrow is the opposite to that in a large part of the literature such as  \cite{PolyakViro94}. 
\end{remark}

\begin{figure}[htbp]
  \centering
  \begin{subfigure}[t]{0.25\textwidth}
    \centering
    \includegraphics[height=1.5cm]{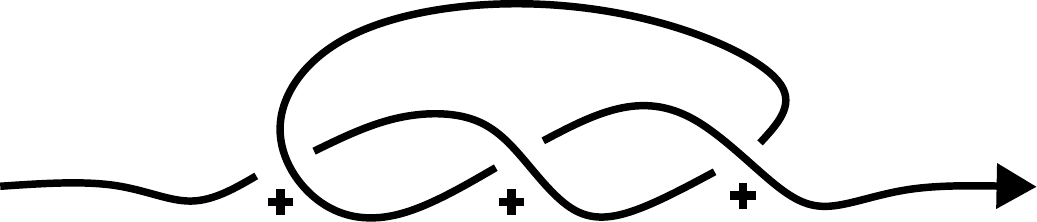}
    \caption{A long trefoil diagram}
    \label{fig:long_trefoil_right}
  \end{subfigure}
  \hspace{0.3\textwidth}
  \begin{subfigure}[t]{0.25\textwidth}
    \centering
    \includegraphics[height = 3cm]{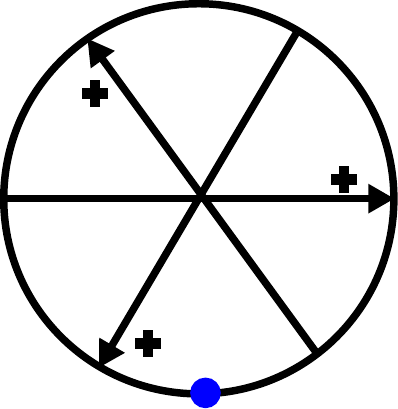}
    \caption{The Gauss diagram of the long knot diagram}
    \label{fig:GD_trefoil_diagram_based}
  \end{subfigure}
  \caption{A long knot diagram and the corresponding Gauss diagram}
  \label{fig:LongKnotD_and_GaussD}
\end{figure}

\begin{remark}
  Not all Gauss diagrams represent a real long knot. 
  Gauss diagrams are exactly the same as \textbf{virtual} long knots up to virtual moves. 
  See \cite[Theorem 1.A.]{GoussarovPolyakViro98}. 
\end{remark}

\subsection{Reidemeister move III}\label{sec:GD_Reidemeister}

Two knot diagrams represent isotopic knots if and only if one diagram can be transformed into the other by a finite sequence of Reidemeister moves.
There are three kinds of Reidemeister moves. We call them \textbf{R1}, \textbf{R2} and \textbf{R3} moves. 
For the 1-cocycles in this paper, we only consider R3 moves. 
An R3 move in a Gauss diagram involves the changes of the positions of three arrows that form the shape of a triangle. See \Cref{fig:eg_R3_GDs} for an example. 

\begin{definition}[Global types of R3 moves]\label{def:R3_global_type}
  The \textbf{global type of an R3 move} in a long knot is defined to be an integer as indicated below:
  \begin{figure}[H]
    \centering
    \begin{subfigure}[t]{0.2\textwidth}
      \centering
      \includegraphics[height=3cm]{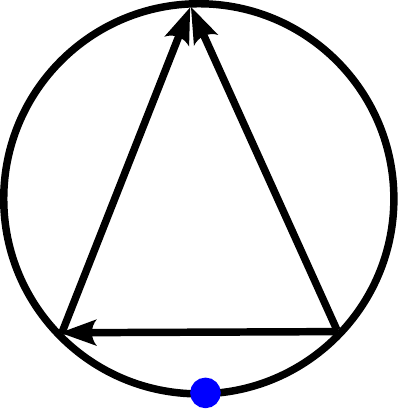}
      \caption{global type 0}
    \end{subfigure}
    \hspace{0.1\textwidth}
    \begin{subfigure}[t]{0.2\textwidth}
      \centering
      \includegraphics[height=3cm]{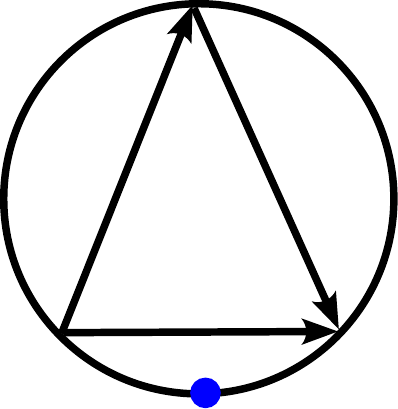}
      \caption{global type 1}
    \end{subfigure}
    \hspace{0.1\textwidth}
    \begin{subfigure}[t]{0.2\textwidth}
      \centering
      \includegraphics[height=3cm]{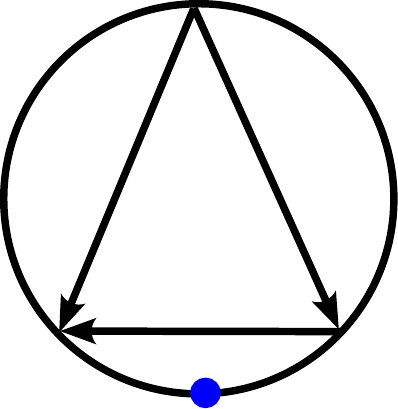}
      \caption{global type 2}
    \end{subfigure}
    \hspace{0.1\textwidth}
    \begin{subfigure}[t]{0.2\textwidth}
      \centering
      \includegraphics[height=3cm]{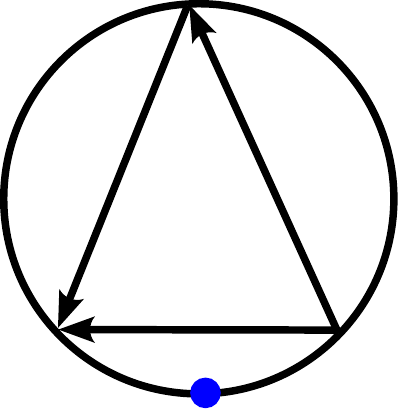}
      \caption{global type 3}
    \end{subfigure}
    \hspace{0.1\textwidth}
    \begin{subfigure}[t]{0.2\textwidth}
      \centering
      \includegraphics[height=3cm]{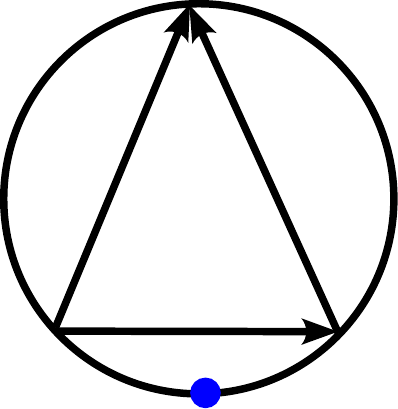}
      \caption{global type 4}
    \end{subfigure}
    \hspace{0.1\textwidth}
    \begin{subfigure}[t]{0.2\textwidth}
      \centering
      \includegraphics[height=3cm]{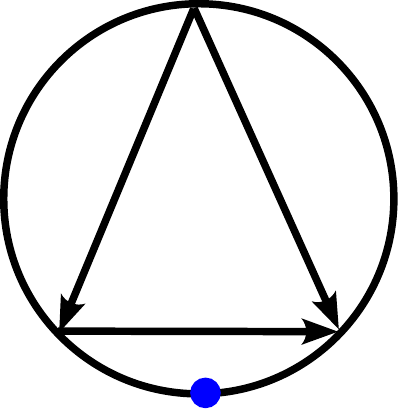}
      \caption{global type 5}
    \end{subfigure}
    \caption{Global types of R3 moves}
  \end{figure}
  We call global types 0, 1, and 2 the \textbf{left types of R3 moves} and global types 3, 4, and 5 the \textbf{right types of R3 moves}.
\end{definition}

\begin{definition}[sign of an R3 move]\label{def:sign_R3}
  We define the \textbf{sign of an R3 move} to be 1 if the corresponding Gauss diagram changes from the left side to the right side of \Cref{fig:sign_R3_moves}, and $-1$ otherwise.
  \begin{figure}[H]
\centering

\begin{tikzcd}[column sep=3cm]
{
\includegraphics[height=3cm]{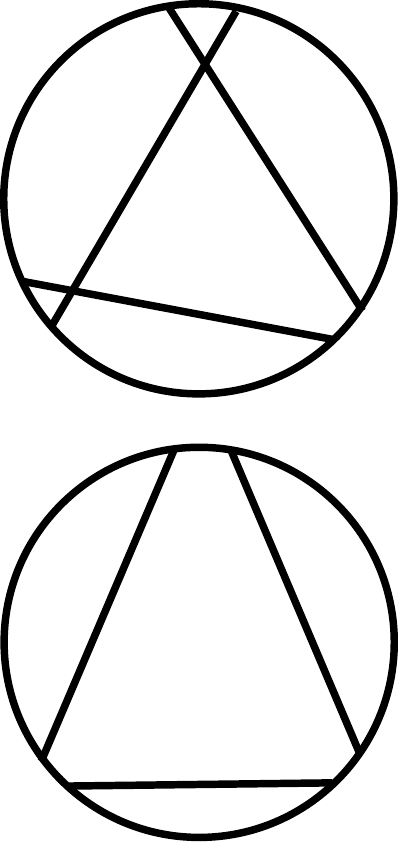}
}
\arrow[
    r,
    shift left=0.5ex,
    yshift=8mm,
    ->,
    >=latex,
    line width=1pt,
    "\scriptstyle +" above
]
&
{
\includegraphics[height=3cm]{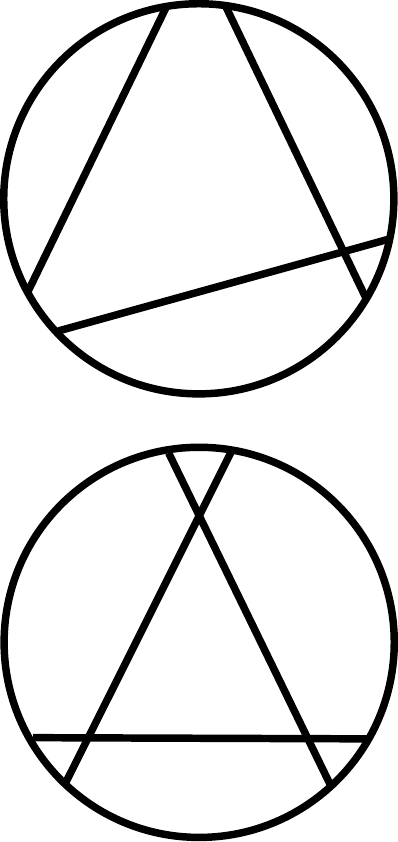}
}
\arrow[
    l,
    shift left=0.5ex,
    yshift=8mm,
    ->,
    >=latex,
    line width=1pt,
    "\scriptstyle -" below
]
\end{tikzcd}

\caption{Sign of R3 moves}
\label{fig:sign_R3_moves}
\end{figure}
\end{definition}


\subsection{Gauss diagram formulae}

Let $\mathscr{A}$ be the free $\mathbb{Z}$-module on the basis of all Gauss diagrams whose arrows are all signed. 
For any two Gauss diagrams $A, B$, we define their inner scalar product by \begin{equation}
 (A,B) \coloneqq \left\{
\begin{array}{rcl}
1 && {\text{$A=B$}}\\
0 && {\text{otherwise.}}
\end{array}
 \right.
\end{equation}
(Equality of two Gauss diagrams is considered up to orientation-preserving homeomorphism of the underlying based circle.)
We extend this scalar product to the whole $\mathscr{A}$ by linearity. 
Then we define $$\langle A,G\rangle \coloneqq (A,\sum\limits_{ C \subset G}C)$$ for any Gauss diagrams $A$ and $G$, where $C\subset G$ means that $C$ is a subdiagram of $G$.  
In words, this product counts the number of times $A$ appears in $G$ as a subdiagram. 
Let $F= \sum\limits_{i}c_iA_i$ be an element of $\mathscr{A}$. 
If 
\[
\langle F,\cdot \rangle : \{\text{Gauss diagrams of long knots}\} \rightarrow \mathbb{Z}
\]
is a long knot invariant, we say that $F$ is a \textbf{Gauss diagram formula} (\textbf{GDF}). 
\begin{remark}
  By \Cref{rem:closure_operation}, a long knot invariant induces a knot invariant. 
  Hence, a Gauss diagram formula $F$ for long knots can also be applied to the Gauss diagram $G$ of a knot by adding a base point on $G$.
  The resulting value is independent of the position of the added base point because of the existence of the rolling moves (see \Cref{def:rolling move}). 
\end{remark}
\begin{example}
  \[
    \left\langle 
      \raisebox{-.4\height}{\includegraphics[width=1.75cm]{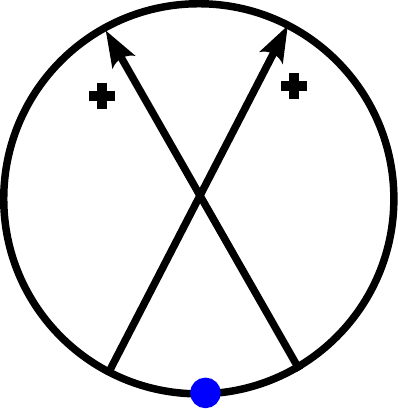}}  \ ,\ 
      \raisebox{-.4\height}{\includegraphics[width=1.75cm]{figures/Gauss_Diagrams/GD_trefoil_right_based.pdf}}
    \right\rangle = 1 , \quad 
    \left\langle 
      \raisebox{-.4\height}{\includegraphics[width=1.75cm]{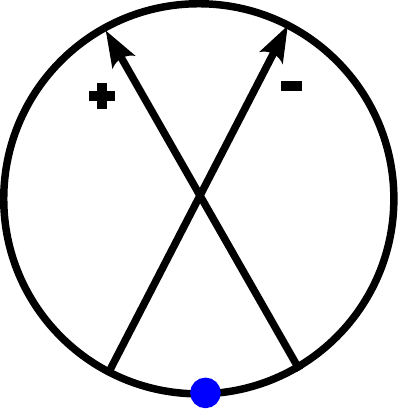}}  \ ,\ 
      \raisebox{-.4\height}{\includegraphics[width=1.75cm]{figures/Gauss_Diagrams/GD_trefoil_right_based.pdf}}
    \right\rangle = 0.
  \]

  \[
    \left\langle 
      \raisebox{-.4\height}{\includegraphics[width=1.75cm]{figures/Gauss_Diagrams/v2_1.pdf}} - 
      \raisebox{-.4\height}{\includegraphics[width=1.75cm]{figures/Gauss_Diagrams/v2_2.pdf}} - 
      \raisebox{-.4\height}{\includegraphics[width=1.75cm]{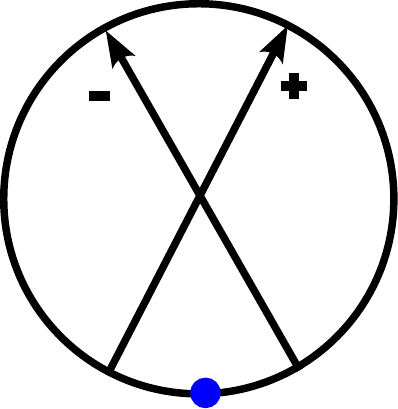}} + 
      \raisebox{-.4\height}{\includegraphics[width=1.75cm]{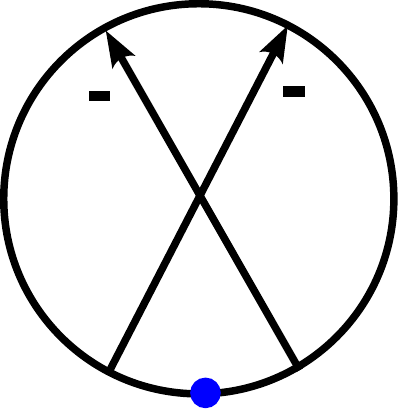}}  \ , \ 
      \raisebox{-.4\height}{\includegraphics[width=1.75cm]{figures/Gauss_Diagrams/GD_trefoil_right_based.pdf}}
    \right\rangle = 1 -0 -0 + 0 = 1.
  \]
\end{example}

We introduce the notion of \textbf{unsigned Gauss diagrams}. 
Let $A$ be an unsigned Gauss diagram, whose arrows are $\alpha_1, \alpha_2, \ldots, \alpha_m$. 
Let $A_{\epsilon_1\epsilon_{2}\cdots \epsilon_{m}}$ be the signed Gauss diagram obtained from $A$ by labelling each arrow $\alpha_i$ by $\epsilon_i$, where $\epsilon_i = \pm 1$. 
We identify the unsigned Gauss diagram $A$ with an element of $\mathscr{A}$ by the formula: 
\[
A \coloneqq \sum\limits_{(\epsilon_1,\epsilon_2,...,\epsilon_{m}) \in \{\pm 1\}^{m}}\epsilon_1\epsilon_{2}\cdots \epsilon_{m}A_{\epsilon_1\epsilon_{2}\cdots \epsilon_{m}},
\]
In other words, if $A$ is an unsigned Gauss diagram, and $G$ is a signed Gauss diagram,
then $\left\langle A, G \right\rangle$ computes the algebraic sum of over all occurrences of the shape $A$ as a subdiagram of G. 
The sign of each appearance of the subdiagram is equal to the product of the signs of the unlabelled arrows of the subdiagram.
\begin{remark}
  For an element $F \in \mathcal{A}$, we will abuse the notation and write $F$ for $\left\langle F, \cdot \right\rangle$.
\end{remark}
See \Cref{sec:GDF} for examples of Gauss diagram formulae. 

Gauss diagram formulae were introduced by Polyak, Viro and Fiedler in \cite{PolyakViro94,Fiedler93,Fiedler01}. 
They are closely related to Vassiliev invariants. 
\begin{theorem}[Theorem 3 in \cite{PolyakViro94}]\label{thm:GDF_finiteness}
  Let $F$ be a linear combination of Gauss diagrams with at most $n$ arrows. 
  If $F$ is a knot invariant, 
  then $F$ is a Vassiliev invariant of order at most $n$.
\end{theorem}
\begin{proof}
  This is a corollary of \Cref{lem:GDF_finiteness}. 
\end{proof}

\begin{remark}
    In fact, Goussarov has proved the important converse result \cite[Corollary 3.B]{GoussarovPolyakViro98}. 
    Any order-$n$ integer-valued Vassiliev invariant of long knots can be presented by a Gauss diagram formula with at most $n$ arrows.
\end{remark}

\begin{lemma}\label{lem:GDF_finiteness}
  For any singular long knot diagram $K$ with $(n + 1)$ singular crossings and a signed Gauss diagram $G$ with $n$ arrows $(n\in \mathbb{N})$, we have $\left\langle 
    G, K
  \right\rangle = 0$. 
\end{lemma}
\begin{proof}
  Every time we find a subdiagram $D = G \subset D_{\epsilon_1, \epsilon_2, ..., \epsilon_{n + 1}}$, 
  there must be at least one arrow from a singular crossing which is not in the subdiagram $D$. 
  After resolving the singular crossings outside $D$, all the resulting resolved diagrams will have the contribution of the subdiagram $D$, 
  but they will cancel out automatically in the sum 
  \[
  \left\langle 
    G, K
  \right\rangle = 
  \sum_{(\epsilon_1, \epsilon_2, ..., \epsilon_{n + 1})\in\{\pm 1\}^{n+1}}
  \epsilon_1 \epsilon_2 ... \epsilon_{n + 1}
  \left\langle 
    G, D_{\epsilon_1, \epsilon_2, ..., \epsilon_{n + 1}}
  \right\rangle
  \]
  due to the algebraic fact 
  \[
  \sum_{(\epsilon_1,\ldots,\epsilon_{t})\in\{\pm 1\}^{t}}\epsilon_1\cdots\epsilon_{t} = 0, \; (t \ge 1).
  \]
\end{proof}

\subsection{Gauss diagrams for combinatorial 1-cocycles}
Recall that an R3 move corresponds to a triangle whose three sides are arrows in Gauss diagrams (see \Cref{sec:GD_Reidemeister}). 
Now we introduce a new kind of Gauss diagram with such a triangle, which will be used to define various weight functions in \Cref{eq:general_1cocycle} (see \Cref{sec:order4_1cocycles}).
A \textbf{Gauss diagram with a triangle} is a Gauss diagram $G$ with a triangle indicating the type of Reidemeister move. 
This type may be a global type, a local type or a full type of R3 moves. 
Notice that the global type can be indicated directly by the direction of arrows of the triangle. (See \Cref{def:R3_global_type}). 
In addition to the triangle there might be other arrows in the Gauss diagram. 
These arrows can be signed or unsigned. For example, see \Cref{fig:GDs_triangle}. 
These arrows outside the triangle are called \textbf{contributing arrows}. 

\begin{figure}[htbp]
  \centering
  \begin{subfigure}[t]{0.25\textwidth}
    \centering
    \includegraphics[height=2.5cm]{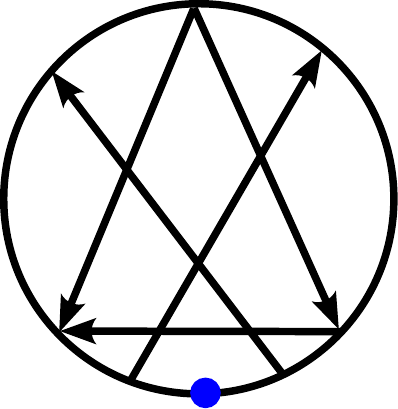}
  \end{subfigure}
  \hspace{0.3\textwidth}
  \begin{subfigure}[t]{0.25\textwidth}
    \centering
    \includegraphics[height = 2.5cm]{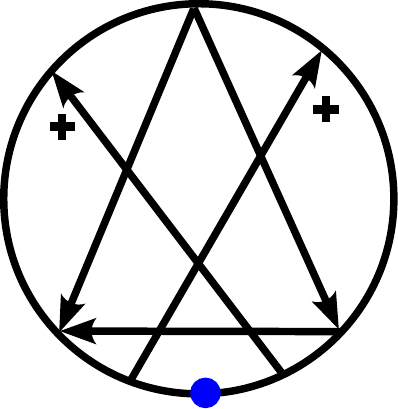}
  \end{subfigure}
  \caption{Two Gauss diagrams with a triangle. Each has 2 contributing arrows. }
  \label{fig:GDs_triangle}
\end{figure}

A Gauss diagram with a triangle is regarded as a map. 
It maps a Reidemeister move to an integer. 
Let $G$ be a Gauss diagram with a triangle. 
The value $\left\langle G, p \right\rangle$ of $G$ on a Reidemeister move $p$ is defined as follows:
\begin{itemize}
    \item If the Reidemeister move $p$ is not an R3 move of the indicated type on $G$,
        the value is $0$.
    \item If the move $p$ is such an R3 move:
        \begin{itemize}
            \item When $G$ is signed: the value is the number of times $G$ appears as a subdiagram in the Gauss diagram of the R3 move.
            \item When $G$ is unsigned: the value is the sum of the contributions of all such subdiagrams, 
                where the contribution of a subdiagram is the product of the signs of the arrows not belonging to the triangle. 
        \end{itemize}
\end{itemize}
From now on, we will also refer to Gauss diagrams with a triangle simply as \textbf{Gauss diagrams}.
By linearity, we also allow Gauss diagram expressions, i.e., formal linear combinations of Gauss diagrams.
Furthermore, we extend the valuation of a Gauss diagram expression to generic paths by requiring it to be additive with respect to concatenation of paths.
\begin{example} \ \\
    Let $p$ be the following R3 move:
    \begin{figure}[H]
    \centering
    \begin{tikzcd}[column sep=3cm]
      {\includegraphics[height=2cm]{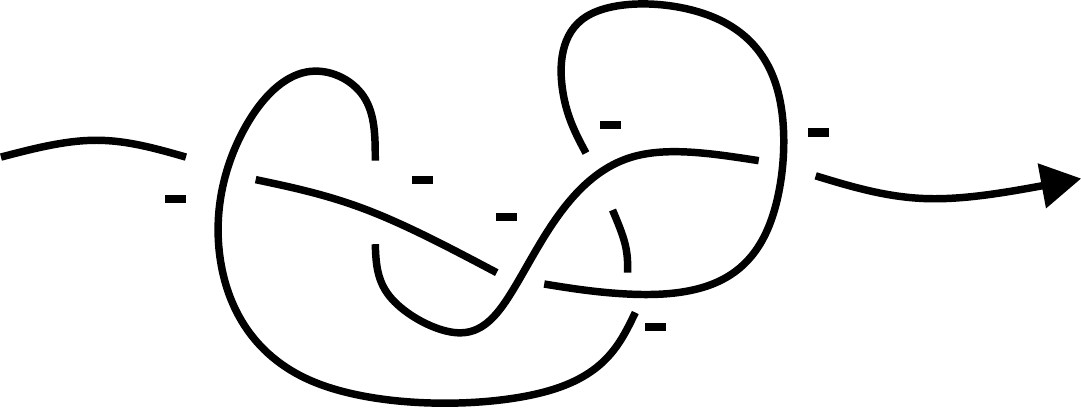}}
      \arrow[r, ->, >=latex, line width=1pt, "\text{R3 move}" above]
      &
      {\includegraphics[height=2cm]{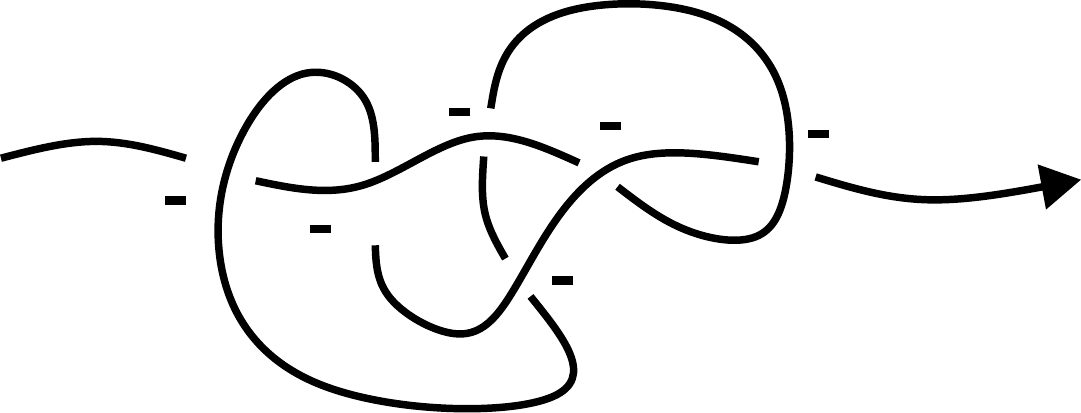}}
    \end{tikzcd}
    \caption{The R3 move $p$}
    \end{figure}
    The Gauss diagrams are as follows:
    \begin{center}
        \begin{tikzcd}[column sep=6cm]
            \includegraphics[height=3cm]{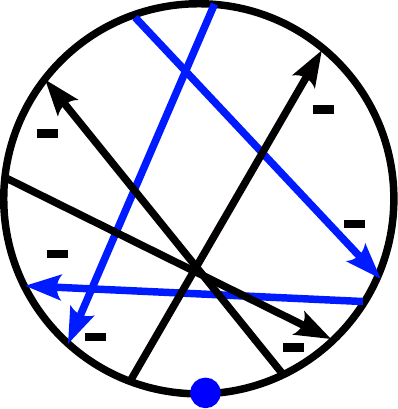}
            \arrow[r, ->, >=latex, line width=1pt,
                "{\includegraphics[height=2.5cm]{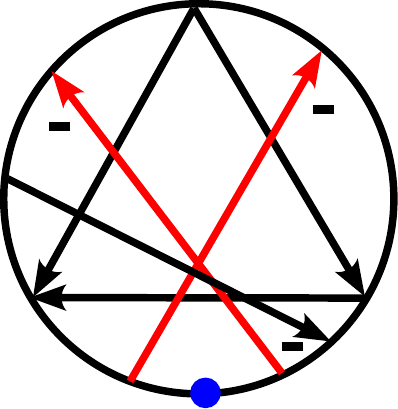}}", 
                "{\text{R3 move}}"']                                                   
            &
            \includegraphics[height=3cm]{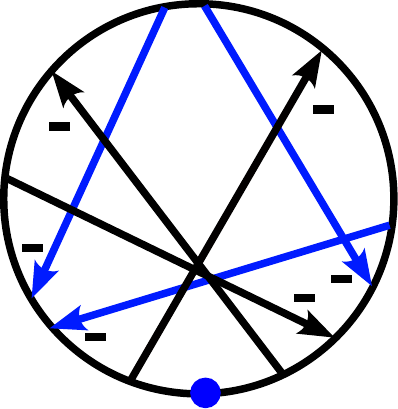}
        \end{tikzcd}
        \captionof{figure}{The Gauss diagram changes for the R3 move $p$}\label{fig:eg_R3_GDs}
    \end{center}

    The following are the values of some Gauss diagram expressions on this R3 move. 
    \[
        \left\langle 
        \raisebox{-.4\height}{\includegraphics[width=1.75cm]{figures/1cocycles/cfg_eg_unlabelled.pdf}} \ , \ 
        p
        \right\rangle = (-1)\times(-1) = 1
    \]
    \[
        \left\langle 
        \raisebox{-.4\height}{\includegraphics[width=1.75cm]{figures/1cocycles/cfg_eg_labelled.pdf}} \ , \ 
        p
        \right\rangle = 0, \ 
        \left\langle 
        \raisebox{-.4\height}{\includegraphics[width=1.75cm]{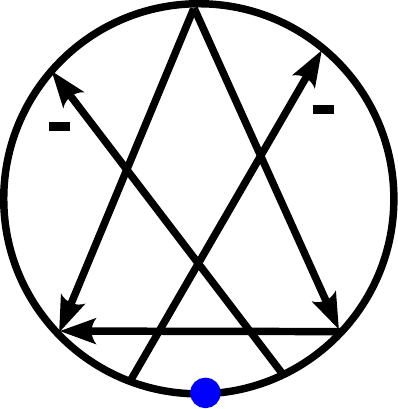}} \ , \ 
        p
        \right\rangle = 1, \ 
        \left\langle 
        \raisebox{-.4\height}{\includegraphics[width=1.75cm]{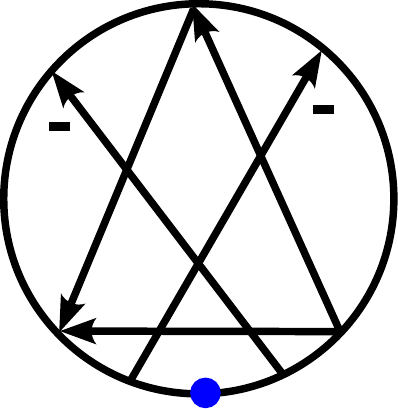}} \ , \ 
        p
        \right\rangle = 0, \ 
    \]

\end{example}

\section{1-cocycles of combinatorial order 4}
\label{sec:order4_1cocycles}


\begin{definition}\label{def:beta_1_2}
    Let $C_g(I, \mathcal{K}_{3,1})$ be the set of generic paths (\Cref{def:generic_path}) in the space of long knots. 
    Define two maps 
    \[
    \begin{aligned}
    \beta_{i} \colon C_g(I, \mathcal{K}_{3,1}) &\longrightarrow \mathbb{Z}, \\
    \gamma \;&\longmapsto\; \beta_{i}(\gamma) \coloneqq 
    \sum_{\text{R3 move } p \in \gamma} \mathrm{sign}(p) \cdot W_{\beta_{i}}(p), \quad i = 1, 2, 
    \end{aligned}
    \]
    where $W_{\beta_1}$ is given in \Cref{fig:beta_1}, $W_{\beta_2}$ is given in \Cref{fig:beta_2} and $sign(p)$ is defined in \Cref{def:sign_R3}. 
\end{definition}
\begin{figure}[H] 
    \centering
    \includegraphics[width=1.0\textwidth]{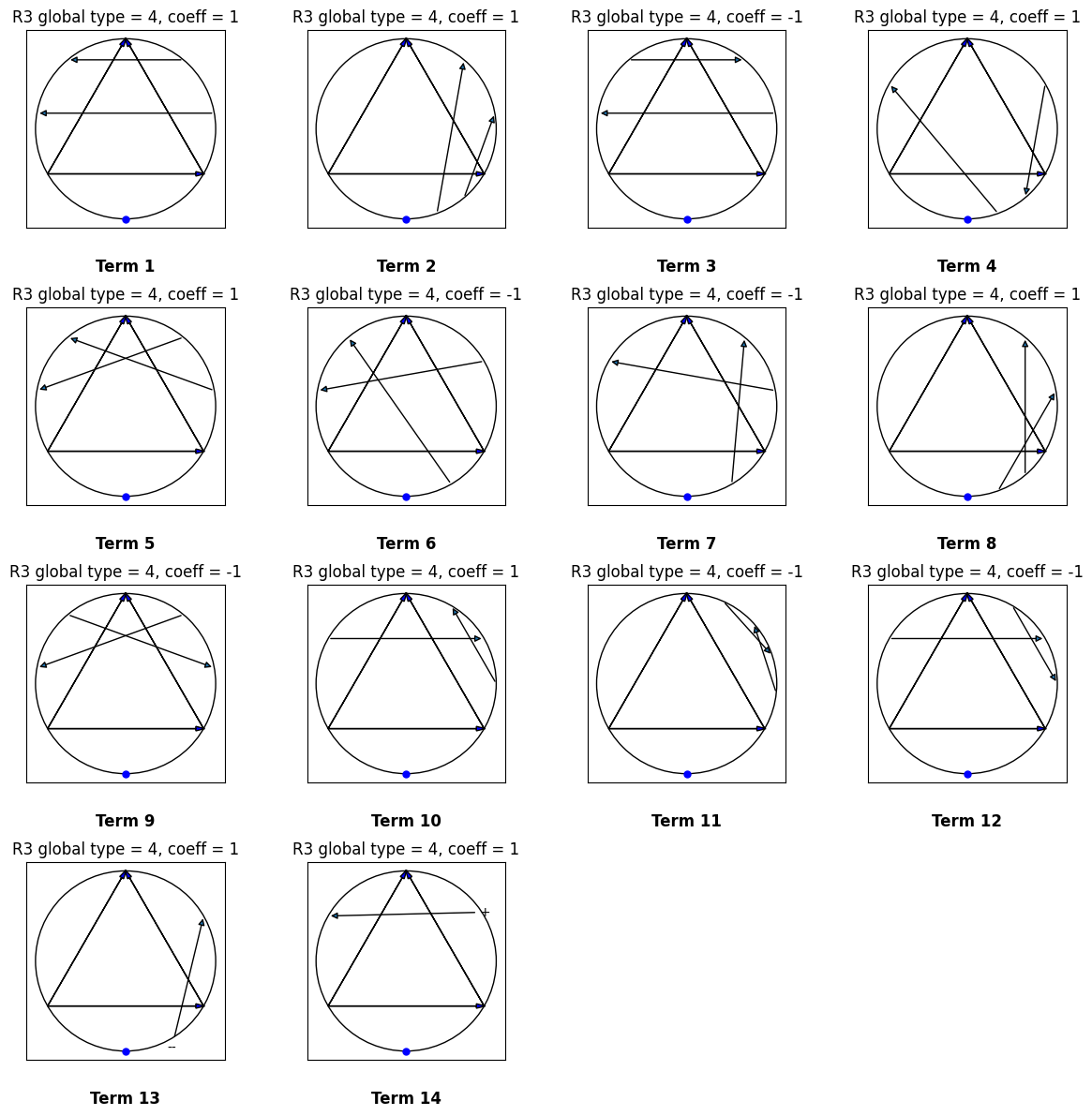}  
    \caption{The weight $W_{\beta_1}$ of $\beta_1$.}
    \label{fig:beta_1}
\end{figure}

\begin{figure}[H] 
    \centering
    \includegraphics[width=1\textwidth]{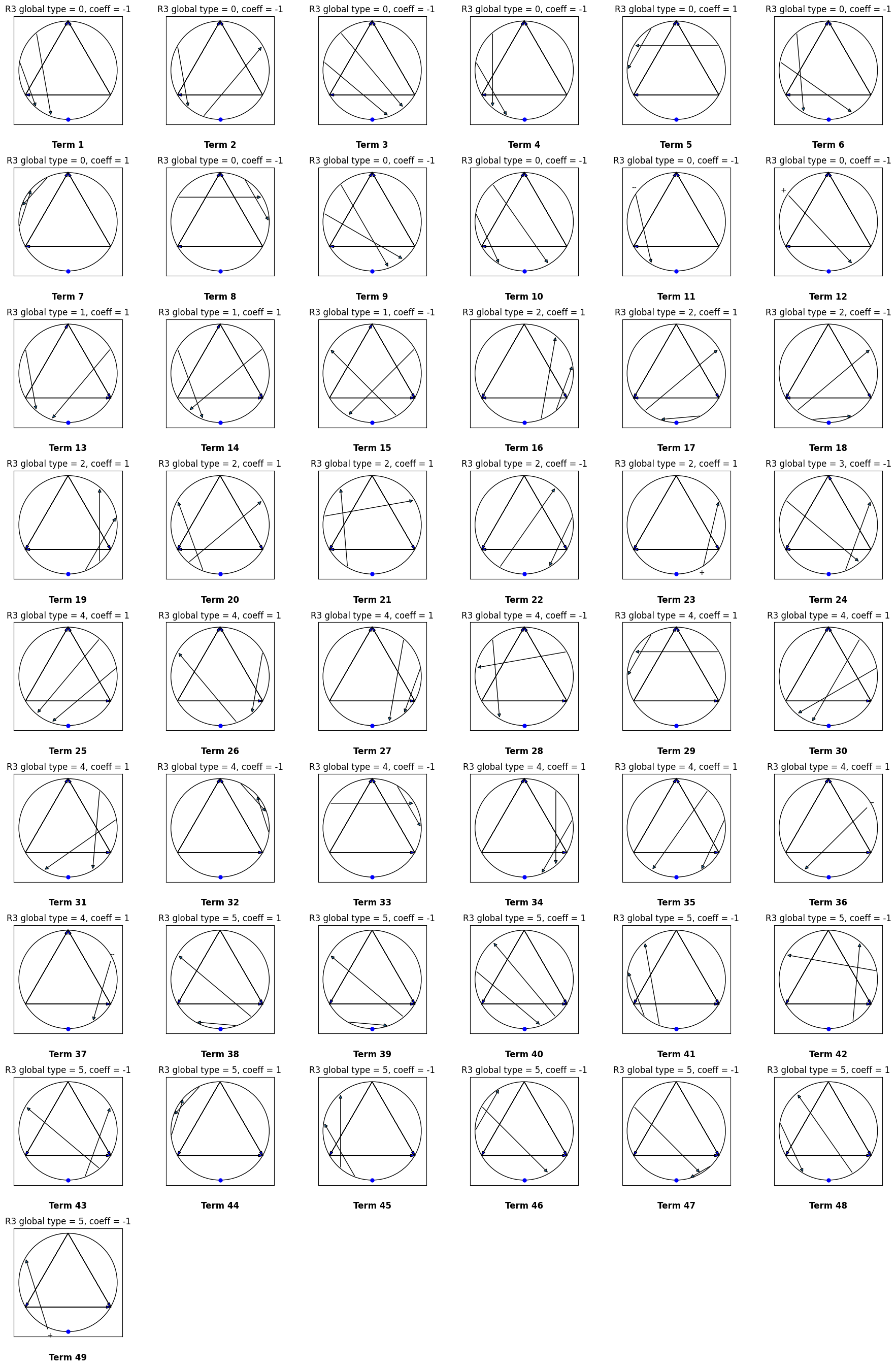}  
    \caption{The weight $W_{\beta_2}$ of $\beta_2$.}
    \label{fig:beta_2}
\end{figure}


\begin{definition}\label{def:beta_3}
    Let $C_g(I, \mathcal{K}_{3,1})$ be the set of generic paths (\Cref{def:generic_path}) in the space of long knots. 
    Define a map
    \[
    \begin{aligned}
    \beta_{3} \colon C_g(I, \mathcal{K}_{3,1}) &\longrightarrow \mathbb{Z}/2\mathbb{Z}, \\
    \gamma \;&\longmapsto\; \beta_{3}(\gamma) \coloneqq 
    \sum_{\text{R3 move } p \in \gamma} W_{\beta_{3}}(p), 
    \end{aligned}
    \]
    where $W_{\beta_3}$ is given in \Cref{fig:beta_3}. 
\end{definition}
\begin{figure}[H] 
    \centering
    \includegraphics[width=1.0\textwidth]{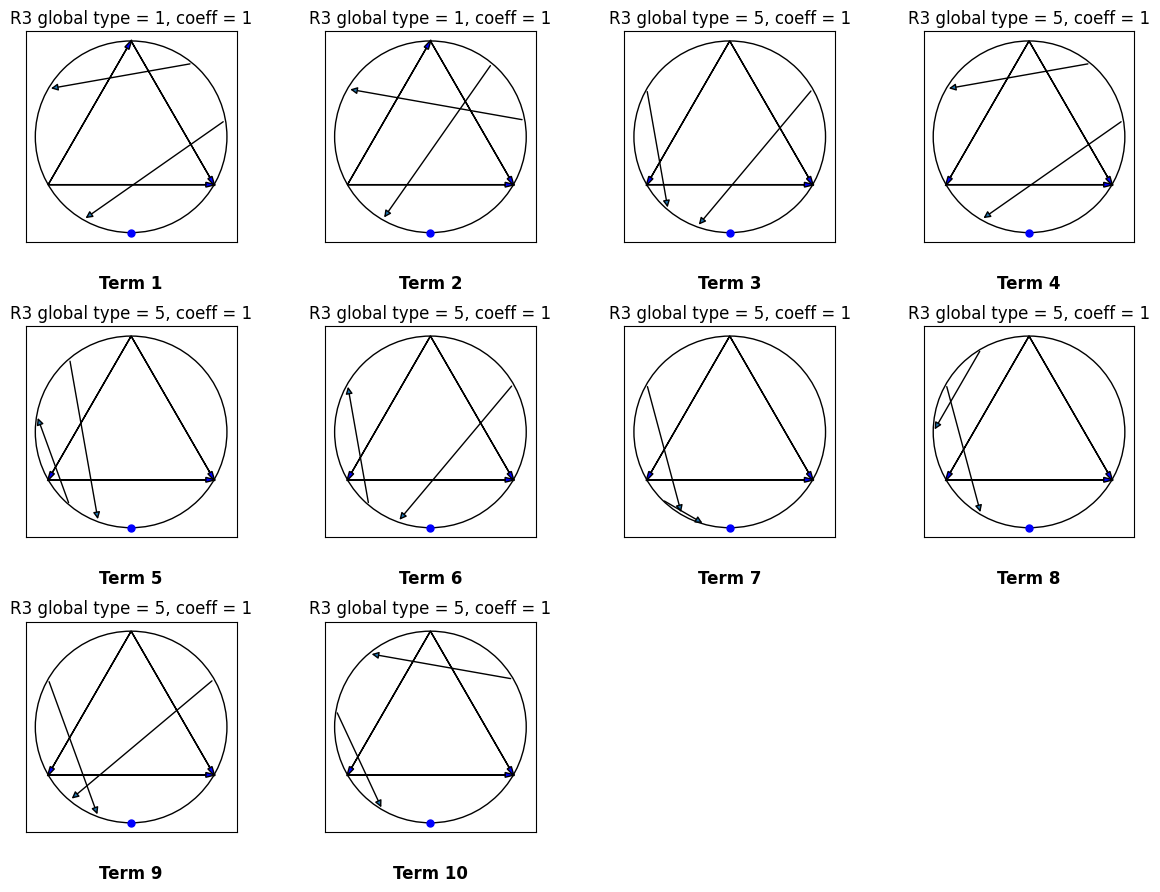}  
    \caption{The weight $W_{\beta_3}$ of $\beta_3$.}
    \label{fig:beta_3}
\end{figure}


\begin{theorem}[Cocyclicity]\label{thm:cocyclicity_beta}
    $\beta_1, \beta_2$ and $\beta_3$ are combinatorial 1-cocycles of combinatorial order 4. 
\end{theorem}
The proof of \Cref{thm:cocyclicity_beta} is given in \Cref{sec:Cocyclicity}.

\section{Pairing with loops}
\label{sec:pairing}

\subsection{Algorithm}
In order to compute our 1-cocycles on the arcs or loops in the space of long knots, we need a systematic way to decompose a given path into Reidemeister moves. 
We introduce an algorithm using Reidemeister moves to realise the push arc for long knot diagrams. 
We use it to obtain the algorithms for the rotation loop and the bracket loops.
Then we adapt this algorithm to realise the rolling loop. 


Smoothing all the crossings of a long knot diagram (in the sense of Conway) yields a line extending to infinity and several circles, which are called \textbf{Seifert circles}. 
We say a Seifert circle \textbf{touches} another Seifert circle at a crossing if the two arcs obtained by smoothing the crossing belong to the two Seifert circles respectively. 

\begin{definition}[Layer of Seifert circles]
    We say that the line to infinity is of \textbf{layer 0}.  
    Inductively, a Seifert circle is said to be of \textbf{layer $n$}  
    if it touches at a crossing a Seifert circle of layer $(n-1)$,  
    and does not touch any circle of layer less than $(n-1)$.
\end{definition}
\begin{example}
    See \Cref{fig:Seifert_circles}. 
\end{example}
\begin{remark}
    The layer of a Seifert circle is well-defined for long knot diagrams. 
    This is because there are no isolated Seifert circles for long knot diagrams. 
    All the Seifert circles can be connected by the crossings of the diagram. 
\end{remark}

\subsubsection{Push arc}\label{algo:push}

\begin{definition}
    A \textbf{push move} is a move depicted in \Cref{fig:push_move}. 
    After performing a push move as in \Cref{fig:push_move}, 
    we say that the left crossing of the original diagram is \textbf{conquered} and the right crossing of the original diagram is \textbf{pushed}. 
\end{definition}
\begin{figure}[H]
    \center  
    \includegraphics[height = 3.2cm]{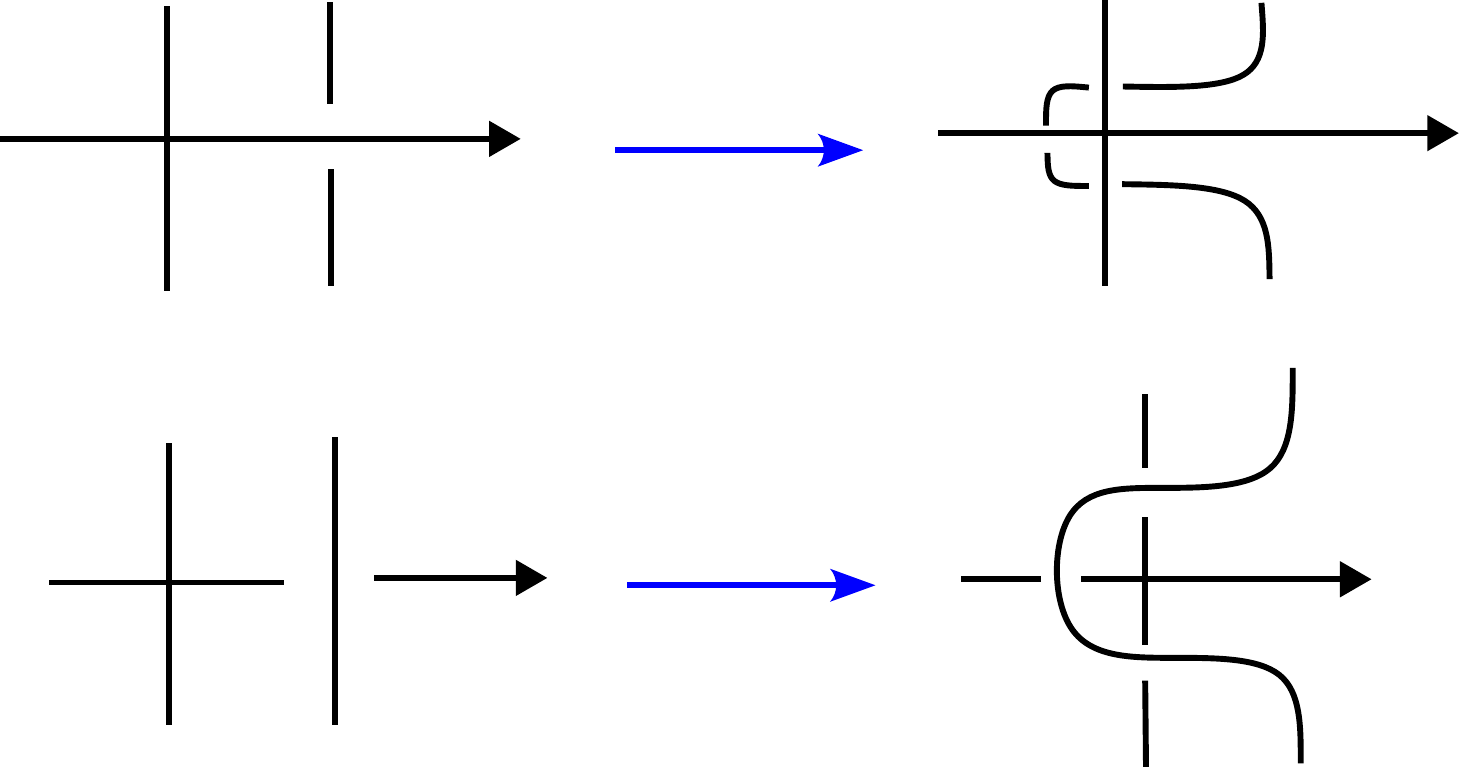}
    \caption{Two types of push moves. The intersection means that either strand may be chosen as the overcrossing or the undercrossing.}
    \label{fig:push_move}
\end{figure}
There are 16 types of push moves ($2^2$ types of crossings and $2^2$ types of choice of the orientations on the 2 strands). 
The push arc is realised by pushing crossings through a long knot diagram consecutively. 
The algorithm for pushing a crossing through a long knot diagram consists of 3 steps. 
\begin{enumerate}
    \item Obtain the Seifert circles of the long knot diagram (see \Cref{fig:Seifert_circles}). 
    \item \label{item:push_along} Perform push moves against the orientation of the Seifert circles from layer 0 to the maximal layer. 
    There might be several starting points on several Seifert circles. 
    We do the push moves from all the starting points at the same time.  
    If the push move meets a previously conquered crossing, this push move stops in this layer (see \Cref{fig:push_along_paths}). 
    \item Now all the crossings of the long knot are conquered. Do R2 moves to reduce the crossings in any possible order (see \Cref{fig:push_crossing_R2s}). The order of R2 moves is irrelevant for our purposes. 
\end{enumerate}

\begin{definition}\label{def:arrows_from_pushing}
    At Step \ref{item:push_along}, the push moves create many new crossings. 
    We collectively refer to these newly created crossings, together with the crossing being pushed, as the \textbf{crossings from pushing}.
    In terms of Gauss diagrams, the corresponding arrows are called the \textbf{arrows from pushing}.
\end{definition}
\begin{remark}\label{rem:cr_pushed_sign}
    The arrows being pushed in Step \ref{item:push_along} could also come from the newly created crossings. 
    But their signs are always the same (the original sign of the crossing in $G$ being pushed). 
\end{remark}

\begin{figure}[H]
  \centering
  \begin{subfigure}[t]{0.7\textwidth}
    \centering
    \includegraphics[height=2.5cm]{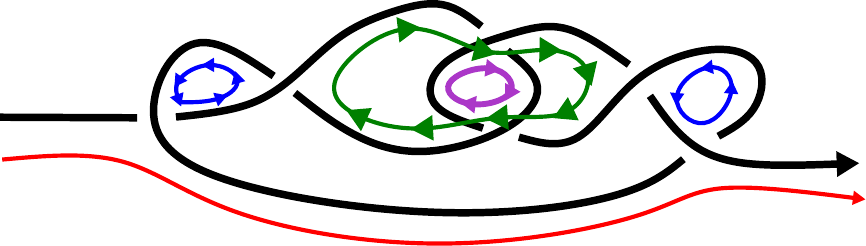}
    \caption{Seifert circles. 
    The red line is of layer 0, 
    the two blue circles are of layer 1, 
    the green circle is of layer 2, 
    and the purple circle is of layer 3.}
    \label{fig:Seifert_circles}
  \end{subfigure}
  \hfill
    \begin{subfigure}[t]{0.5\textwidth}
        \centering
        \includegraphics[height = 3cm]{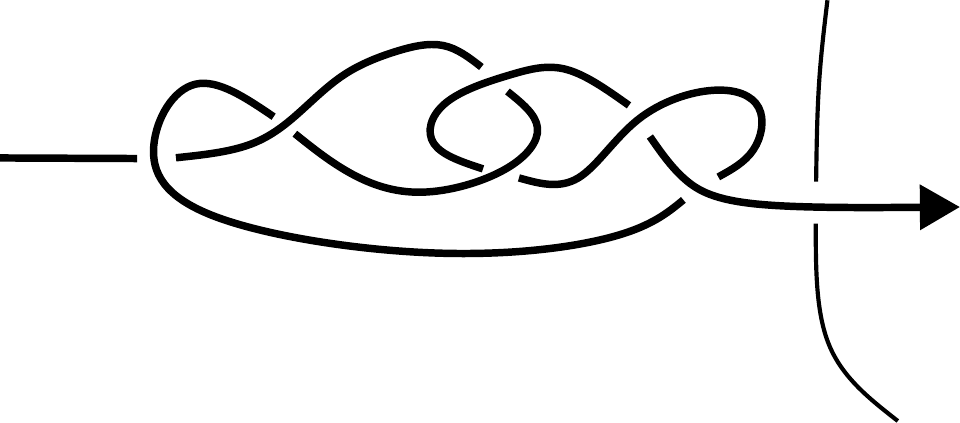}
        \caption{Crossing to be pushed}
        \label{fig:cr_to_be_pushed}
    \end{subfigure}
  \hfill
  \begin{subfigure}[t]{0.7\textwidth}
    \centering
    \includegraphics[height = 3cm]{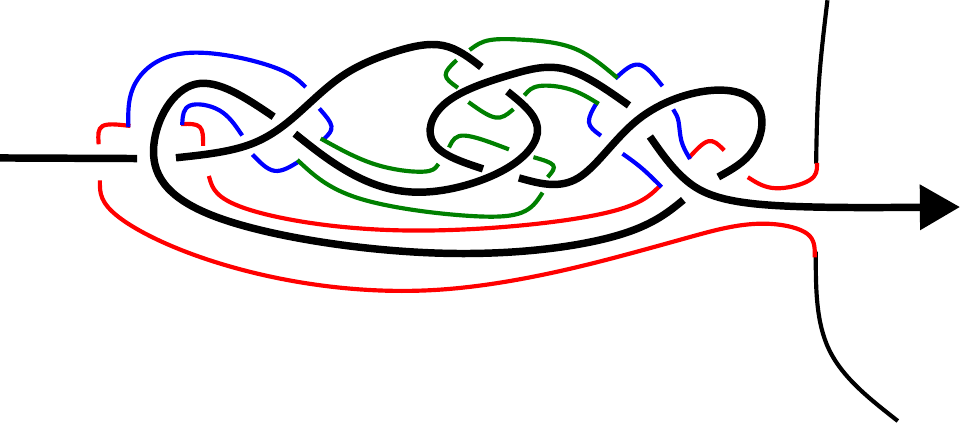}
    \caption{Push moves against the orientation of the Seifert circles in the order of layers}
    \label{fig:push_along_paths}
  \end{subfigure}
  \hfill
  \begin{subfigure}[t]{0.7\textwidth}
    \centering
    \includegraphics[height = 3cm]{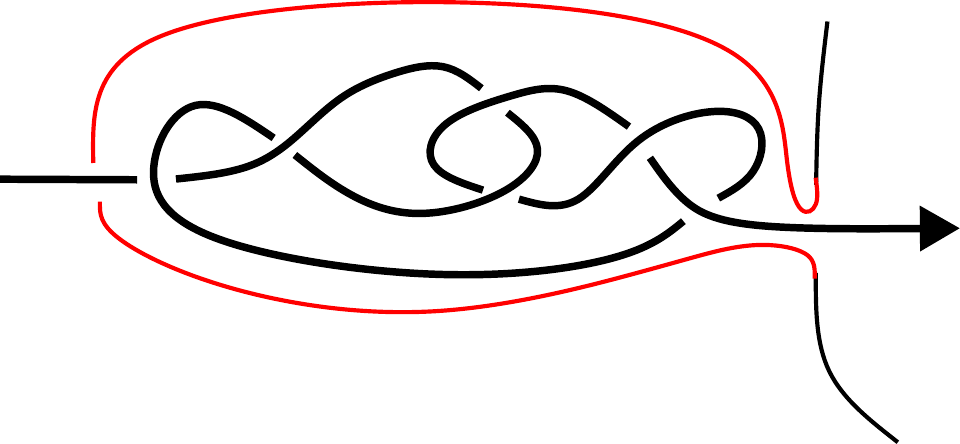}
    \caption{Do R2 moves to reduce the crossings}
    \label{fig:push_crossing_R2s}
  \end{subfigure}

  \caption{Push a crossing through a long knot diagram via push moves. All crossings are conquered after layer 2. }
  \label{fig:push_through_via_push_moves}
\end{figure}


As in \Cref{eg:rot_via_push} and \Cref{fig:rot}, 
the rotation loop can be realised by first adding a curl at the end of the long knot diagram, 
pushing the long knot diagram through the curl, 
and deleting the curl. 
The preceding algorithm here also applies to the rotation loop. 

\subsubsection{Rolling loop}
\label{algo:Rolling}
We present the algorithm for the reverse of the rolling loop, because it can be adapted directly from the algorithm for the push arc.
The algorithm for the rolling loop itself can then be obtained easily by comparing it with the reversed version.

A single rolling move is given by the algorithm below: 
\begin{enumerate}
    \item Obtain the Seifert circles of the long knot diagram. (See \Cref{fig:Seifert_circles}.)
    \item Add an auxiliary curl near the last crossing. (See \Cref{fig:6_1_FH_curl_preparation}, \Cref{fig:FH_auxiliary_curl}.)
    \item\label{item:push_along_FH} Do push moves against the orientation of the Seifert circles from layer 0 to the maximal layer as in the case of push arcs. (See \Cref{fig:6_1_FH_push}.)
    \item Do R2 moves to reduce the crossings. (See \Cref{fig:6_1_FH_R2R1}.)
    \item Delete the remaining curl. (See \Cref{fig:6_1_FH_R2R1}.)
\end{enumerate}

\begin{definition}\label{def:arrows_from_pushing_FH}
    At Step \ref{item:push_along_FH}, the push moves create many new crossings. 
    We collectively refer to these newly created crossings, together with the crossing being pushed (including those coming from the last crossing and the auxiliary curl), as the \textbf{crossings from pushing}.
    In terms of Gauss diagrams, the corresponding arrows are called the \textbf{arrows from pushing}.
\end{definition}

By the algorithm, if the long knot diagram $K$ has $n$ crossings, 
the whole reversed rolling loop consists of $2n$ rolling moves 
and each rolling move consists of $(n-1)$ crossing-pushing operations. 

\begin{figure}[H]
  \centering
  \begin{subfigure}[t]{0.7\textwidth}
    \centering
    \includegraphics[height=2.5cm]{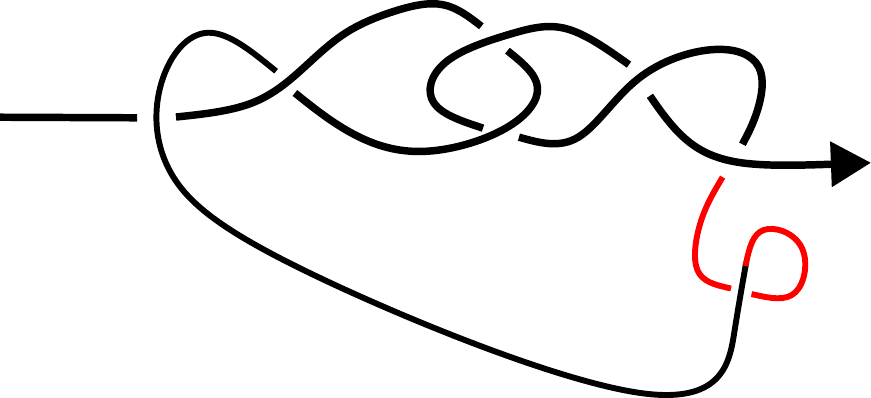}
    \caption{Add an auxiliary curl}
    \label{fig:6_1_FH_curl_preparation}
  \end{subfigure}
  \hfill
  \begin{subfigure}[t]{0.7\textwidth}
    \centering
    \includegraphics[height=3cm]{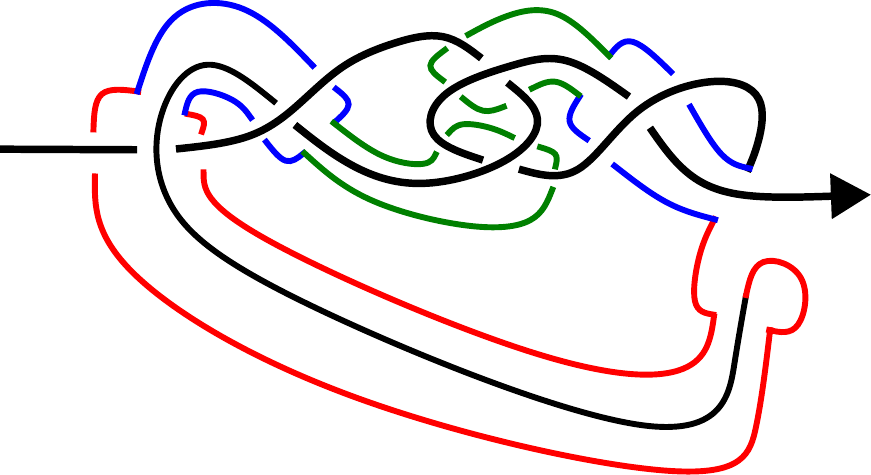}
    \caption{Push moves against the Seifert circles}
    \label{fig:6_1_FH_push}
  \end{subfigure}
  \hfill
  \begin{subfigure}[t]{0.7\textwidth}
    \centering
    \includegraphics[height=4.2cm]{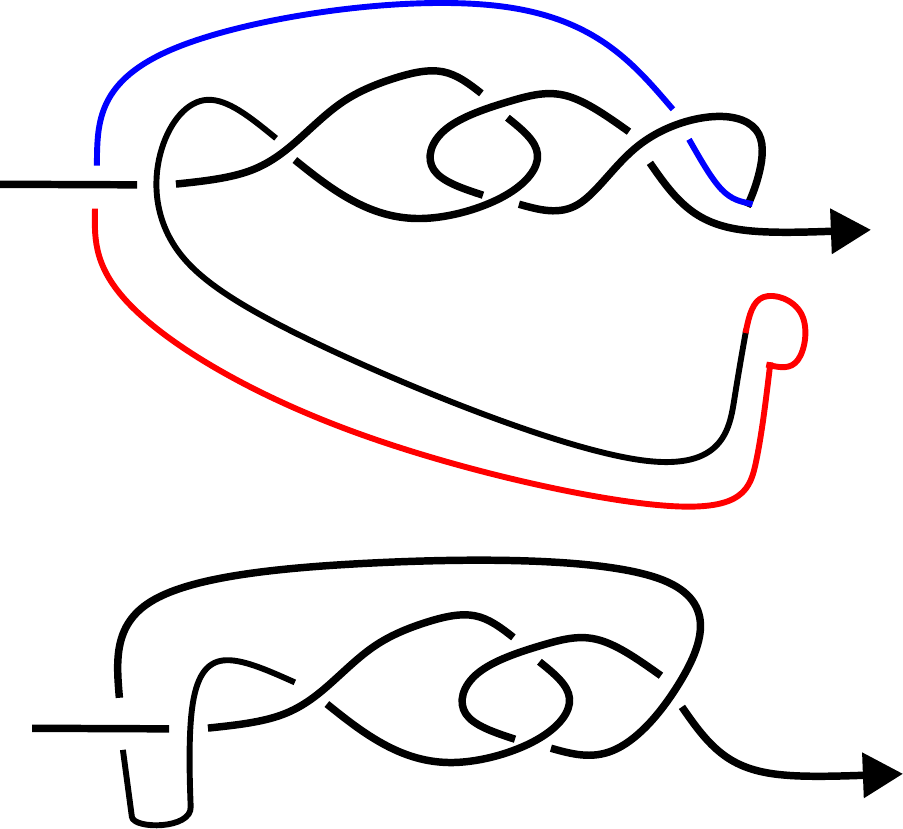}
    \caption{Do R2 and R1 moves to reduce the crossings}
    \label{fig:6_1_FH_R2R1}
  \end{subfigure}
  \hfill

  \caption{Rolling move via push moves. The auxiliary curl is added at the last crossing.}
  \label{fig:rolling_via_push_moves}
\end{figure}




\begin{figure}[H]
    \center  
    \includegraphics[width = 10cm]{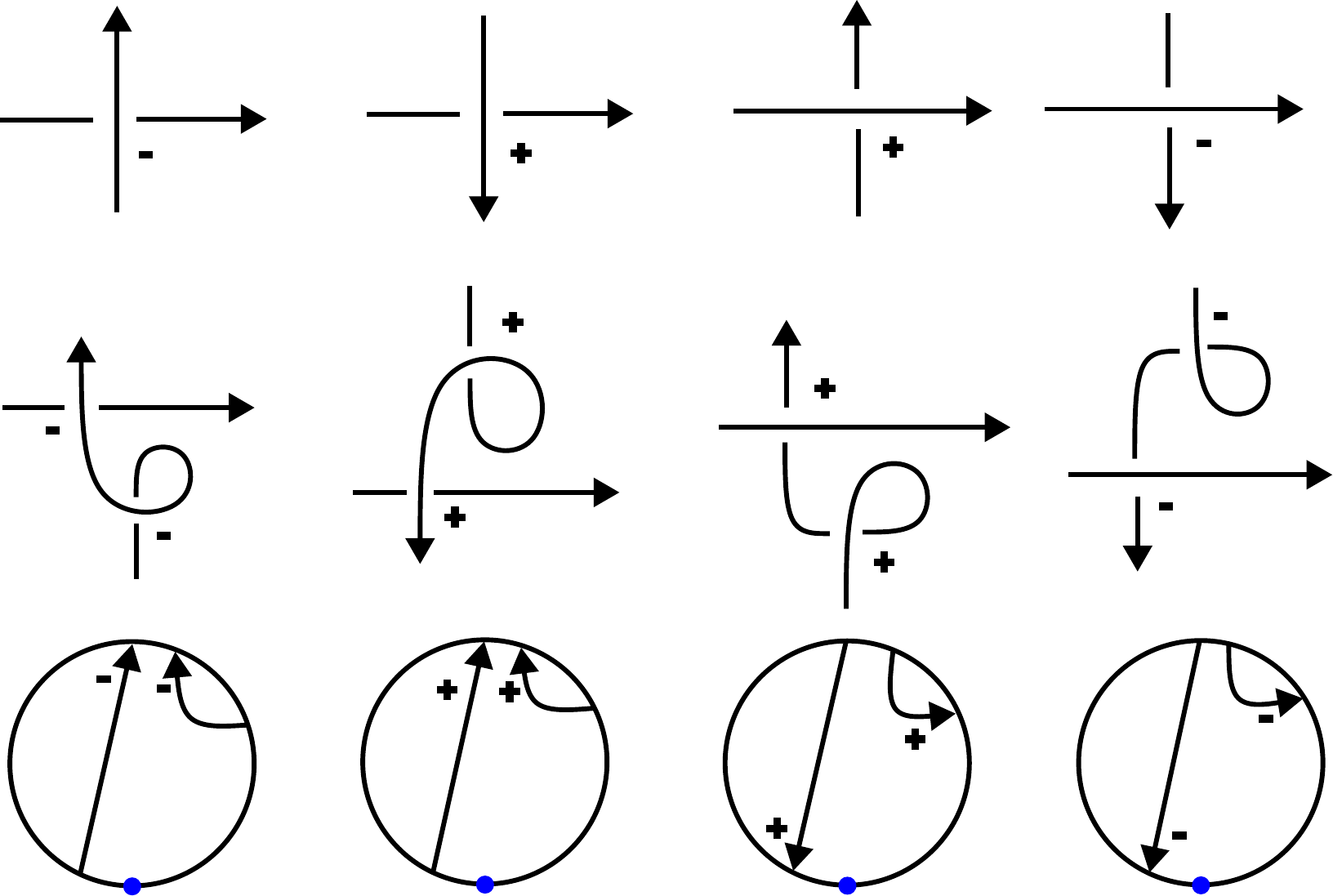}
    \caption{Rules for adding an auxiliary curl near the last crossing. 
    The horizontal strand represents the last part of the long knot.}
    \label{fig:FH_auxiliary_curl}
\end{figure}

\subsection{Pairing with the bracket loop}

\begin{lemma}\label{lem:mixed_order}
    Let $k, l$ be two natural numbers such that $k + l > n$ ($n \ge 2$).
    Let $K$ be a singular long knot diagram with $k$ singular crossings, whose resolutions are denoted by $K_{\epsilon_1\epsilon_2...\epsilon_k}$. 
    Let $G$ be a singular long knot diagram with $l$ singular crossings, whose resolutions are denoted by $G_{\delta_1\delta_2...\delta_l}$. 
    Let $F$ be a signed Gauss diagram with a triangle with at most $(n-2)$ contributing arrows. 
    Then 
    \[
    \sum_{(\epsilon_1, \epsilon_2, ..., \epsilon_k) \in \{\pm1\}^{k}}\sum_{(\delta_1, \delta_2, ..., \delta_l) \in \{\pm1\}^{l}}\epsilon_1\epsilon_2\cdots\epsilon_k \delta_1\delta_2\cdots\delta_l \left\langle F , \push(K_{\epsilon_1\epsilon_2...\epsilon_k}, G_{\delta_1\delta_2...\delta_l}) \right\rangle = 0. 
    \]
\end{lemma}
\begin{proof}
    The proof follows the spirit of \Cref{lem:GDF_finiteness}. 
    \begin{obs*}\label{obs:push_same_projection}
    Regardless of the choice of $(\epsilon_1, \epsilon_2, \ldots, \epsilon_k)$ and $(\delta_1, \delta_2, \ldots, \delta_l)$, 
    the 1-parameter family of the knot projection curves 
    (i.e., the knot diagrams forgetting the over/under information at the crossings) 
    in the push arc described by the algorithm in \Cref{algo:push} is identical. 
    \end{obs*}
    During the algorithm, if the pushed crossing is a resolved one, then all the crossings from pushing are determined by the resolution status of the pushed crossing. 
    Every time there is a contribution of $F$ during an R3 move, there must be at least one singular crossing that is not a contributing arrow and not a crossing from pushing. 
    This is because in the three crossings involved in the R3 move, there are at most 1 resolved crossing from $K$ and 1 resolved crossing from $G$. 
    And there are $(n-2)$ contributing arrows, so $k + l - 2 - (n-2) > 0$. 
    Resolving these singular crossings causes the contributions of $F$ in the sum in the statement to cancel by the algebraic identity 
    \[
    \sum_{(a_1,\ldots,a_{t})\in\{\pm 1\}^{t}}a_1\cdots a_{t} = 0, \; (t \ge 1).
    \]
\end{proof}
From \Cref{lem:mixed_order}, it is not hard to obtain the following theorem. 
\begin{theorem}\label{thm:mixed_order}
    Let $\phi$ be a combinatorial 1-cocycle defined by Gauss diagrams with a triangle of combinatorial order at most $n$ (i.e., at most (n-2) contributing arrows in each configuration, $n \ge 2$) over a field $\mathbb{K}$. 
    Then $\phi([f, g]_0)$ is a finite type invariant of the pair of long knots $(f, g)$ of total Vassiliev order $\le n$. 
    More precisely, let $\{v_{i,a}\}$ be a basis of Vassiliev invariants of order $i$ over $\mathbb{K}$. 
    Then, 
    \[
    \phi([f, g]_0) = \sum_{i + j \le n}\sum_{a,b}c_{i,a,j,b}(\phi)v_{i,a}(f)v_{j, b}(g), 
    \]
    for some coefficients $c_{i,a,j,b}(\phi) \in \mathbb{K}$. 
\end{theorem}
Recall that the rotation loop can be realised by adding a curl, pushing the knot diagram through a curl and deleting the curl.
Using the fact that the crossing of the curl does not come from the resolution of any singular crossing, one can adapt the proof of \Cref{lem:mixed_order} to obtain the following theorem. 
\begin{theorem}\label{thm:rotation_finiteness}
    Let $\phi$ be a combinatorial 1-cocycle defined by Gauss diagrams with a triangle of combinatorial order at most $n$ (i.e., at most (n-2) contributing arrows in each configuration, $n \ge 2$). 
    Then $\phi(\Rot(f))$ is a Vassiliev invariant of order at most $(n-1)$.  
\end{theorem}
\begin{definition}
    A Gauss diagram with a triangle is called \textbf{connected} if cutting at any two different points of the circle of the Gauss diagram, the resulting graph either possesses an arc with no arrows or remains connected as a graph. 
    A combinatorial 1-cocycle defined by Gauss diagrams with triangles is called \textbf{connected} if all of its configurations are connected. 
\end{definition}
Clearly, $\alpha_3^1, \beta_1, \beta_2$ and $\beta_3$ are all connected. 
\begin{proposition}
    Let $\phi$ be a connected combinatorial 1-cocycle defined by Gauss diagrams with triangles of combinatorial order $n$ ($n \ge 2$). 
    Let $K, G$ be two long knot diagrams. Then $\phi(\push(\cdot, G))$ is a finite type invariant of order at most $n$ and $\phi(\push(K, \cdot))$ is a framed finite type invariant of order at most $n$. 
\end{proposition}
\begin{proof}
    The values of a connected 1-cocycle on the vertical paths of the commutative diagram in \Cref{fig:push_commutative} are all the same. 
\end{proof}

For
\[
\phi\in\{\alpha_3^1,\beta_1,\beta_2,\beta_3\},
\]
we have
\[
\phi(\push(K,G))
 =a_{\phi}X_1(K,G)+b_{\phi}X_2(K,G)
 +c_{\phi}X_3(K,G)+d_{\phi}X_4(K,G),
\]
where
\[
\begin{aligned}
X_1(K,G)&=v_2(K)\writhe(G),&
X_2(K,G)&=v_2(K)\binom{\writhe(G)}{2},\\
X_3(K,G)&=v_2(K)v_2(G),&
X_4(K,G)&=v_3(K)\writhe(G).
\end{aligned}
\]

Let $u_{+m}$ denote the trivial long-knot diagram with $m$ positive
curls. The standard diagram of $4_1$ used below has writhe zero.
Direct calculation gives the following table.

\begin{table}[H]
\centering
\renewcommand{\arraystretch}{1.15}
\[
\begin{array}{cc|rrrr|rrrr}
K & G
& X_1 & X_2 & X_3 & X_4
& \alpha_3^1 & \beta_1 & \beta_2 & \beta_3 \bmod 2
\\ \hline
4_1     & u_{+1} & -1 &  0 & 0 & 0 & -1 &  0 &  0 & 0\\
4_1     & u_{+2} & -2 & -1 & 0 & 0 & -2 &  0 &  0 & 0\\
4_1     & 4_1    &  0 &  0 & 1 & 0 &  0 & -1 & -3 & 0\\
3_1^{+} & u_{+1} &  1 &  0 & 0 & 1 &  1 &  0 & -1 & 0
\end{array}
\]
\caption{Values on the push arcs used to determine the coefficients.
The last four columns denote the values of
$\alpha_3^1$, $\beta_1$, $\beta_2$, and $\beta_3$ on
$\push(K,G)$, respectively.}
\label{tab:push-arc-coefficients}
\end{table}

By the proposition and theorems above, together with the calculation results for the 1-cocycles on several example loops obtained by the computer program \cite{ZBT99_25_10}, we can find the values of the pairings of the 1-cocycles. 


\begin{corollary}\label{thm:calculation_pushes}
For any two long knot diagrams $K$ and $G$, we have
\[
\begin{aligned}
\alpha_3^1(\push(K,G))
    &=\writhe(G)v_2(K),\\
\beta_1(\push(K,G))
    &=-v_2(K)v_2(G),\\
\beta_2(\push(K,G))
    &=-3v_2(K)v_2(G)-\writhe(G)v_3(K),\\
\beta_3(\push(K,G))
    &=0 \pmod 2.
\end{aligned}
\]
\end{corollary}

\begin{corollary}
For any two long knots $f$ and $g$, we have
\[
\alpha_3^1([f,g]_0)
=
\alpha_3^1\left(\tfrac12[f,f]_0\right)
=0.
\]
Moreover,
\[
\renewcommand{\arraystretch}{1.3}
\begin{array}{c|ccc}
 & \Rot(f) & [f,g]_0 & \tfrac12[f,f]_0 \\ \hline
\beta_1
 & 0
 & -2v_2(f)v_2(g)
 & -v_2(f)^2 \\

\beta_2
 & -v_3(f)
 & -6v_2(f)v_2(g)
 & -3v_2(f)^2 \\

\beta_3
 & 0
 & 0
 & 0
\end{array}
\]
where the last row is understood over $\mathbb{Z}/2\mathbb{Z}$.
\end{corollary}


\subsection{Pairing with the rolling loop}


\begin{lemma}\label{lem:FH_finiteness}
    Let $K$ be a singular long knot diagram with $(n + 1)$ singular crossings, whose resolutions are denoted by $K_{\epsilon_1\epsilon_2...\epsilon_{n+1}}$ ($n \ge 2$).
    Let $G$ be a signed Gauss diagram with a triangle and $(n - 2)$ contributing arrows. 
    Then 
    \[
    \sum_{(\epsilon_1, \epsilon_2, ..., \epsilon_{n+1}) \in \{\pm1\}^{n+1}}\epsilon_1\epsilon_2\cdots\epsilon_{n+1}
    \left\langle G, \Roll^{-1}(K_{\epsilon_1\epsilon_2...\epsilon_{n+1}}) \right\rangle = 0. 
    \]
\end{lemma}
\begin{proof}
    The proof is similar to that of \Cref{lem:mixed_order}. 
    \begin{obs*}\label{obs:FH_same_projection}
    Regardless of the choice of $(\epsilon_1, \epsilon_2, \ldots, \epsilon_{n+1})$, 
    the 1-parameter family of the knot projection curves 
    (i.e., the knot diagrams forgetting the over/under information at the crossings) 
    in the reversed rolling loop described by the algorithm in \Cref{algo:Rolling} is identical.
    \end{obs*}

    During the algorithm process, if the last crossing is a resolved one, then all the crossings from pushing are determined by the resolution status of the last crossing. 
    Every time there is a contribution of $F$ at some R3 move, there must be at least one singular crossing that is not a contributing arrow, not a crossing from pushing and not an R3 move crossing. 
    This is because in the 3 concerning crossings in the R3 move, there are at most 2 coming from the resolved crossings. 
    And there are $(n-2)$ contributing arrows, so $(n+1) - 2 - (n-2) > 0$. 
    Resolving these singular crossings causes the contributions of $F$ in the sum in the conclusion to cancel by the algebraic identity 
    \[
    \sum_{(a_1,\ldots,a_{t})\in\{\pm 1\}^{t}}a_1\cdots a_{t} = 0, \; (t \ge 1).
    \]

\end{proof}

\begin{theorem}\label{thm:FH_finiteness}
    Let $\phi$ be a combinatorial 1-cocycle defined by Gauss diagrams with a triangle of combinatorial order at most $n$ (i.e., at most $(n - 2)$ contributing arrows in each configuration, $n \ge 2$). 
    Let $f$ be a long knot. 
    Then \[
    \phi(\Roll_m(f)) = \phi(\Roll_0(f)) - m\cdot\phi(\Rot(f)),
    \]
    where $\phi(\Roll_0(f))$ is a finite type invariant of order at most $n$ for long knots. 
\end{theorem}
\begin{proof}
    We aim to prove the finite type property of $\phi(\Roll_0(f))$.
    Let $K$ be a singular long knot diagram with $(n + 1)$ singular crossings, whose resolutions are denoted by $K_{\epsilon_1\epsilon_2...\epsilon_{n+1}}$ ($n \ge 2$). 
    By \Cref{prop:roll_framing}, we have 
    \[
    \phi(\Roll_0(K_{\epsilon_1\epsilon_2...\epsilon_{n+1}})) = \phi(\Roll_{m_{\epsilon_1\epsilon_2...\epsilon_{n+1}}}(K_{\epsilon_1\epsilon_2...\epsilon_{n+1}})) + m_{\epsilon_1\epsilon_2...\epsilon_{n+1}}\phi(\Rot(K_{\epsilon_1\epsilon_2...\epsilon_{n+1}})), 
    \]
    where we choose $m_{\epsilon_1\epsilon_2...\epsilon_{n+1}} = \writhe(K_{\epsilon_1\epsilon_2...\epsilon_{n+1}})$. 
    By the blackboard framing convention, 
    \[
    \Roll_{m_{\epsilon_1\epsilon_2...\epsilon_{n+1}}}(K_{\epsilon_1\epsilon_2...\epsilon_{n+1}}) = \Roll(K_{\epsilon_1\epsilon_2...\epsilon_{n+1}}). 
    \]
    Denote $w_0 = \writhe(K_{-1, -1, \cdots, -1})$. Then $m_{\epsilon_1\epsilon_2...\epsilon_{n+1}} = w_0 + (1 + \epsilon_1 + \cdots + 1 + \epsilon_{n+1}) = w_0 + n+1 + (\epsilon_1 + \cdots \epsilon_{n+1})$. 
    Therefore, by \Cref{thm:rotation_finiteness} and \Cref{lem:FH_finiteness}, we have
    \begin{align*}
        &\sum_{(\epsilon_1,\epsilon_2,\ldots,\epsilon_{n+1})\in\{\pm1\}^{n+1}}
        \epsilon_1\epsilon_2\cdots\epsilon_{n+1}
        \phi\!\left(
        \Roll_0\!\left(K_{\epsilon_1\epsilon_2\cdots\epsilon_{n+1}}\right)
        \right)
        \\
        ={}&
        \sum_{(\epsilon_1,\epsilon_2,\ldots,\epsilon_{n+1})\in\{\pm1\}^{n+1}}
        \epsilon_1\epsilon_2\cdots\epsilon_{n+1}
        \phi\!\left(
        \Roll\!\left(K_{\epsilon_1\epsilon_2\cdots\epsilon_{n+1}}\right)
        \right)
        \\
        &\quad+
        \sum_{(\epsilon_1,\epsilon_2,\ldots,\epsilon_{n+1})\in\{\pm1\}^{n+1}}
        \epsilon_1\epsilon_2\cdots\epsilon_{n+1}
        \bigl(w_0+n+1 +(\epsilon_1+\cdots+\epsilon_{n+1})\bigr) \cdot\phi\!\left(
        \Rot\!\left(K_{\epsilon_1\epsilon_2\cdots\epsilon_{n+1}}\right)
        \right)
        \\
        ={}& 0+0=0.
    \end{align*}
\end{proof}

By \Cref{thm:FH_finiteness}, we can calculate our 1-cocycles on several examples to obtain the general formulae. 
Let $v_2, v_3, v_{4,1}, v_{4,2}$ be the primitive Vassiliev invariants of order 2, 3, 4, 4 (see \Cref{sec:GDF}). 
Their values and the corresponding cocycle values are as follows \cite{ZBT99_25_10}. 
\[
\begin{array}{c|rrrrrrrrr}
    & \beta_1(\Roll_0) & \beta_2(\Roll_0) & \beta_3(\Roll_0) \bmod 2 & v_2 & v_3 & v_{4,1} & v_{4,2} & \frac{v_2(v_2 - 1)}{2} \\ \hline
    3_1^{+} & 0 & -6 & 0 & 1 & 1 & 1 & 0  & 0 \\
    4_1 & 2 & 2 & 0 & -1 & 0 & 0 & 0  & 1 \\
    5_1 & 0 & -50 & 0 & 3 & 5 & 5 & 3  & 3 \\
    5_2 & 0 & -26 & 0 & 2 & 3 & 3 & 1  & 1 \\
    6_1 & 8 & 10 & 0 & -2 & -1 & 0 & -1  & 3\\
\end{array}
\]
Notice that over $\mathbb{Q}$, $(1, v_2, v_3, v_{4,1}, v_{4,2}, \binom{v_2}{2})$ form a basis of Vassiliev invariants up to order 4.
And over $\mathbb{Z}/2\mathbb{Z}$, $(1, v_2, v_3, v_{4,1}, v_{4,2}, \binom{v_2}{2}) \mod 2$ form a basis as well. ($\binom{v_2}{2}$ is inspired from the work \cite{Stanford02}.) 
Therefore, we have 
\begin{theorem}\label{thm:calculation_FH}
    Given a long knot $f$, the values of the 1-cocycles $\beta_1$, $\beta_2$ and $\beta_3$
    on the rolling loop with 0-framing are as follows:
    \[
    \beta_1(\Roll_0(f)) = 2[\binom{v_2(f)}{2}-v_{4,2}(f)]
    \]
    \[
    \beta_2(\Roll_0(f)) = 2[3 \cdot \binom{v_2(f)}{2}  + 2v_2(f) + 5v_3(f)-10v_{4,1}(f)-5v_{4,2}(f)]
    \]
    \[
    \beta_3(\Roll_0(f)) = 0 \text{ over } \mathbb{Z}/2\mathbb{Z}. 
    \]
\end{theorem}
Recall that for a torus knot, 
the rolling loop and the Gramain cycle are collinear. 
Combining the results in \Cref{thm:calculation_pushes} and \Cref{thm:calculation_FH}, 
we have the following result, 
which is included in \cite{AlvarezLabastida96} 
if the primitive Vassiliev invariants of order 4 are well chosen. 
\begin{corollary}
    If $f$ is a torus knot, 
    then 
    \[
    \binom{v_2(f)}{2}-v_{4,2}(f) = 0
    \]
\end{corollary}

Since our 1-cocycles are all connected, their values on the loops of each factor of a connected sum knot are the same as those on the loop of the factor knot. 
The Table \ref{tab:independence_1cocycles} demonstrates linear independence of the 1-cocycles on the component of $4_1 \csum 6_1$. 
\begin{table}[h!]
\centering
\[
\begin{array}{c|cccc}
    & \alpha_{3}^{1} & \beta_1 & \beta_2 & \beta_3 \\ \hline
    \Rot(4_1) & 1 & 0 & 0 & 0 \\ [4pt]
    \frac{1}{2}\Roll(4_1) & 0 & 1 & 1 & 0 \\ [4pt]
    \Rot(6_1) & 0 & 0 & 1 & 0 \\ [4pt]
    \frac{1}{2}\Roll(6_1) & 1 & 0 & 1 & 1 
\end{array}
\]
\caption{Values of the 1-cocycles on the half rolling loops and rotation loops over $\mathbb{Z}/2\mathbb{Z}$. }
\label{tab:independence_1cocycles}
\end{table}
\begin{remark}
    The pairings with the push arc and the rotation loop have been computed differently by diagrammatic method in \cite[Chapter 7]{Zhang25}. 
\end{remark}

\section{Parametrized and unparametrized closed knot spaces}
\label{sec:closed_knot_space}
In this section, we clarify the relationships among the loops in the space of long knots $\LongKnotSpace$, in the \textbf{parametrized closed knot space} $\Emb(S^1, S^3)$ and in the \textbf{oriented unparametrized closed knot space} $\Emb(S^1, S^3)/\Diff^+(S^1)$. 
Here $\Diff^+(S^1)$ is the group of orientation-preserving diffeomorphisms of $S^1$ and the quotient space $\Emb(S^1, S^3)/\Diff^+(S^1)$ is defined by the equivalence relation $f \sim g$ if and only if there exists $\phi \in \Diff^+(S^1)$ such that $f = g \circ \phi$.

We will give a criterion for a 1-cohomology class in $H^1(\LongKnotSpace; R)$ to descend to a class in $H^1(\Emb(S^1, S^3); R)$ or $H^1(\Emb(S^1, S^3)/\Diff^+(S^1); R)$, where $R$ is a commutative ring.

\subsection{Settings}
Let $S^1 = \{z\in \mathbb{C} | \; |z| = 1\}$ and $S^3 = \{(a,b,c,d)\in \mathbb{R}^4 | \; a^2 + b^2 + c^2 + d^2 = 1\}$ be the standard 1-sphere and 3-sphere, respectively. 
We identify $S^1$ with $\mathbb{R} \cup \{\infty\}$ via the stereographic projection $\sigma_1: S^1 \setminus \{1\} \to \mathbb{R}$ defined by $z\mapsto \frac{-\Im(z)}{1-\Re(z)}$, which extends to a diffeomorphism $ S^1 \to \mathbb{R} \cup \{\infty\}$.
We identify $S^3$ with $\mathbb{R}^3 \cup \{\infty\}$ via the stereographic projection $\sigma_3: S^3 \setminus \{(0,0,0,1)\} \to \mathbb{R}^3$ defined by $(a,b,c,d) \mapsto (\frac{a}{1-d}, \frac{b}{1-d}, \frac{c}{1-d})$, which extends to a diffeomorphism $ S^3 \to \mathbb{R}^3 \cup \{\infty\}$.
Via the subspace embedding $S^3 \subset \mathbb{R}^4$, we have $TS^3 = \{(x, v)\in S^3 \times \mathbb{R}^4 \; | \; <x, v> = 0\} \subset \mathbb{R}^4 \times \mathbb{R}^4$.
Define 
\[
\Emb_{*}(S^1, S^3) \coloneqq \left\{f\in \Emb(S^1, S^3) \; \middle|\; f(1) = (0,0,0,1)\right\}, 
\]

Given a long knot $f\in \LongKnotSpace$, we can extend it to a smooth embedding $\tilde{f}: S^1 \to S^3$ by defining \[
\widetilde f(z)=
\begin{cases}
\sigma_3^{-1}\bigl(f(\sigma_1(z))\bigr),&z\neq1,\\
(0, 0, 0, 1),&z=1,
\end{cases}
\]
which gives the compactification map 
\[
c:\LongKnotSpace\longrightarrow \Emb(S^1,S^3),
\qquad
f\longmapsto \widetilde f. 
\]
Let 
\[
q: \Emb(S^1, S^3) \to \Emb(S^1, S^3)/\Diff^+(S^1)
\]
be the quotient map.
From the sequence of maps: 
\[\LongKnotSpace \xrightarrow{c} \Emb(S^1, S^3) \xrightarrow{q} \Emb(S^1, S^3)/\Diff^+(S^1),\]
we have the following sequence of fundamental groups and cohomology groups:
\[\pi_1(\LongKnotSpace, f) \xrightarrow{c_*} \pi_1(\Emb(S^1, S^3), \tilde{f}) \xrightarrow{q_*} \pi_1(\Emb(S^1, S^3)/\Diff^+(S^1), \bar{f}),\]
\[H^1(\LongKnotSpace; R) \leftarrow H^1(\Emb(S^1, S^3); R) \leftarrow H^1(\Emb(S^1  , S^3)/\Diff^+(S^1); R),\]
where $\bar{f}$ is the image of $\tilde{f}$ in the quotient space $\Emb(S^1, S^3)/\Diff^+(S^1)$, and $R$ is a commutative ring.  

\subsection{Loops in the closed knot space}
First we define the corresponding versions of the rotation loop and the rolling loop in $\Emb(S^1, S^3)$.
\begin{definition}
    Let $f\in \Emb(S^1, S^3)$ be a closed knot, $T(z) = \frac{(df)_z(iz)}{\|(df)_z(iz)\|}$ be the tangent vector of $f$ at $z \in S^1$, and $z \mapsto N(z) \in T_{f(z)}S^3$ be a smooth unit normal vector field along $f$ such that $N(z) \perp T(z)$ and $N(z) \perp f(z) $ for all $z\in S^1$. 
    Let $B(z)$ be a unit vector in $\mathbb{R}^4$ uniquely determined by $T(z), N(z), f(z)$ such that $A(z) \coloneqq [T(z), N(z), B(z), f(z)]\in SO(4)$. 
    \begin{enumerate}
        \item The \textbf{rotation loop} $\Rot(f)$ is the loop in $\Emb(S^1, S^3)$ defined by 
        \[\Rot(f)(t)(z) = R_{2\pi t}f(z), \quad t\in [0,1], z\in S^1,\]
        where
        \[
        R_{2\pi t}
        =
        \begin{pmatrix}
        1 & 0 & 0 & 0 \\
        0 & \cos(2\pi t) & -\sin(2\pi t) & 0 \\
        0 & \sin(2\pi t) & \cos(2\pi t) & 0 \\
        0 & 0 & 0 & 1
        \end{pmatrix}
        \in SO(4).
        \]
        \item The \textbf{rolling loop} $\Roll_N(f)$ of $f$ along $N$ is the loop in $\Emb(S^1, S^3)$ defined by
        \[\Roll_N(f)(t)(z) \coloneqq A(1)A(e^{2\pi it})^{-1}f(e^{2\pi it}z), \quad t\in [0,1], z\in S^1.\]
        If $N$ has the \textbf{framing number} $n$, i.e., the linking number $\lk(f, \frac{f + \epsilon N}{\sqrt{1 + \epsilon^2}}) = n$ for sufficiently small $\epsilon > 0$, then the homotopy class of $\Roll_N(f)$ depends only on $n$. 
        We denote this class by
        \[
        [\Roll_n(f)]\in\pi_1(\Emb(S^1,S^3),f).
        \]
        \item The \textbf{reparametrization loop} $\Repar(f)$ of $f$ is the loop in $\Emb(S^1, S^3)$ defined by
        \[\Repar(f)(t)(z) = f(e^{2\pi it}z), \quad t\in [0,1], z\in S^1.\]
    \end{enumerate}
\end{definition}
\begin{remark}
    If $R(\cdot)\in SO(4)$ is a loop based at $I_4$ representing the generator in $\pi_1(SO(4)) = \mathbb{Z}/2\mathbb{Z}$, then $[\Rot(f)] = [R(\cdot)f] \in \pi_1(\Emb(S^1, S^3), f)$. 
\end{remark}



\begin{proposition}
    Let $f\in \LongKnotSpace$ be a long knot and $\tilde{f} \in \Emb(S^1, S^3)$ be the corresponding closed knot. 
    Then we have the following relations in $\pi_1(\Emb(S^1, S^3), \tilde{f})$:
    \[\Rot(\tilde{f}) = c(\Rot(f)), \quad [\Roll_n(\tilde{f})] = c_*([\Roll_n(f)]).\]
\end{proposition}

\subsection{Relations among the fundamental groups}
The purpose of this subsection is to state some direct consequences of the literature \cite{Gramain77,BudneyCohen05,BinzFischer81} in the form needed for the descending criteria later. 
The relation of fundamental groups of $\LongKnotSpace$ and $\Emb(S^1, S^3)$ goes back to the work of Gramain \cite[Lemma 2]{Gramain77}. 
Here we use the associated bundle description of Budney and Cohen \cite[Corollary~4.2]{BudneyCohen05} to display it. 

Set
\[
p_{\infty}=(0,0,0,1)\in S^3, \quad u_{\infty}=(-1,0,0,0)\in S(T_{p_{\infty}}S^3) = S^2,
\]
where we identify the unit sphere in \(T_{p_{\infty}}S^3\) with \(S^2\).
Let
\[
\nu\colon \Emb_*(S^1,S^3)\longrightarrow S^2,
\qquad
\nu(g)=\frac{(dg)_1(i)}{\lVert(dg)_1(i)\rVert},
\]
where \(i\in T_1S^1\) is the positively oriented unit tangent vector.

\begin{lemma}[{\cite[Corollary~4.2]{BudneyCohen05}}]\label{lem:long-closed-bundle}
There is a homeomorphism
\[
\Emb(S^1,S^3)
\cong
S^3\times\Emb_*(S^1,S^3).
\]
Moreover, the bundle
\[
\LongKnotSpace
\longrightarrow
\Emb_*(S^1,S^3)
\xrightarrow{\;\nu\;}
S^2
\]
is fiberwise homotopy equivalent to the associated bundle
\[
\LongKnotSpace
\longrightarrow
SO(3)\times_{SO(2)}\LongKnotSpace
\xrightarrow{\;[A,g]\mapsto A u_{\infty}\;}
S^2.
\]
where \(SO(2)\subset SO(3)\) is the stabilizer of \(u_{\infty}\) and acts on \(\LongKnotSpace\) by rotation about the long axis.  
In the standard trivializations over the two hemispheres, this
associated bundle is obtained by gluing two copies of
\(D^2\times\LongKnotSpace\) along their common boundary by
\[
S^1\times\LongKnotSpace
\longrightarrow
S^1\times\LongKnotSpace,
\qquad
(z,g)\longmapsto(z,z^2\cdot g),
\]
where \(S^1\) is identified with \(SO(2)\).  
\end{lemma}

\begin{proposition}\label{prop:compactification-pi1}
Let \(f\in\LongKnotSpace\), and let
\(\widetilde f\in\Emb(S^1,S^3)\) be its compactification.
Then there is an exact sequence
\[
\mathbb Z
\xrightarrow{\;\partial_f\;}
\pi_1(\LongKnotSpace,f)
\xrightarrow{\;c_*\;}
\pi_1\bigl(\Emb(S^1,S^3),\widetilde f\bigr)
\longrightarrow 1,
\]
where the generator of
\(\mathbb Z\cong\pi_2(S^2,u_{\infty})\) may be chosen so that
\[
\partial_f(1)=[\Rot(f)]^2.
\]
Consequently, the cyclic subgroup generated by
\([\Rot(f)]^2\) is normal, and
\[
\pi_1\bigl(\Emb(S^1,S^3),\widetilde f\bigr)
\cong
\pi_1(\LongKnotSpace,f)\big/
\bigl\langle[\Rot(f)]^2\bigr\rangle.
\]
\end{proposition}

\begin{proof}
Since \(\nu(\widetilde g)=u_{\infty}\) for every \(g\in\LongKnotSpace\), the compactification map corresponds, under the bundle equivalence of Lemma~\ref{lem:long-closed-bundle}, to the inclusion of the fiber over \(u_{\infty}\).
By \Cref{lem:long-closed-bundle}, we also have 
\[
\pi_1(\Emb_*(S^1, S^3), \tilde{f}) \cong \pi_1(\Emb(S^1, S^3), \tilde{f}). 
\]
The homotopy exact sequence therefore gives
\[
\pi_2(S^2,u_{\infty})
\xrightarrow{\;\partial_f\;}
\pi_1(\LongKnotSpace,f)
\xrightarrow{\;c_*\;}
\pi_1\bigl(\Emb(S^1,S^3),\widetilde f\bigr)
\longrightarrow
\pi_1(S^2,u_{\infty})=0.
\]

By the gluing description in Lemma~\ref{lem:long-closed-bundle}, the
boundary of a suitably oriented generator
of \(\pi_2(S^2,u_{\infty})\) is represented by
\[
S^1\longrightarrow\LongKnotSpace,
\qquad
z\longmapsto z^2\cdot f.
\]
This is the rotation loop of \(f\) traversed twice, and hence
\[
\partial_f(1)=[\Rot(f)]^2.
\]
The quotient description follows from exactness.
\end{proof}

Another ingredient needed is that 
\[
\Emb(S^1, S^3) \to \Emb(S^1, S^3)/\Diff^{+}(S^1)
\] 
is a $\Diff^{+}(S^1)$-bundle \cite{BinzFischer81}.
Therefore, we have a homotopy fiber sequence 
\[
\Diff^{+}(S^1) \to \Emb(S^1, S^3) \to \Emb(S^1, S^3)/\Diff^+(S^1),
\]
inducing the exact sequence of fundamental groups
\[
\pi_1(\Diff^{+}(S^1)) \to \pi_1(\Emb(S^1, S^3), f) \xrightarrow{q_{*}} \pi_1(\Emb(S^1, S^3)/\Diff^+(S^1), \bar{f}) \to 1,
\]
where $\bar{f}$ is the image of $f$ in the quotient space $\Emb(S^1, S^3)/\Diff^+(S^1)$.
So $q_{*}$ is surjective. 
Notice that $\Diff^{+}(S^1) \simeq SO(2) \cong S^1$.
The image of the first map $\pi_1(\Diff^{+}(S^1)) \to \pi_1(\Emb(S^1, S^3), f)$ is generated by the reparametrization loop $\Repar(f)$.
Thus, we have the following proposition.
\begin{proposition}
    Let $f \in \Emb(S^1, S^3)$ be a closed knot, and let $\bar{f}$ denote its image in $\Emb(S^1, S^3)/\Diff^+(S^1)$. Then we have 
    \[\pi_1(\Emb(S^1, S^3)/\Diff^+(S^1), \bar{f}) \cong \pi_1(\Emb(S^1, S^3), f)/ \langle [\Repar(f)] \rangle .\]
\end{proposition} 

\subsection{Descending criteria}
\begin{theorem}\label{thm:repar}
    Let $f \in \Emb(S^1, S^3)$ be a closed knot. Then in the fundamental group $\pi_1(\Emb(S^1, S^3), f)$, we have
    \[[\Repar(f)] = [\Roll_0(f)]*[\Rot(f)] = [\Rot(f)]*[\Roll_0(f)]. \]
\end{theorem}

\begin{proof}
    By definition, we have 
    \[
    \Repar(f)(t)(z) = A(e^{2\pi it})A(1)^{-1}\Roll_0(f)(t)(z). 
    \]
    Consider the action 
    \begin{align*}
        \mu \colon SO(4) \times \Emb(S^1, S^3) &\longrightarrow \Emb(S^1, S^3) \\
        (R, g) &\longmapsto Rg,
    \end{align*}
    which induces maps on the fundamental groups
    \[ 
    \pi_1(SO(4), I_4) \times \pi_1(\Emb(S^1, S^3), f) \xrightarrow[\cong]{\eta} \pi_1(SO(4) \times \Emb(S^1, S^3), (I_4, f)) \xrightarrow{\mu_*} \pi_1(\Emb(S^1, S^3), f).
    \]
    \[
    \mu_{*}^{f}: \pi_1(SO(4), I_4) \longrightarrow \pi_1(\Emb(S^1, S^3), f), \quad [R(\cdot)] \longmapsto [\mu(R(\cdot), f)].
    \]
    Let $Q(t) \coloneqq A(e^{2\pi it})A(1)^{-1} \; (t\in [0,1])$ be the loop in $SO(4)$ based at $I_4$. 
    Since 
    \[
    \Repar(f)(\cdot) = \mu(Q(\cdot), \Roll_0(f)(\cdot)),
    \]
    and 
    \begin{align*}
        [(Q(\cdot),\Roll_0(f)(\cdot))] &= \eta(([Q], Const_f) * (Const_{I_4}, [\Roll_0(f)])) \\
        &= \eta((Const_{I_4}, [\Roll_0(f)]) * ([Q], Const_f)), 
    \end{align*}
    we obtain
    \[
    [\Repar(f)] = \mu_*^f([Q])*[\Roll_0(f)] = [\Roll_0(f)] * \mu_*^f([Q]), 
    \]
    where $\mu_*^f([Q])$ is either trivial or $[\Rot(f)]$ in $\pi_1(\Emb(S^1, S^3), f)$ depending on whether $[Q]$ is trivial or nontrivial in $\pi_1(SO(4), I_4) \cong \mathbb{Z}_2$.

    Let $\Sigma$ be a genus $g$ Seifert surface of $f$. 
    Since $N$ is zero-framed and $[\Roll_0]$ depends only on the framing number, we may assume that $N(\cdot) \in T\Sigma$ points inward along the boundary of $\Sigma$. 
    Let $\iota: \Sigma \to \mathbb{R}^4$ be the inclusion vector field.
    Choose a smooth vector field $X$ on $\Sigma$ possessing only non-degenerate zeros $\{p_1, \ldots, p_m\}$ and a unit normal vector field $\nu(\cdot) \in TS^3$ on $\Sigma$ such that 
    \[
    X(f(z)) = -N(z), \nu(f(z)) = B(z), \quad \forall z\in S^1. 
    \]
    Let $JX$ be the smooth vector field on $\Sigma$ obtained by rotating $X$ by $\frac{\pi}{2}$ such that $JX(f(z)) = T(z), \; (\forall z\in S^1)$. 
    Let $F(q) \coloneqq [X(q), JX(q), \nu(q), \iota(q)], \; q \in \Sigma$.
    Then we have 
    \begin{align*}
        A(z) &= [T(z), N(z), B(z), f(z)]\\
        &= F(f(z)) \begin{pmatrix}
        0 & -1 & 0 & 0 \\
        1 & 0 & 0 & 0 \\
        0 & 0 & 1 & 0 \\
        0 & 0 & 0 & 1
        \end{pmatrix}. \\
    \end{align*}
    Since $\pi_1(SO(4), I_4) = \mathbb{Z}/2\mathbb{Z}$ is abelian, we will no longer consider the base point for simplicity.  
    Let $C_0$ be the image of $f$. 
    Because $Q(z) = F(f(z))F(f(1))^{-1}$, we have $[Q(\cdot)] = [F|_{C_0})] \in \pi_1(SO(4))$. 
    Let $D_i$ be a small closed disk neighborhood of $p_i$ in $\Sigma$ such that $D_i \cap D_j = \emptyset$ for $i\neq j$. 
    Let $C_i$ be the boundary of $D_i$.
    On $\Sigma - \cup_i \{p_i\}$, we define 
    \[
    \tilde{F}(q) \coloneqq \left[\dfrac{X(q)}{\|X(q)\|}, \dfrac{JX(q)}{\|JX(q)\|}, \nu(q), \iota(q)\right] \in SO(4), \; q\in \Sigma - \cup_i \{p_i\}.
    \]
    Then $\tilde{F} = F $ on $C_0$. 
    Since 
    \[
    \pi_1(\Sigma-\cup_i D_i) = \langle c_0, c_1, \cdots, c_m, a_1, b_1, \cdots, a_g, b_g \mid [a_1,b_1] \cdots [a_g,b_g] c_0 c_1 \cdots c_m = 1 \rangle, 
    \]
    we have 
    \[
    [\title{F}|_{C_0}] + \sum_{i>0} [\title{F}|_{C_i}] + \sum_{j>0}\title{F}_*([a_j,b_j]) = [\title{F}|_{C_0}] + \sum_{i>0} [\title{F}|_{C_i}] = 0 \in \pi_1(SO(4)) \cong \mathbb{Z}_2.
    \]
    For each $D_i$, choose a smooth orthogonal unit frame $\{u_i^1(q), u_i^2(q)\} \subset T_qD_i, \; q \in D_i$ such that 
    \[
    \tilde{F}(q) = (u_i^1(q), u_i^2(q), \nu(q), \iota(q))M_i(q), \; q\in D_i - \{p_i\}, 
    \]
    where 
    \begin{align*}
        M_i(q) &= (u_i^1(q), u_i^2(q), \nu(q), \iota(q))^T(\dfrac{X(q)}{\|X(q)\|}, \dfrac{JX(q)}{\|JX(q)\|}, \nu(q), \iota(q)) \\
        &= \begin{pmatrix}
        G_i(q) & 0 \\
        0 & I_2 \\
        \end{pmatrix} \in SO(4),
    \end{align*}
    and 
    \[
    G_i(q) = (u_i^1(q), u_i^2(q))^T(\dfrac{X(q)}{\|X(q)\|}, \dfrac{JX(q)}{\|JX(q)\|}) \in SO(2). 
    \]
    Notice that 
    \begin{align*}
    T_qD_i &\longrightarrow \mathbb{R}^2 \\
    u &\mapsto (u_i^1(q), u_i^2(q))^Tu,
    \end{align*}
    gives a trivialization of $TD_i$ with respect to the orientation. 
    Define 
    \[
    g_i(q) \coloneqq (u_i^1(q), u_i^2(q))^T\frac{X(q)}{\|X(q)\|} \in S^1\subset \mathbb{R}^2, \; q\in C_i. 
    \]
    By the definition of the index of a vector field, we have 
    \[
    ind_{p_i}(X) = deg(g_i : C_i \to S^1). 
    \]
    Define 
    \begin{align*}
        \theta : S^1 &\longrightarrow SO(2) \\
        y &\mapsto \left[y, \begin{pmatrix}
                        0 & -1 \\
                        1 & 0
                    \end{pmatrix}y\right]. 
    \end{align*}
    Then $\theta$ is a diffeomorphism and $G_i = \theta \circ g_i$, which induces 
    \[
    (G_i)_*: H_1(C_i) \xrightarrow{(g_i)_*} H_1(S^1) \xrightarrow[\cong]{\theta_*} H_1(SO(2)) = \pi_1(SO(2)). 
    \]
    Therefore, $[G_i] = ind_{p_i}(X) \in \pi_1(SO(2)) = \mathbb{Z}.$ 
    The inclusion 
    \[
    SO(2) \hookrightarrow SO(4), R \mapsto \begin{pmatrix}
                                                R & 0 \\
                                                0 & I_2
                                            \end{pmatrix}
    \]
    induces 
    \begin{align*}
        \pi_1(SO(2)) \cong \mathbb{Z} &\longrightarrow \pi_1(SO(4)) \cong \mathbb{Z}/2\mathbb{Z} \\
        n &\mapsto n \mod 2. 
    \end{align*}
    We obtain 
    \[
    [M_i] = ind_{p_i}(X) \mod 2 \in \pi_1(SO(4)) = \mathbb{Z}/2\mathbb{Z}.
    \]
    Since $q \mapsto (u_i^1(q), u_i^2(q), \nu(q), \iota(q))\in SO(4)$ is defined on the disk $D_i$,
    its restriction to $C_i$ gives a trivial element in $\pi_1(SO(4))$. 
    Hence, we have 
    \[
    [\tilde{F}|_{C_i}] = ind_{p_i}(X) \mod 2 \in \pi_1(SO(4)) = \mathbb{Z}/2\mathbb{Z},\; (i >0). 
    \]
    So 
    \[
    [F|_{C_0}] = [\tilde{F}|_{C_0}] = \sum_{i>0}[\tilde{F}|_{C_i}] = \sum_{i>0}ind_{p_i}X \mod 2 \in \pi_1(SO(4)) = \mathbb{Z}/2\mathbb{Z}. 
    \]
    By the Poincaré--Hopf theorem (the version with boundary), we have 
    \[
    \sum_{i>0}ind_{p_i}(X) = \chi(\Sigma) = 2 - 2g - 1 = 1-2g. 
    \]
    Therefore, 
    \[
    [Q] = [F|_{C_0}] = \bar{1} \in \pi_1(SO(4)) = \mathbb{Z}/2\mathbb{Z},
    \]
    and $\mu_*^f([Q]) = [\Rot(f)] \in \pi_1(\Emb(S^1, S^3), f)$. 
    We are done. 
\end{proof}

Combining the two quotient descriptions with the identity
\(
[\Repar]=[\Roll_0][\Rot],
\)
we obtain the following criterion.
\begin{corollary}[Descending criterion]
    Let $R$ be a commutative ring with unit and $\xi \in H^{1}(\LongKnotSpace; R)$ be a 1-cohomology class. 
    Then we have 
    \begin{enumerate}
        \item $\xi$ descends to a 1-cohomology class in $H^{1}(\Emb(S^1, S^3); R)$ if and only if $2\xi([\Rot(f)]) = 0$ for all $f\in \LongKnotSpace$.
        \item $\xi$ descends to a 1-cohomology class in $H^{1}(\Emb(S^1, S^3)/\Diff^+(S^1); R)$ if and only if $2\xi([\Rot(f)]) = 0$ and $\xi([\Roll_0(f)]) + \xi([\Rot(f)]) = 0$ for all $f\in \LongKnotSpace$.
    \end{enumerate}
\end{corollary}

\begin{remark}
    The descending criterion also works componentwise. 
\end{remark}


\begin{corollary}
    $\beta_1$ descends to a nontrivial 1-cohomology class in $H^{1}(\Emb(S^1, S^3); \mathbb{Z})$. 
    $\beta_1 \bmod2$ and $\beta_3$ descend to independent nontrivial 1-cohomology classes in $H^{1}(\Emb(S^1, S^3)/\Diff^+(S^1); \mathbb{Z}/2\mathbb{Z})$. 
\end{corollary}

\appendix
\section{Gauss diagram formulae of low orders}
\label{sec:GDF}

The Gauss diagram formulae used in this paper are listed below.  
$v_2$ uses the formula from \cite{PolyakViro94}. 
$v_3$ uses the formula from \cite{Zhang23}. 
$v_{4,1}$ and $v_{4,2}$ listed here are new expressions. 

\newcommand{\gdfterm}[1]{%
  \raisebox{-.42\height}{%
    \includegraphics[height=1.5cm]{figures/GDF/#1.pdf}}}

\begin{equation*}
  v_2=\gdfterm{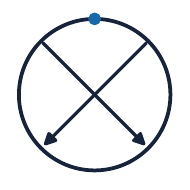}.
  \label{eq:gdf-v2}
\end{equation*}
 
\begin{equation*}
  \begin{aligned}
  v_3={}&
  \gdfterm{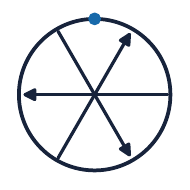}
  +\gdfterm{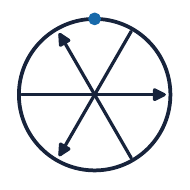}
  +\gdfterm{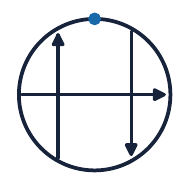}
  +\gdfterm{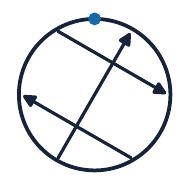}
  +\gdfterm{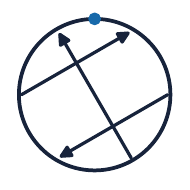}.
  \end{aligned}
  \label{eq:gdf-v3}
\end{equation*}

\begin{equation*}
  \begin{aligned}
  v_{4,1}={}&
  \gdfterm{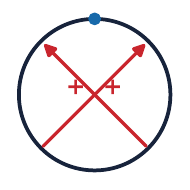}
  +\gdfterm{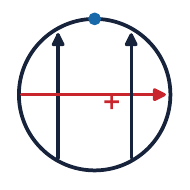}
  +\gdfterm{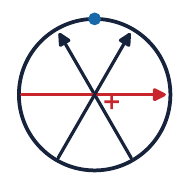}
  +\gdfterm{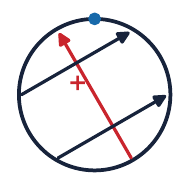}
  +\gdfterm{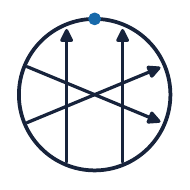}
  +\gdfterm{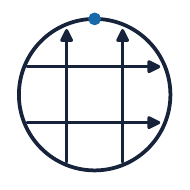}
  +\gdfterm{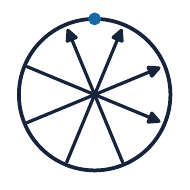}
  +\gdfterm{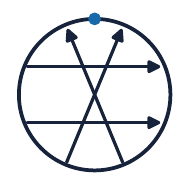}
  \\[.4em] 
  &+\gdfterm{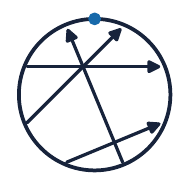}
  +\gdfterm{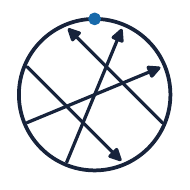}
  -\gdfterm{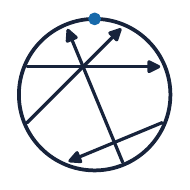}
  +\gdfterm{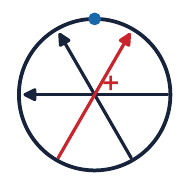}
  +\gdfterm{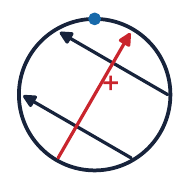}
  +\gdfterm{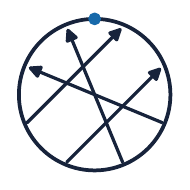}
  +\gdfterm{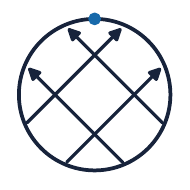}.
  \end{aligned}
  \label{eq:gdf-J}
\end{equation*}

\begin{equation*}
  \begin{aligned}
  v_{4,2}={}&
  2\,\gdfterm{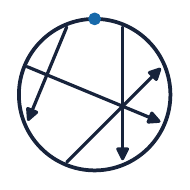}
  +2\,\gdfterm{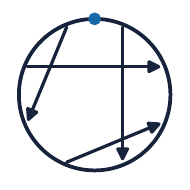}
  +\gdfterm{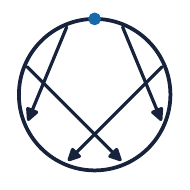}
  -\gdfterm{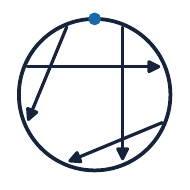}
  +2\,\gdfterm{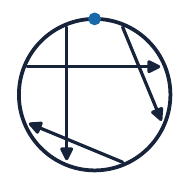}
  +2\,\gdfterm{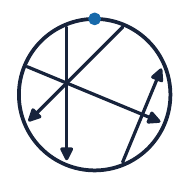}
  -\gdfterm{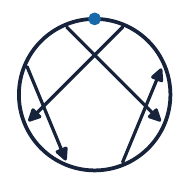}
  -\gdfterm{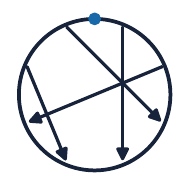}
  \\[.4em]
  &+3\,\gdfterm{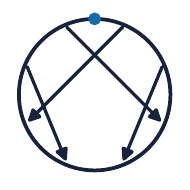}
  +\gdfterm{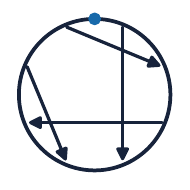}
  -\gdfterm{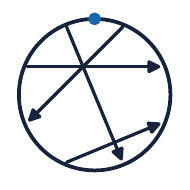}
  +\gdfterm{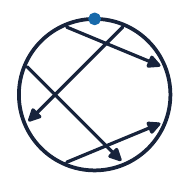}
  +\gdfterm{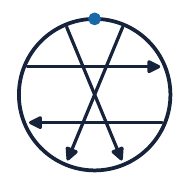}
  +3\,\gdfterm{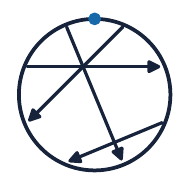}
  +\gdfterm{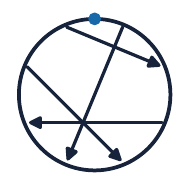}
  +\gdfterm{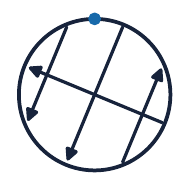}
  \\[.4em]
  &+2\,\gdfterm{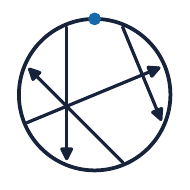}
  -\gdfterm{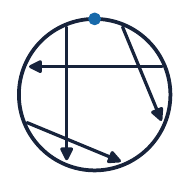}
  +\gdfterm{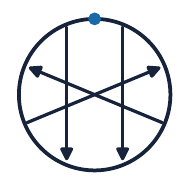}
  +\gdfterm{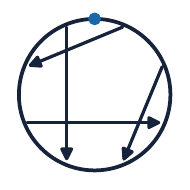}
  +\gdfterm{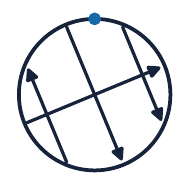}
  +3\,\gdfterm{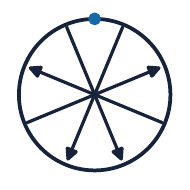}
  +\gdfterm{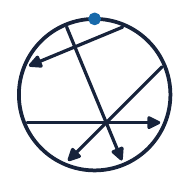}
  +\gdfterm{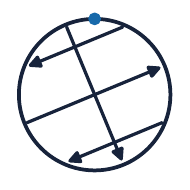}
  \\[.4em]
  &+\gdfterm{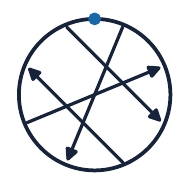}
  +\gdfterm{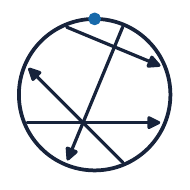}
  +3\,\gdfterm{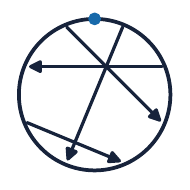}
  +\gdfterm{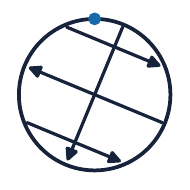}
  +3\,\gdfterm{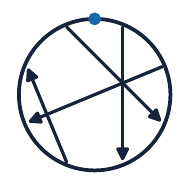}
  -\gdfterm{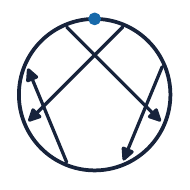}
  \\[.4em]
  &+\gdfterm{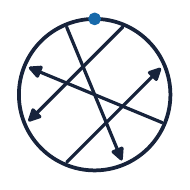}
  +\gdfterm{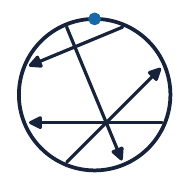}
  -\gdfterm{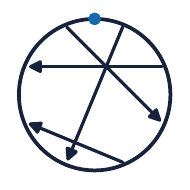}
  +\gdfterm{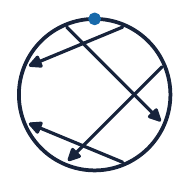}.
  \end{aligned}
  \label{eq:gdf-E}
\end{equation*}

\section{Cocyclicity}
\label{sec:Cocyclicity}
We explain the outline proof of cocyclicity of the 1-cocycles $\beta_1, \beta_2$ and $\beta_3$. 
We refer the reader to \cite[Chapter 6]{Zhang25} and \cite{ZBT99_25_10} for the complete proof and all the details. 
The proof of cocyclicity consists of verifying the vanishing of the combinatorial 1-cocycles on the meridians of the six types of codimension-two singularities of plane curves (see \Cref{fig:codimension_two_meridians}). 
\begin{figure}[H]
    \centering

    \begin{subfigure}[t]{0.48\textwidth}
        \centering
        \includegraphics[width=\linewidth]{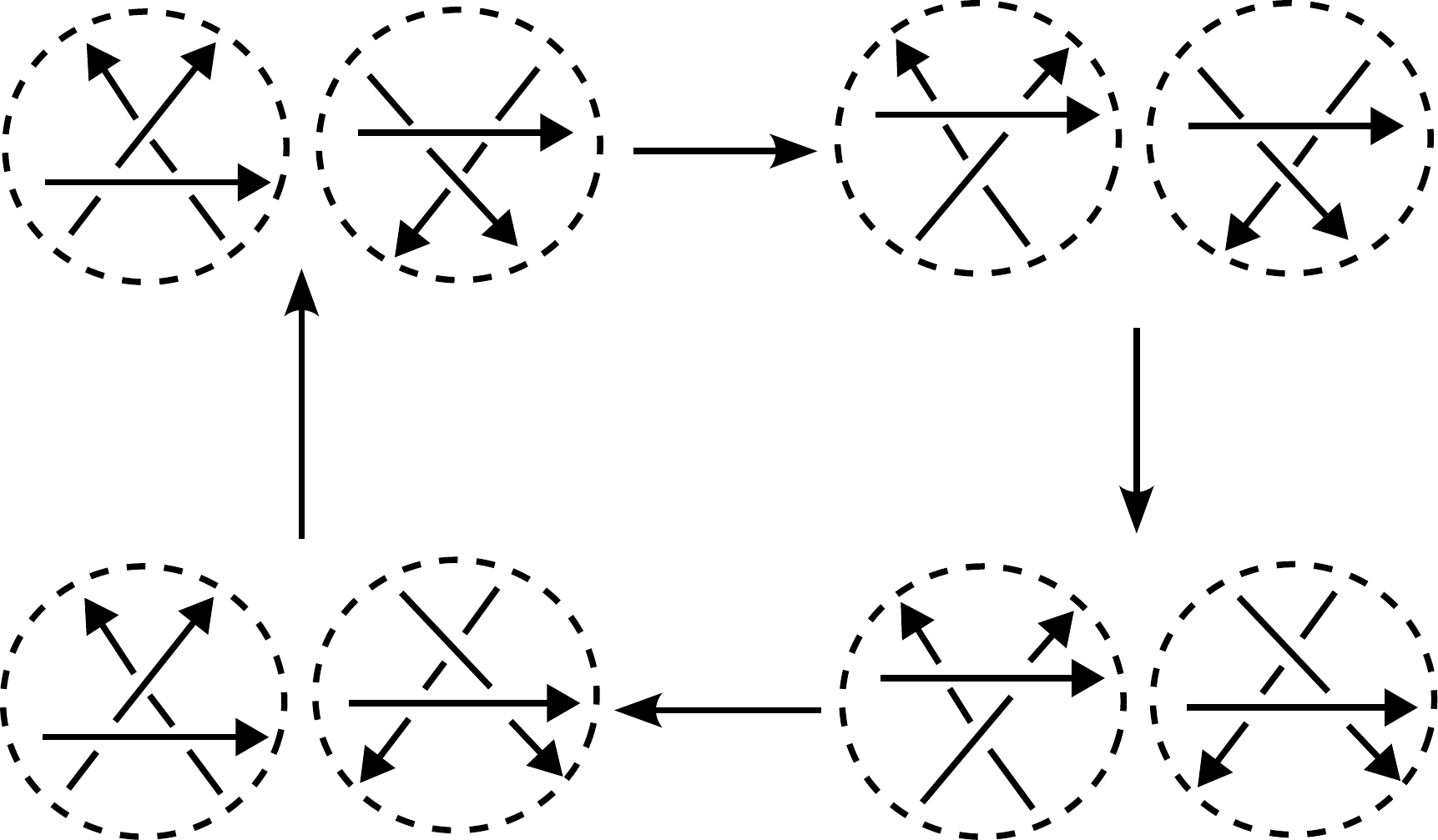}
        \caption{The meridian around the singularity of two triple points}
        \label{fig:meridian_commutation}
    \end{subfigure}
    \hfill
    \begin{subfigure}[t]{0.48\textwidth}
        \centering
        \includegraphics[width=\linewidth]{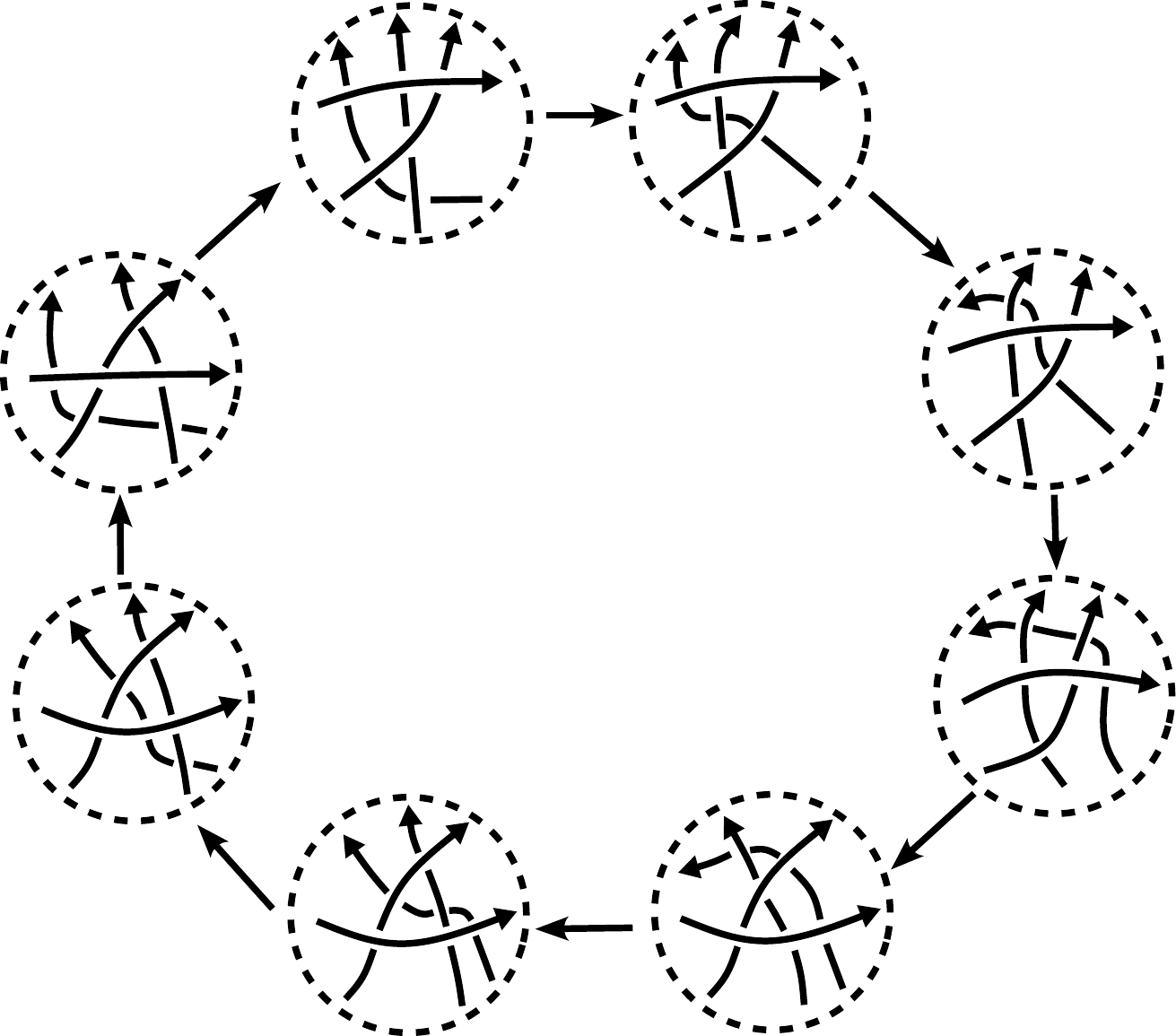}
        \caption{The meridian around a quadruple point}
        \label{fig:meridian_tetra}
    \end{subfigure}

    \vspace{0.5cm}

    \begin{subfigure}[t]{0.48\textwidth}
        \centering
        \includegraphics[width=\linewidth]{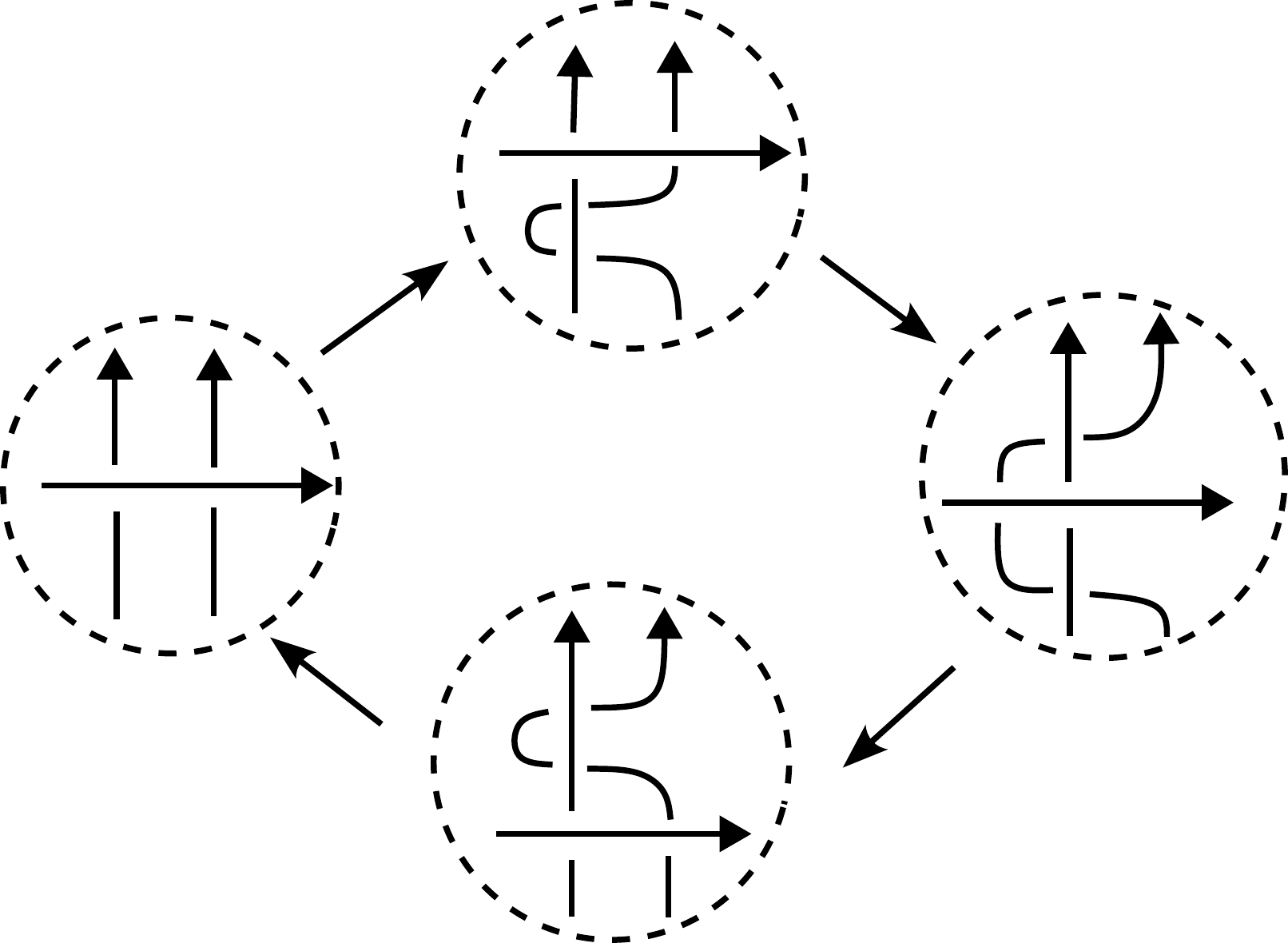}
        \caption{The meridian around a tangent triple point}
        \label{fig:meridian_cube}
    \end{subfigure}
    \hfill
    \begin{subfigure}[t]{0.48\textwidth}
        \centering
        \includegraphics[width=\linewidth]{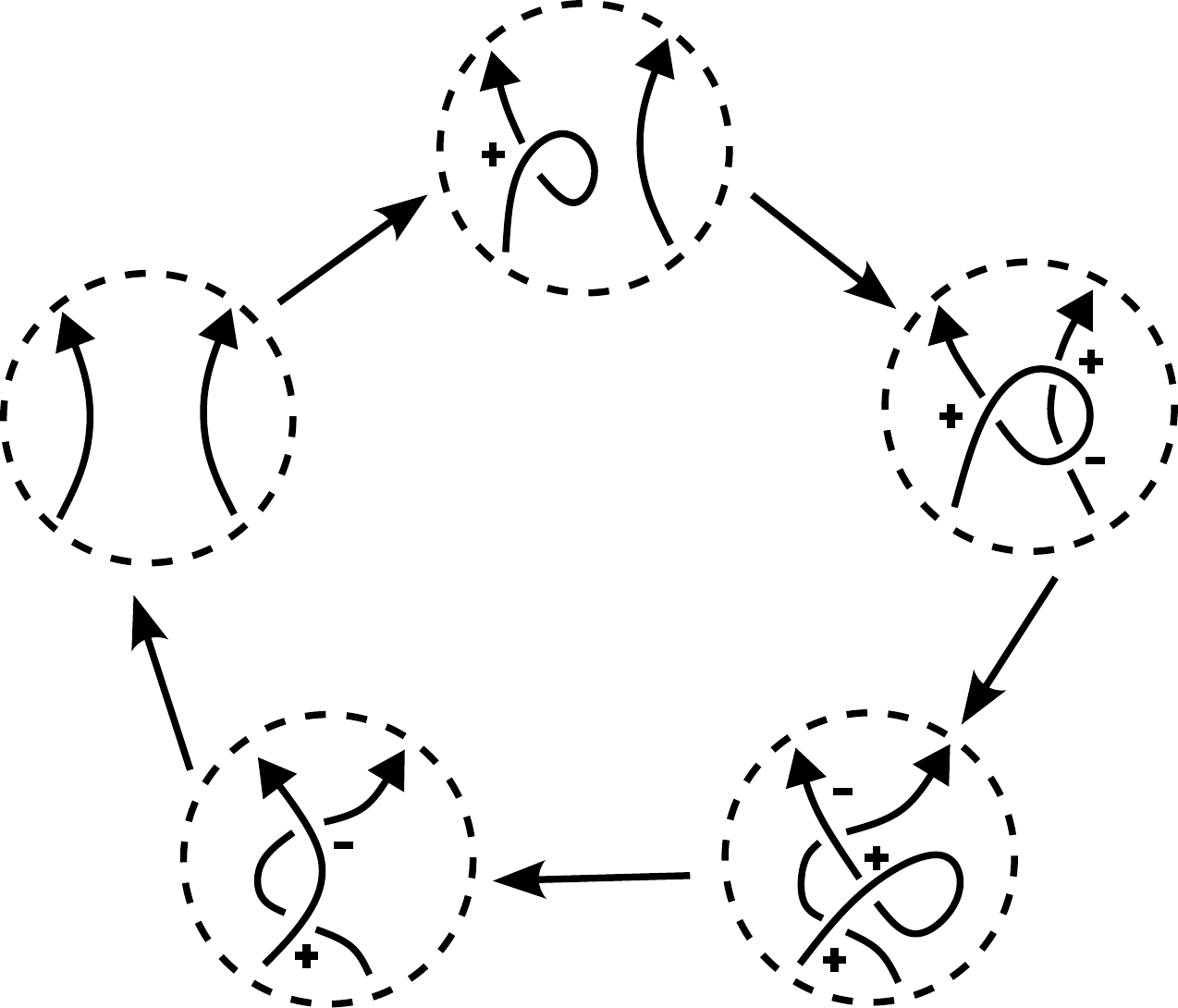}
        \caption{The meridian around an intersected cusp}
        \label{fig:meridian_cusp_0}
    \end{subfigure}

    \vspace{0.5cm}

    \begin{subfigure}[t]{0.48\textwidth}
        \centering
        \includegraphics[width=\linewidth]{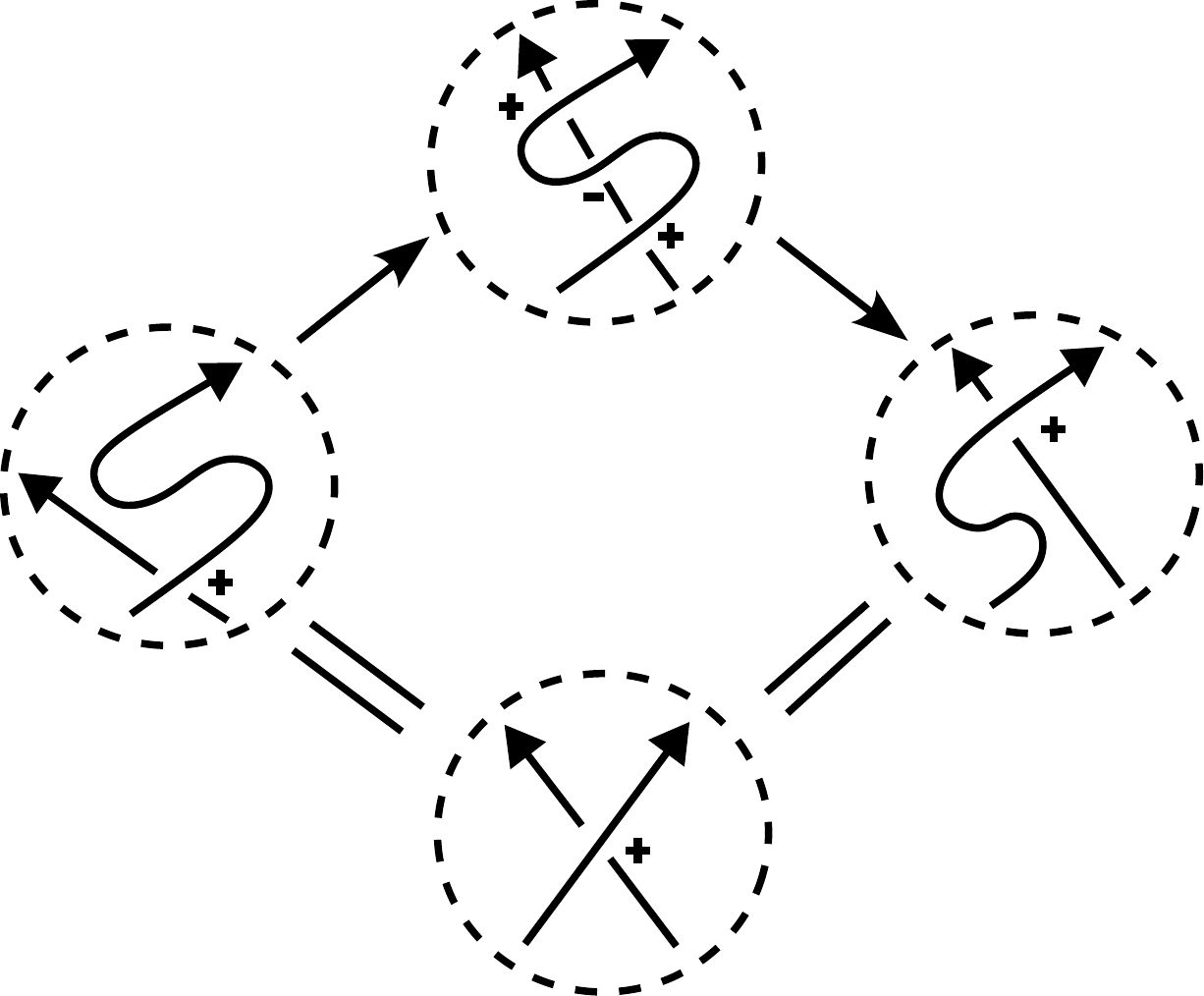}
        \caption{The meridian around a cubic tangency}
        \label{fig:meridian_cubic_tangency}
    \end{subfigure}
    \hfill
    \begin{subfigure}[t]{0.48\textwidth}
        \centering
        \includegraphics[width=0.8\linewidth]{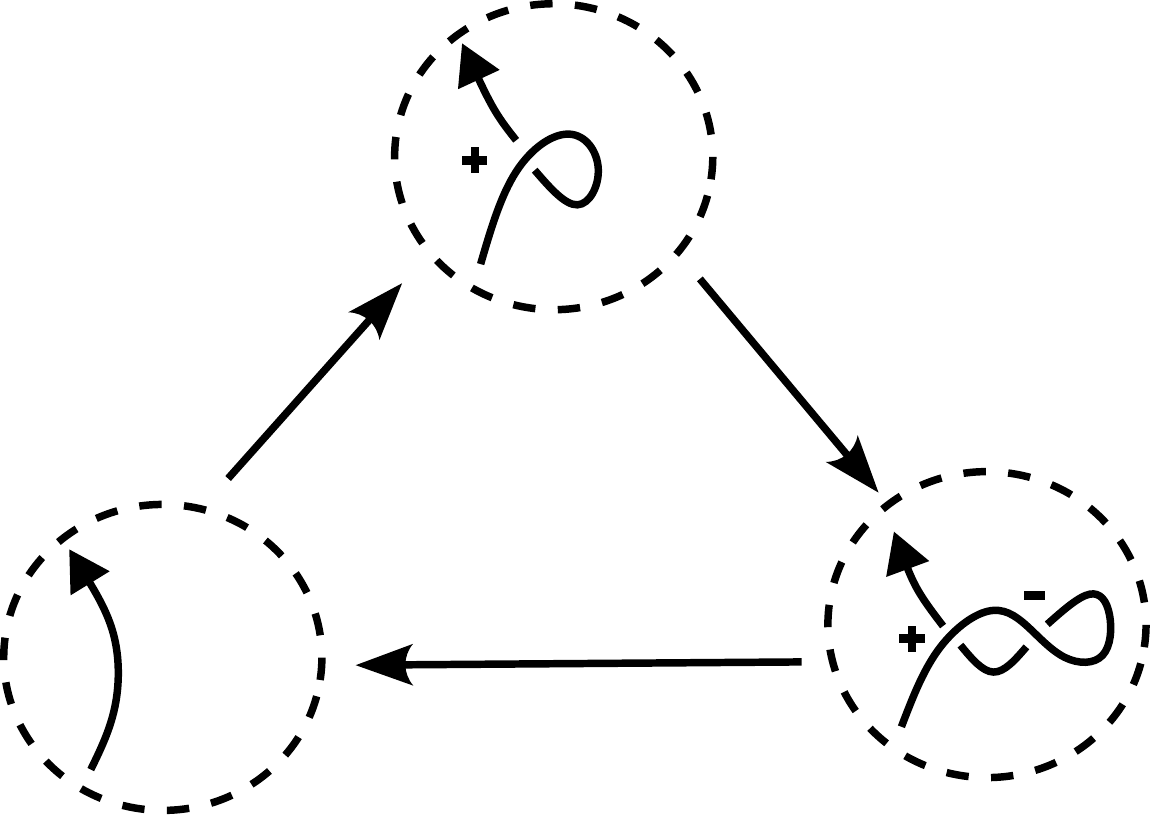}
        \caption{The meridian around a ramphoidal cusp}
        \label{fig:meridian_ramphoidal_cusp}
    \end{subfigure}

    \caption{Meridians of six types of codimension-two singularities}
    \label{fig:codimension_two_meridians}
\end{figure}
Since we only use R3 moves in the definitions of $\beta_1, \beta_2$ and $\beta_3$, we only need to prove that they vanish on the meridians around 
\begin{enumerate}
    \item two singularities of codimension 1; (This is called the commutation equation.)
    \item ordinary quadruple points; (This is called the tetrahedron equation.)
    \item tangent triple points; (This is called cube equation.)
    \item intersected cusps. (This is called the wandering cusp equation.)
\end{enumerate}

\subsection{Commutation equations}
Let $D_1$, $D_2$ be two subdiagrams in two disjoint disks of a long knot diagram, denoted by $(D_1, D_2)$, to perform R3 and R1 moves respectively. 
The meridian around the singularity (ordinary triple point, ordinary cusp) is the following (analogous to \Cref{fig:meridian_commutation}): 
\begin{equation}\label{diag:commutation}
\begin{tikzcd}[column sep=3.5cm, row sep=1.5cm]
    (D_1, D_2)
    \arrow[r, "R3", "p"']
    &
    (D_1', D_2)
    \arrow[d, "R1/R2/R3"']
    \\
    (D_1, D_2')
    \arrow[u, "R1/R3/R3"']
    &
    (D_1', D_2')
    \arrow[l, "R3"', "{p'}"]
\end{tikzcd}
\end{equation}

\begin{lemma}
    For R1-R3, and R2-R3 commutation cases in \Cref{diag:commutation}, we have $\beta_{i}(p) = \beta_{i}(p'), \; i = 1, 2, 3$. 
    For R3 commutation cases, we have $\beta_{i}(p) = \beta_{i}(p'), \; i = 1, 3$. 
\end{lemma}
\begin{lemma}
    $\beta_2$ vanishes on the whole meridian \Cref{diag:commutation}. 
\end{lemma}

\subsection{Tetrahedron equations}
There are 48 local types of quadruple points. 
We call the quadruple point whose meridian contains only positive crossings, the \textbf{positive quadruple point}. 
\begin{lemma}[{\cite[Lemma 12]{Fiedler07}}]
    All types of meridians around the quadruple points can be generated by 
    the meridians around the positive ordinary quadruple points together 
    with the meridians around 2 singularities of codimension 1, 
    tangent triple points and cubic tangencies. 
\end{lemma}
Therefore, we only need to prove for the positive quadruple point. 

\begin{lemma}
    $\beta_i, \; (i = 1, 2, 3)$ vanishes on the meridian around the positive quadruple point. 
\end{lemma}
\begin{proof}
    The verification is performed by a computer program \cite{ZBT99_25_10} case by case. 
    We generate the Gauss diagrams arising in all the cases of the tetrahedron equations (4 arrows from the quadruple point) with 1 outside arrow that is not from the quadruple point. 
    We intended to verify that $\beta_i$ vanishes on all these loops. 
    However, this is not true for some cases. 
    We should notice that not all the cases generated here can be realised by a real knot. 
    Some Gauss diagram identities (which hold only for real knots) must also be satisfied. 
    The failed cases vanish due to these identities (see \cite[Lemma 6.2.2]{Zhang25} for details). 
    In the program, these cases have been treated using the Gauss diagram identity obtained by counting linking numbers in two ways. 
\end{proof}

\subsection{Cube equations}
Locally, there are 24 types of tangent triple points. 
For the meridians around each type of tangent triple point, 
we use one diagram to represent it (see \Cref{fig:cube_meridians}). 

\begin{figure}[H]
    \center 
    \includegraphics[width = 10cm]{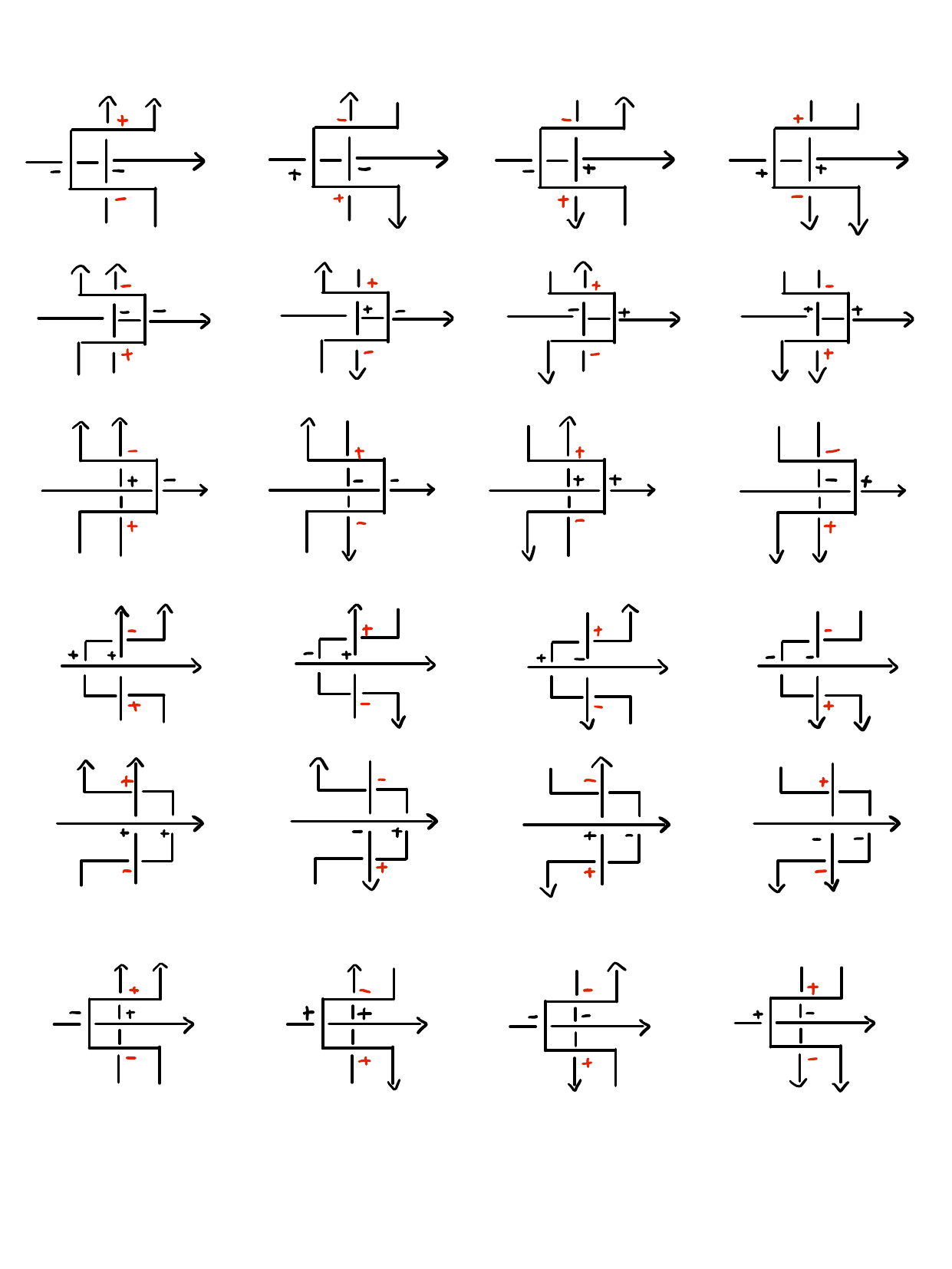}
    \caption{The meridians around all types of tangent triple points}
    \label{fig:cube_meridians}
\end{figure}


\begin{lemma}
    $\beta_i, \; i = 1, 2, 3$ vanishes on the meridian around the tangent triple points. 
\end{lemma}
In the verification for $\beta_1$, a Gauss-diagram identity is needed. 
    
    

\subsection{Wandering cusp equations}
We use just one diagram in the meridian around an ordinary cusp to represent the whole meridian.
There are 16 types of cusps. Their diagrams and Gauss diagrams are shown in \Cref{fig:Cusp_meridians}.  
\begin{figure}[htpb]
    \center 
    \includegraphics[width = 11.5cm]{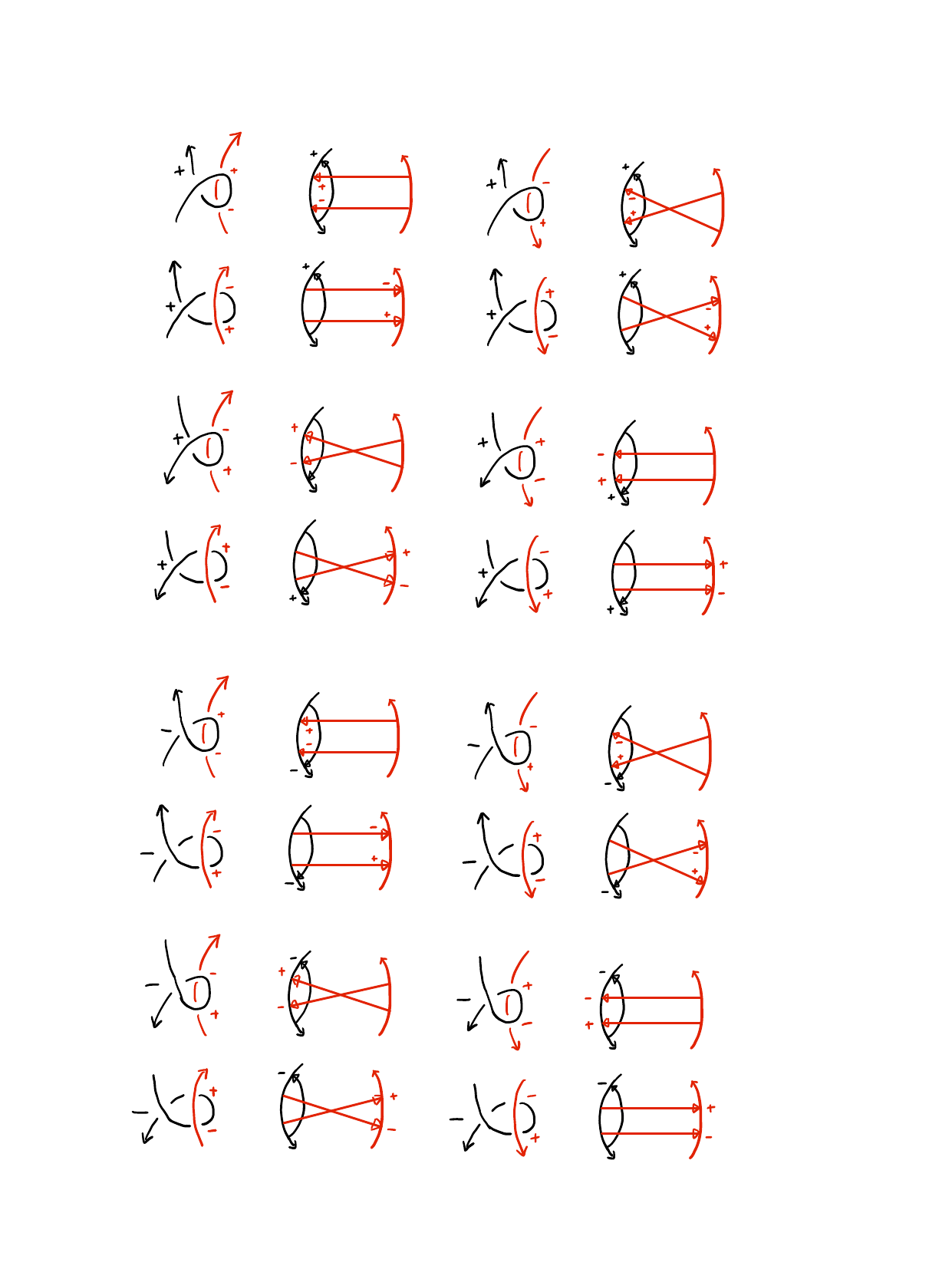}
    \caption{The meridians around intersected cusps}
    \label{fig:Cusp_meridians}
\end{figure}
\begin{lemma}
    $\beta_i, \; i = 1, 2, 3$ vanishes on the meridian around the intersected cusps. 
\end{lemma}

\subsection{Combinatorial order 4}
\begin{proposition}
    $\beta_i, \; (i = 1, 2, 3)$ has combinatorial order exactly 4. 
\end{proposition}
\begin{proof}
    The exact order of $\beta_1$ and $\beta_2$ is a direct corollary of \Cref{mainthm:pairing,thm:calculation_FH}. 
    For $\beta_3$, we have solved all the combinatorial 1-cocycles defined by Gauss diagrams with a triangle of order at most 3 over $\mathbb{Z}/2\mathbb{Z}$. 
    No combinatorial 1-cocycle of order at most 3 behaves the same as $\beta_3$ on the half rolling loops \cite{ZBT99_25_10}. 
\end{proof}

\bibliographystyle{alpha}
\bibliography{bibliography/references}

\end{document}